\documentclass[10pt]{amsart}
\usepackage{latexsym}
\usepackage{amsmath}
\usepackage{amssymb}
\usepackage{mathrsfs}
\usepackage{graphicx}
\usepackage{color}
\usepackage{pgfpages}
\usepackage{ifthen}
\usepackage{leftidx,tensor}
\usepackage[T1]{fontenc}
\usepackage[utf8]{inputenc}
\usepackage{mathtools}
\usepackage{comment}
\usepackage{dsfont}
\usepackage[nocompress]{cite}
\usepackage{newunicodechar}
\newunicodechar{，}{,}

\usepackage[shortlabels]{enumitem}
\usepackage{aliascnt}
\usepackage[bookmarks=true,pdfstartview=FitH, pdfborder={0 0 0}, colorlinks=true,citecolor=red, linkcolor=blue]{hyperref}
\usepackage{bbm}
\usepackage{nicefrac}

\theoremstyle{plain}
\newtheorem{thm}{Theorem}[section]
\newaliascnt{cor}{thm}
\newaliascnt{prop}{thm}
\newaliascnt{lem}{thm}

\newtheorem{cor}[cor]{Corollary}

\newtheorem{prop}[prop]{Proposition}
\newtheorem{lem}[lem]{Lemma}
\aliascntresetthe{cor}
\aliascntresetthe{prop}
\aliascntresetthe{lem} 
\theoremstyle{definition}
\newaliascnt{defn}{thm}
\newaliascnt{asu}{thm}
\newaliascnt{con}{thm}

\newtheorem{asu}[asu]{Assumption}

\aliascntresetthe{defn}
\aliascntresetthe{asu}
\aliascntresetthe{con}
\newcounter{stp}
\newcounter{stpi}
\newcounter{stpci}
\newcounter{stpiii}

\newtheorem{step}[stp]{Step}

\theoremstyle{thm}
\newaliascnt{rem}{thm}
\newaliascnt{exa}{thm}
\newaliascnt{masu}{thm}
\newaliascnt{nota}{thm}
\newaliascnt{sett}{thm}
\aliascntresetthe{rem}
\aliascntresetthe{exa}
\aliascntresetthe{masu}
\aliascntresetthe{nota}
\aliascntresetthe{sett}
\numberwithin{equation}{section}

\setlist[enumerate]{font = \normalfont}

\newcommand{\eps}{\varepsilon}
\renewcommand{\bar}[1]{\overline{#1}}

\newcommand{\R}{\mathbb{R}}

\newcommand{\E}{\mathbb{E}}

\newcommand{\rC}{\mathrm{C}}
\newcommand{\rL}{\mathrm{L}}
\newcommand{\rW}{\mathrm{W}}
\newcommand{\rH}{\mathrm{H}}

\newcommand{\rD}{\mathrm{D}}
\newcommand{\rN}{\mathrm{N}}

\newcommand{\rX}{\mathrm{X}}

\newcommand{\rLq}{\rL^q}
\newcommand{\rLp}{\rL^p}

\newcommand{\rd}{\mathrm{d}}

\renewcommand{\rm}{\mathrm{m}}
\newcommand{\rr}{\mathrm{r}}
\newcommand{\rv}{\mathrm{v}}
\newcommand{\rvs}{\mathrm{vs}}
\newcommand{\rc}{\mathrm{c}}
\newcommand{\rev}{\mathrm{ev}}
\newcommand{\rcd}{\mathrm{cd}}
\newcommand{\rac}{\mathrm{ac}}
\newcommand{\rcr}{\mathrm{cr}}
\newcommand{\rth}{\mathrm{th}}
\newcommand{\rcp}{\mathrm{cp}}
\newcommand{\rref}{\mathrm{ref}}
\newcommand{\rpd}{\mathrm{pd}}
\newcommand{\rpv}{\mathrm{pv}}
\newcommand{\rcn}{\mathrm{cn}}

\newcommand{\cL}{\mathcal{L}}

\DeclareMathOperator{\Cof}{Cof}
\newcommand{\Hinfty}{\mathcal{H}^\infty}

\renewcommand{\div}{\mathrm{div} \,}
\newcommand{\divH}{\mathrm{div}_{\rH} \,}
\newcommand{\nablaH}{\nabla_{\rH}}

\newcommand{\del}{\partial}
\newcommand{\dk}[1]{\partial_{#1}}
\newcommand{\dt}{\dk{t}} 
\newcommand{\dz}{\dk{z}} 

\newcommand{\tfor}{\enspace \text{for} \enspace}

\newcommand{\tand}{\enspace \text{and} \enspace}

\newcommand{\T}{\mathbb{T}}

\begin{document}

\title[Moisture dynamics in the hydrostatic Atmosphere]{Moisture Dynamics with Phase Changes in a Compressible Hydrostatic Atmosphere}

\author{Lin Ma}
\address{State Key Laboratory of Mathematical Sciences, Academy of Mathematics and Systems Science,
Chinese Academy of Sciences, Beijing 100190, China}
\email{malin.cnu@foxmail.com}

\author{Tarek Z\"{o}chling}
\address{Technische Universit\"{a}t Darmstadt,
Schlo\ss{}gartenstra{\ss}e 7, 64289 Darmstadt, Germany.}
\email{zoechling@mathematik.tu-darmstadt.de}

\maketitle

\begin{abstract}
In this article, a rigorous well-posedness result is established for a moist compressible primitive-equation system with phase changes.
The model couples the compressible primitive equations to a bulk microphysics model for vapor water, cloud water, and rain water, including evaporation, condensation, auto-conversion, collection, and sedimentation effects. 
Global strong well-posedness is proved for initial data sufficiently close to constant equilibria. 
\end{abstract}

\section{Introduction}\label{sec:intro}
\noindent
Atmospheric models that incorporate moisture dynamics and bulk microphysics have a long tradition in fluid dynamics and atmospheric science. In the hydrostatic regime, the pressure gradient is balanced by gravity in the vertical direction, that is
\[
    \partial_z p = -\rho g.
\]
Here $p$ denotes
the pressure, $\rho$ the density, and $g$ the gravitational constant. 
This makes the compressible setting particularly delicate, since pressure, density, temperature, and moisture are coupled both through the equation of state and through the hydrostatic constraint. Therefore, rigorous analytical results for moist, hydrostatic, compressible models are so far only available either in pressure coordinates, where the pressure variable is posed on a fixed interval and an evolving surface pressure is not treated directly, or for moisture systems with prescribed velocity fields; see, for instance, \cite{CZT:12,CZFTT:13,BCZT:14,CZHKTZ:15,HKLT:17,HKLT:20,HKLT:23}. The purpose of this paper is to establish a well-posedness result for a moist, hydrostatic, compressible atmosphere with phase changes, without passing to pressure coordinates and without prescribing the velocity field.

\smallskip

\noindent
Let us briefly explain the structural difficulty caused by combining hydrostatic
balance with compressibility in the presence of moisture. 
Denoting by $\rho_{\rd}$, $\rho_{\rv}$, $\rho_{\rc}$ and $\rho_{\rr}$ the densities of dry air, water vapor, cloud water and rain water, we introduce the corresponding mixing ratios $q_j$ by
\(
q_j \coloneqq \nicefrac{\rho_j}{\rho_{\rd}} \text{ for } j\in\{\rv,\rc,\rr\}.
\)
Thus, the total moist-air density is given by
\(
    \rho
    =
    \rho_{\rd}(1+q_{\rv}+q_{\rc}+q_{\rr}).
\)
With physical constants normalized, the pressure is determined by the moist ideal gas law
\(
    p
    =
    \rho_{\rd}(1+q_{\rv})T ,
\)
which implies that the hydrostatic balance takes the form
\[
    \del_z p
    =
    -\frac{p(1+q_{\rv}+q_{\rc}+q_{\rr})}{(1+q_\rv) T} .
\]
Here only the gaseous water component, namely water vapor, contributes to the pressure through the equation of state, whereas cloud water and rain water contribute to the total moist-air density and therefore enter the hydrostatic balance; cf. \cite{HK:18,DKLT:24,BHMZ:26a,BHMZ:26b}. This special structure shows that the pressure evolution is influenced not only
by the temperature, as in the dry non-isothermal case treated in \cite{TZ:26},
but also by all mixing ratios \(q_{\rv},q_{\rc}\) and $q_{\rr}$. Indeed, solving with respect to the vertical variable, yields
\begin{equation*}
    p(t,x,y,z)
    =
    p_s(t,x,y)\,
    \exp\Bigl(
        -\int_0^z
        \frac{(1+q_{\rv}+q_{\rc}+q_{\rr})}{(1+q_{\rv})T}(t,x,y,\eta)\,\rd\eta
    \Bigr),
\end{equation*}
where $p_s$ denotes the pressure evaluated at $z=0$. Inserting this expression for the pressure into the moist ideal gas law, one obtains a representation of the form
\[
    \rho_{\rd}
    =
    \bar\rho_{\rd}\,
    \hat B(T,q_{\rv},q_{\rc},q_{\rr}),
\]
where \(\bar\rho_{\rd}\) is the vertical average of the dry-air density and
where \(\hat B\) is a \emph{nonlinear, nonlocal} functional of the temperature and
moisture variables, given by
\begin{equation*}
    \hat B(T,q_{\rv},q_{\rc},q_{\rr})
    :=
    \frac{
        \displaystyle
        \frac{1}{T(1+q_{\rv})}
        \exp
        \big (
            -\int_0^z
            \frac{(1+q_{\rv}+q_{\rc}+q_{\rr})}{T(1+q_{\rv})}
            \,\rd\eta
        \big )
    }{
        \displaystyle
        \int_0^1
        \frac{1}{T(1+q_{\rv})}
        \exp
        \big (
            -\int_0^\zeta
            \frac{(1+q_{\rv}+q_{\rc}+q_{\rr})}{T(1+q_{\rv})}
            \,\rd\eta
        \big )
        \,\rd\zeta
    }.
\end{equation*} 
This representation is the first key structural ingredient. It allows us to replace the full dry-air density equation by an equation for the vertically averaged dry-air density and, at the same time, yields an explicit formula for the vertical velocity $w$. In the hydrostatic regime, this velocity does not satisfy an evolution equation of its own; rather, it is determined diagnostically through the dry-air continuity equation. This diagnostic character of $w$ is one of the classical features of hydrostatic compressible dynamics and is a major source of analytical difficulty. In the moist setting, this difficulty becomes even more pronounced than in the dry case, since the diagnostic determination of $w$ is coupled not only to the temperature as in \cite{TZ:26}, but also to the moisture variables entering the total density and the hydrostatic balance; see also \autoref{sec:coupled model}.

\smallskip

\noindent
Next, let us also comment on the structure of the moisture dynamics.
The moisture evolution is described by transport--diffusion equations for the mixing ratios. In schematic form,
\[
    D_t q_j-\Delta q_j=S_j \ \text{ for }
    j\in\{\rv,\rc,\rr\},
\]
where
\(
    D_t=\partial_t+v\cdot\nablaH+w \, \partial_z
\)
denotes the material derivative,
with $u=(v,w)$ being the air velocity, separated into its horizontal component $v$ and its vertical component $w$. The source terms \(S_j\) account for phase transitions between the moisture species. The dominant mechanisms are condensation and evaporation (vapor--cloud exchange), autoconversion and accretion (cloud-to-rain conversion), and sedimentation (vertical fall of rain relative to the air).
A key analytical difficulty is that these phase transitions are activated only in certain thermodynamic regimes. As a result, the sources involve positive-part ``switch terms,'' such as 
\[
    (q_{\rv}-q_{\rvs})^+,
    \quad
    (q_{\rvs}-q_{\rv})^+ \ \text{ and } \ 
    (q_{\rc}-1)^+,
\]
where \(q_{\rvs}\) denotes the saturation mixing ratio. Here $f^+$ denotes the positive part a function, that is $f^+ =\max \{ f,0\} $. The switch terms have two main effects. On the one hand, they strongly couple the moisture variables to the
thermodynamic equation, in particular through the different heat capacities of
dry air, water vapor, and liquid water. On the other hand, they are only
Lipschitz continuous and therefore not Fr\'echet differentiable. This lack of
differentiability is reflected directly in the choice of the solution spaces.
Indeed, if the switch terms are treated as forcing terms, then they cannot be
expected to have two spatial derivatives in \(\rL^2\). For instance, for a
positive-part contribution one only has, in general,
\[
    \nabla(q^+)
    =
    \mathbf 1_{\{q>0\}}\nabla q
    \quad\text{a.e.},
\]
whereas second derivatives would produce derivatives of the characteristic
function of the active set. Thus the switch terms do not naturally belong to
\(\rH^2(\Omega)\) in space.
Consequently, the right-hand sides have to be estimated in spaces with only
\(\rH^1\)-regularity in space. This lowering of the spatial regularity creates
an additional difficulty in the nonlinear estimates. In three space dimensions,
\(\rH^1(\Omega)\) is not an algebra with respect to pointwise multiplication.
Therefore, products cannot be controlled by a simple algebra estimate. Compared
to the dry non-isothermal case treated in \cite{TZ:26}, the nonlinear estimates
become substantially more delicate and require a careful use of product
estimates in mixed space-time norms; see also \autoref{lem: est nonlinear moist}.
\smallskip

\noindent
Let us now comment on previous work and categorize our result.
For general background on moisture dynamics in atmospheric models, we refer to the survey article of Stevens \cite{Ste:05}, the article by Pauluis, Czaja and Korty \cite{PCK:08}, and the monograph by Khouider \cite{Kho:19}, as well as the references therein. 
The microphysical model considered here is based on the description of warm clouds in terms of vapor water, cloud water, and rain water. 
From a modeling point of view, this line goes back in particular to the work of Kessler \cite{Kes:69} and to the cloud model of Grabowski and Smolarkiewicz \cite{GS:96}. 
Building on these ideas, Klein and Majda \cite{KM:06} developed a systematic multiscale model for moist deep convection on mesoscales, involving a bulk microphysics closure for the interaction of vapor, cloud water, and rain water. 
This model was later refined by Hittmeir and Klein \cite{HK:18}, where a more detailed thermodynamic setting was introduced. 
In particular, the phase changes between the different forms of water are coupled to the temperature equation through latent heat effects, leading to a substantially more strongly coupled system.

\smallskip

\noindent
The rigorous mathematical analysis of moisture models with phase changes has developed in several directions. 
For models involving a single moisture quantity and saturation effects, we refer to the works of Coti-Zelati and Temam \cite{CZT:12}, Coti-Zelati, Fr\'emond, Temam and Tribbia \cite{CZFTT:13}, Bousquet, Coti-Zelati and Temam \cite{BCZT:14}, and Coti-Zelati, Huang, Kukavica, Temam and Ziane \cite{CZHKTZ:15}. 
These works already address important analytical difficulties caused by saturation and condensation effects in atmospheric models. 
Another class of reduced moist atmospheric models was introduced by Frierson, Majda and Pauluis \cite{FMP:04} in the context of large-scale precipitation fronts in the tropical atmosphere. 
The corresponding mathematical analysis was carried out by Majda and Souganidis \cite{MS:10} in a weak formulation, while Li and Titi \cite{LT:16} proved global well-posedness and a relaxation limit for a related tropical atmosphere model with moisture. 
Further reduced models involving clouds and phase changes have recently been studied by Zhang, Smith and Stechmann \cite{ZSS:21}, who numerically investigated the effect of phase changes on fast-wave averaging, and by Remond-Tiedrez, Smith and Stechmann \cite{RTSS:24a,RTSS:24b}, who analyzed models arising in precipitating quasi-geostrophic and moist Boussinesq dynamics.

\smallskip

\noindent
The rigorous analysis of the three-component moisture model of Klein and Majda \cite{KM:06}, and later of its thermodynamically refined version due to Hittmeir and Klein \cite{HK:18}, started only recently. 
In \cite{HKLT:17}, Hittmeir, Klein, Li and Titi proved global strong well-posedness for the nonlinear moisture dynamics with phase changes under the assumption that the velocity field is prescribed. 
The same authors subsequently coupled the moisture model to the viscous incompressible primitive equations in \cite{HKLT:20}. 
This result relies on the hydrostatic primitive-equation structure and builds on the fundamental global well-posedness theory for the primitive equations due to Cao and Titi \cite{CT:07}. 
The thermodynamically refined moisture model from \cite{HK:18} was then treated in \cite{HKLT:23}, again in the case of a prescribed velocity field. 
Thus, in the available rigorous results for the full three-component moisture dynamics with phase changes, either the velocity field is prescribed or the fluid equations are formulated within the hydrostatic primitive-equation framework in pressure coordinates. Concerning the dry compressible primitive equations we refer to \cite{HIRZ:25, LT:20, LT:21} in the isothermal regime as well as \cite{TZ:26} in the heat conducting situation.

\smallskip

\noindent
A different recent direction concerns the coupling of the microphysics model to compressible Navier--Stokes dynamics.
Doppler, Klein, Liu and Titi \cite{DKLT:24} proved local strong well-posedness for the refined moisture model coupled to compressible atmospheric dynamics.
This result was further developed in \cite{BHMZ:26a,BHMZ:26b}, where the corresponding compressible Navier--Stokes--moisture model with phase changes is shown to be globally strongly well-posed close to constant equilibria; the analysis is carried out in an $\rL^p$-$\rL^q$ framework in \cite{BHMZ:26a} and in a Hilbert-space setting in \cite{BHMZ:26b}.

\medskip

\noindent
The article is structured as follows. In \autoref{sec:coupled model}, we collect the
preliminaries and derive the structural
representation of the dry-air density. In \autoref{sec:main results}, we state the
main well-posedness result. In \autoref{sec: Lagrange moist}, we introduce the
moisture-dependent Lagrangian transformation and prove the corresponding
transformation estimates. Using this transformation we state the fully transformed system in \autoref{sec Lagrange model}. In \autoref{sec: local}, we establish the linear
maximal-regularity estimates followed by the nonlinear estimates, and 
auxiliary switch estimates for the moisture equations in \autoref{subsec: nonlinear estimates}. Finally, in
\autoref{sec: proof}, we combine these ingredients and prove global-wellposedness using a contraction mapping argument.

\section{Description of the coupled moisture-fluid model}\label{sec:coupled model}
\noindent
\noindent
The purpose of this section is to introduce the coupled model of the moisture
dynamics and the non-isothermal compressible primitive equations. We work in
physical height coordinates and consider the domain
\[
    \Omega=\T^2\times(0,1),
    \qquad
    \Omega_\tau:=(0,\tau)\times\Omega,
\]
where \(\tau\in(0,\infty]\). We write
\[
    \nablaH:=(\partial_x,\partial_y)^\top \ \text{ and } \
    \divH v:=\nablaH\cdot v
\]
The velocity field is written as
\[
    u=(v,w)\colon \Omega_\tau\to\R^3,
\]
where \(v=(v_1,v_2)\colon\Omega_\tau\to\R^2\) denotes the horizontal velocity
and \(w\colon\Omega_\tau\to\R\) the vertical velocity. The material derivative
is then denoted by
\[
    \rD_t
    :=
    \partial_t+u\cdot\nabla
    =
    \partial_t+v\cdot\nablaH+w\, \partial_z .
\]
Furthermore, denote by \(\rho_\rd\), \(\rho_\rv\), \(\rho_\rc\), and \(\rho_\rr\) the
densities of dry air, water vapor, cloud water, and rain water, respectively.
The corresponding mixing ratios are
\[
    q_\rv=\frac{\rho_\rv}{\rho_\rd},
    \qquad
    q_\rc=\frac{\rho_\rc}{\rho_\rd},
    \qquad
    q_\rr=\frac{\rho_\rr}{\rho_\rd}.
\]
Thus, the total moist-air density is
\[
    \rho
    :=
    \rho_\rd+\rho_\rv+\rho_\rc+\rho_\rr
    =
    \rho_\rd(1+q_\rv+q_\rc+q_\rr).
\]
The terminal fall velocity of rain water is denoted by \(V_\rr\). It is assumed
to be a given smooth function describing the downward sedimentation speed of
rain droplets.
For the potential temperature \(\theta\), the pressure \(p\), and positive
constants \(p_\rref\), \(c_\rpd\), and \(R_\rd\), with \(c_\rpd>R_\rd\), the
absolute temperature is given by
\[
    T
    :=
    \theta
    \Bigl(\frac{p}{p_\rref}\Bigr)^{\frac{\gamma-1}{\gamma}},
    \qquad
    \gamma
    :=
    \frac{c_\rpd}{c_\rpd-R_\rd}
    >1.
\]
The mixed heat capacity is
\[
    c_\nu
    :=
    c_\rpd+c_\rpv q_\rv+c_1(q_\rc+q_\rr),
\]
and we define
\[
    \sigma
    :=
    \Bigl(\frac{c_\rpv}{c_\rpd}R_\rd-R_\rv\Bigr)q_\rv
    +
    \frac{c_1}{c_\rpd}R_\rd(q_\rc+q_\rr).
\]
The latent heat of condensation is assumed to be affine in \(T\), namely
\[
    L(T)
    :=
    L_\rref+(c_\rpv-c_1)(T-T_\rref),
\]
where \(L_\rref\) and \(T_\rref\) are positive reference values.
With these quantities we set
\begin{equation*}
    \begin{aligned}
        Q_\rm
        &\coloneqq 1 + q_\rv + q_\rc + q_\rr,\\
        Q_\rth
        &\coloneqq \frac{c_\nu}{\gamma} + \sigma = \frac{c_\rpd}{\gamma} + \Bigl(\frac{c_\rpv}{\gamma} + \frac{c_\rpv}{c_\rpd} R_\rd - R_\rv\Bigr) q_\rv + \Bigl(\frac{c_1}{\gamma} + \frac{c_1}{c_\rpd}R_\rd\Bigr)(q_\rc + q_\rr),\\
        Q_\rcp
        &\coloneqq \sigma - \frac{R_\rd}{c_\rpd} c_\nu = - R_\rd - R_\rv q_\rv,\\
        Q_1
        &\coloneqq \frac{R_\rv}{R_\rd + R_\rv q_\rv}\Bigl(\sigma - \frac{R_\rd}{c_\rpd} c_\nu\Bigr) + c_\rpv - c_1 = c_\rpv - c_1 - R_\rv, \tand Q_2 \coloneqq L_\rref - (c_\rpv - c_1) T_\rref.
    \end{aligned}
\end{equation*}
The saturation mixing ratio is denoted by
\(
    q_\rvs=q_\rvs(p,T).
\)
As in \cite{HKLT:17}, we assume that there are temperatures
\(0\leq T_A\leq T_B\), given in Kelvin, such that
\[
    q_\rvs(p,T)=0
    \qquad
    \text{for } T\leq T_A \text{ and } T\geq T_B .
\]
Moreover, we
assume that \(q_\rvs\) is non-negative, uniformly bounded, and Lipschitz
continuous as a function of \(p\) and \(T\). For a function \(f\), we
write
\[
    f^+:=\max\{f,0\}.
\]
Finally, we define the horizontal Lam\'e operator by
\[
    \rL_\rH v
    :=
    \mu\Delta v
    +
    (\mu+\lambda)\nablaH\divH v\ \text{ with } \
    \mu>0 \ \text{ and } \ 
    2\mu+\lambda>0.
\]
With this notation, the hydrostatic moist compressible primitive-equation
system in \(\Omega_\tau\) reads
\begin{equation}
    \label{eq:coupled moisture compr PE comp intro}
    \left\{
    \begin{aligned}
        \rD_t\rho_{\rd}+\rho_{\rd}\div u
        &=0,
        \\
        \rho_\rd Q_\rm \rD_t v-\rL_\rH v+\nablaH p
        &=
        \rho_{\rd}q_{\rr}V_{\rr}\partial_z v,
        \\
        \partial_z p
        &=
        -\rho_{\rd}Q_{\rm}g,
        \\
        Q_{\rth}\rD_t T-\kappa \Delta T
        &=
        c_1 q_{\rr}V_{\rr}\partial_z T
        +
        Q_{\rcp}T\div u
        -
        (Q_1T+Q_2)(S_\rev-S_\rcd),
        \\
        \rD_t q_{\rv}-\Delta q_{\rv}
        &=
        S_\rev-S_\rcd,
        \\
        \rD_t q_{\rc}-\Delta q_{\rc}
        &=
        S_\rcd-(S_\rac+S_\rcr),
        \\
        \rD_t q_{\rr}-\Delta q_{\rr}
        &=
        \partial_z(q_\rr V_\rr)
        +
        q_\rr V_\rr \partial_z\log\rho_\rd
        +
        (S_\rac+S_\rcr)
        -
        S_\rev,
        \\
        p
        &=
        \rho_{\rd}(R_\rd+R_\rv q_{\rv})T .
    \end{aligned}
    \right.
\end{equation}
Here \(S_\rev\), \(S_\rcd\), \(S_\rac\), and \(S_\rcr\) denote the conversion
terms. More precisely, \(S_\rev\) describes evaporation of rain water,
\(S_\rcd\) condensation of water vapor into cloud water and the inverse process,
\(S_\rac\) auto-conversion of cloud water into rain water, and \(S_\rcr\)
collection of cloud water by falling rain. They are given by
\[
\begin{aligned}
    S_\rev
    &:=
    c_\rev
    \frac{p}{\rho}
    (q_\rvs-q_\rv)^+q_\rr
    =
    c_\rev
    T
    \frac{R_\rd+R_\rv q_\rv}
    {1+q_\rv+q_\rc+q_\rr}
    (q_\rvs-q_\rv)^+q_\rr,
    \\
    S_\rcd
    &:=
    c_\rcd(q_\rv-q_\rvs)q_\rc
    +
    c_\rcn(q_\rv-q_\rvs)^+q_\rcn,
    \\
    S_\rac
    &:=
    c_\rac(q_\rc-q_\rac)^+,
    \qquad
    S_\rcr
    :=
    c_\rcr q_\rc q_\rr .
\end{aligned}
\]
Here \(c_\rev,c_\rcd,c_\rcn,c_\rac,q_\rcn,q_\rac\), and \(c_\rcr\) are positive
constants.
Moist primitive-equation models of this type have been studied extensively in
the literature, see for instance
\cite{CZT:12,CZFTT:13,BCZT:14,CZHKTZ:15,HKLT:17,HKLT:20,HKLT:23}. These works are formulated after a
transformation to pressure coordinates. In particular, the pressure variable is
taken on a fixed interval \(p\in(p_0,p_1)\), so that the lower boundary is
represented by a prescribed pressure level rather than by the evolving physical
surface pressure. In contrast, we work directly in physical height coordinates.

\smallskip

\noindent
For simplicity of notation, and since this does not affect the analytical
structure of the problem, we impose from now on the normalization
\[
    p_\rref=1,
    \qquad
    R_\rd=R_\rv=1,
    \qquad
    c_\rpd=c_\rpv=2,
    \qquad
    c_1=1,
    \qquad
    L_\rref=T_\rref+1,
\]
and
\[
    c_\rev=c_\rcd=c_\rcn=q_\rcn=c_\rac=q_\rac=c_\rcr=1,
    \qquad
    \kappa=1.
\]
Then
\[
    \gamma=2,
    \qquad
    T=\theta p^{1/2},
\]
and
\[
    c_\nu
    =
    2+2q_\rv+q_\rc+q_\rr,
    \qquad
    \sigma
    =
    \frac12(q_\rc+q_\rr).
\]
Consequently,
\[
    Q_\rm
    =
    Q_\rth
    =
    1+q_\rv+q_\rc+q_\rr,
    \qquad
    Q_\rcp
    =
    -(1+q_\rv),
\]
and
\[
    Q_1=0,
    \qquad
    Q_2=1.
\]
The conversion terms reduce to
\[
\begin{aligned}
    S_\rev
    &=
    T
    \frac{1+q_\rv}{Q_\rm}
    (q_\rvs-q_\rv)^+q_\rr,
    \\
    S_\rcd
    &=
    (q_\rv-q_\rvs)q_\rc
    +
    (q_\rv-q_\rvs)^+,
    \\
    S_\rac
    &=
    (q_\rc-1)^+,
    \qquad
    S_\rcr
    =
    q_\rc q_\rr .
\end{aligned}
\]
For further details on the physical origin of these coefficients, we refer to
\cite[Section~1]{DKLT:24}.
Plugging these normalized quantities into
\eqref{eq:coupled moisture compr PE comp intro}, we arrive at the following
system on \(\Omega_\tau\):
\begin{equation}
    \label{eq:coupled moisture compr PE detailed}
    \left\{
    \begin{aligned}
        \partial_t\rho_\rd+\div(\rho_\rd u)
        &=0,
        \\
        \rho_\rd Q_\rm
        \bigl[
            \partial_t v+(u\cdot\nabla)v
        \bigr]
        -\rL_\rH v+\nablaH p
        &=
        \rho_\rd q_\rr V_\rr\partial_z v,
        \\
        \partial_z p
        &=
        -\rho_\rd Q_\rm g,
        \\
        Q_\rm
        \bigl[
            \partial_t T+(u\cdot\nabla)T
        \bigr]
        -\Delta T
        &=
        q_\rr V_\rr\partial_z T
        -(1+q_\rv)T\div u
        \\
        &\quad
        -
        \Bigl[
            T
            \frac{1+q_\rv}{Q_\rm}
            (q_\rvs-q_\rv)^+q_\rr
            -
            (q_\rv-q_\rvs)q_\rc
            -
            (q_\rv-q_\rvs)^+
        \Bigr],
        \\
        \partial_t q_\rv+(u\cdot\nabla)q_\rv-\Delta q_\rv
        &=
        T
        \frac{1+q_\rv}{Q_\rm}
        (q_\rvs-q_\rv)^+q_\rr
        -
        (q_\rv-q_\rvs)q_\rc
        -
        (q_\rv-q_\rvs)^+,
        \\
        \partial_t q_\rc+(u\cdot\nabla)q_\rc-\Delta q_\rc
        &=
        (q_\rv-q_\rvs)q_\rc
        +
        (q_\rv-q_\rvs)^+
        -
        (q_\rc-1)^+
        -
        q_\rc q_\rr,
        \\
        \partial_t q_\rr+(u\cdot\nabla)q_\rr-\Delta q_\rr
        &=
        \frac{1}{\rho_\rd}
        \partial_z(\rho_\rd q_\rr V_\rr)
        +
        (q_\rc-1)^+
        +
        q_\rc q_\rr
        -
        T
        \frac{1+q_\rv}{Q_\rm}
        (q_\rvs-q_\rv)^+q_\rr,
        \\
        p
        &=
        \rho_\rd(1+q_\rv)T .
    \end{aligned}
    \right.
\end{equation}
The model is completed by the boundary conditions
\begin{equation}\label{eq:bc fluid PE}
    \del_z v|_{z=0,1} =
    w |_{z=0,1}= 0,
\end{equation}
concerning the velocity
and we assume homogeneous Neumann boundary conditions for the other quantities, so
\begin{equation}\label{eq:bc heat & moisture PE}
    \del_z T|_{z=0,1} =
    \del_z q_j |_{z=0,1}= 0,
    \tfor j \in \{\rv,\rc,\rr\}.
\end{equation}
Finally, we take into account the initial conditions
\begin{equation}\label{eq:init conds PE}
    \rho_\rd(0) = \rho_{\rd,0},
    \enspace
    v(0) = v_0,
    \enspace
    T(0) = T_0,
    \enspace
    q_\rv(0) = q_{\rv,0},
    \enspace
    q_\rc(0) = q_{\rc,0}
    \tand
    q_\rr(0) = q_{\rr,0},
\end{equation}
which are required to satisfy the boundary conditions \eqref{eq:bc fluid PE} and \eqref{eq:bc heat & moisture PE} provided they are assumed to be sufficiently regular. In this context, let us remark that we use \(Q_{\rm,0}\) in order to denote
\[
Q_{\rm,0} = Q_\rm(0) = 1 + q_{\rv,0} + q_{\rc,0} + q_{\rr,0}.
\]
In the following, we reformulate the system \eqref{eq:coupled moisture compr PE detailed} taking into account the explicit vertical dependence of the pressure $p$ given the hydrostatic balance. 
Specifically,
we combine hydrostatic balance with the equation of state and obtain, assuming non degeneracy of the temperature $T$ and $1+ q_\rv$ that
\[
\del_z p
=
-\frac{gQ_{\rm}}{(1+q_{\rv})T}\,p.
\]
Hence the pressure admits the representation
\begin{equation}
    \label{eq: pressure moist}
    p(t,x,y,z)
    =
    p_s(t,x,y)\,
    \exp\Bigl(
        -\int_0^z
        \frac{gQ_{\rm}}{(1+q_{\rv})T}(\cdot,\eta)\,\rd\eta
    \Bigr),
\end{equation}
where $p_s$ denotes the pressure evaluated at $z=0$, that is, $ p_s(t,x,y)\coloneqq p(t,x,y,0).$
In particular, the pressure at the top boundary, \(p_t(t,x,y)\coloneqq p(t,x,y,1)\), is given by
\begin{equation*}
    p_t(t,x,y)
    =
    p_s(t,x,y)\,
    \exp\Bigl(
        -\int_0^1
        \frac{gQ_{\rm}}{(1+q_{\rv})T}(\cdot,\eta)\,\rd\eta
    \Bigr).
\end{equation*}
Using the equation of state, the dry air density can be written as
\begin{equation}
    \label{eq: density moist}
    \rho_{\rd}(t,x,y,z)
    =
    p_s(t,x,y)\,B(T,q_{\rv},q_{\rc},q_{\rr})(t,x,y,z),
\end{equation}
where the functional \(B\) is defined by 
\begin{equation}
    \label{eq: func B moist}
    B(T,q_{\rv},q_{\rc},q_{\rr})(t,x,y,z)
    \coloneqq
    \frac{1}{(1+q_{\rv}(t,x,y,z))\,T(t,x,y,z)}
    \exp\Bigl(
        -\int_0^z
        \frac{gQ_{\rm}}{(1+q_{\rv})T}(\cdot,\eta)\,\rd\eta
    \Bigr).
\end{equation}
Averaging vertically, we infer
\begin{equation}
    \label{eq: avg density moist}
    \bar{\rho}_{\rd}(t,x,y)
    =
    p_s(t,x,y)\,\bar{B}(T,q_{\rv},q_{\rc},q_{\rr})(t,x,y),
\end{equation}
where $\bar f$ denotes the vertical average of a sufficiently regular function $f$, namely $\bar f = \int_0^1 f(\cdot ,\eta) \rd \eta$.
Next, we introduce the normalized functional \(\hat{B}\), defined by the ration of $B$ and his vertical average $\bar B$, i.e,
\begin{equation}
    \label{eq: hat B moist}
    \hat{B}(T,q_{\rv},q_{\rc},q_{\rr})(t,x,y,z)
    \coloneqq
    \frac{B(T,q_{\rv},q_{\rc},q_{\rr})(t,x,y,z)}
    {\bar{B}(T,q_{\rv},q_{\rc},q_{\rr})(t,x,y)} \ \text{ and }\
    \int_0^1
    \hat{B}(T,q_{\rv},q_{\rc},q_{\rr})(\cdot,\eta)\,\rd\eta
    =1.
\end{equation}
With this notation, the dry air density can be expressed as 
\begin{equation}\label{eq:dry airdensity}
    \rho_{\rd}(t,x,y,z)
    =
    \bar{\rho}_{\rd}(t,x,y)\,
    \hat{B}(T,q_{\rv},q_{\rc},q_{\rr})(t,x,y,z),
\end{equation}
and the pressure is recovered by the ideal gas law.
Motivated by this identity, we vertically average the continuity equation and obtain
\begin{equation}
    \label{eq: avg cont moist}
    \del_t \bar{\rho}_{\rd} + \divH(\bar{\rho}_{\rd}\,b)=0,
    \qquad
    b(t,x,y)
    \coloneqq
    \int_0^1
    \bigl(
        \hat{B}(T,q_{\rv},q_{\rc},q_{\rr})\,v
    \bigr)(\cdot,\eta)\,\rd\eta.
\end{equation}
In the hydrostatic setting, the un-averaged continuity equation is no longer employed as an independent evolution law. Instead, once \(\bar{\rho}_{\rd}\) and \(\hat{B}(T,q_{\rv},q_{\rc},q_{\rr})\) are represented as above, it provides a diagnostic formula for the vertical velocity \(w\) in terms of \(\bar{\rho}_{\rd}\), \(T\), \(q_{\rv}\), \(q_{\rc}\), \(q_{\rr}\), and \(v\). More precisely, using the density representation, we infer that \(w\) is given by
\begin{equation}
    \label{eq:W-moist}
    \begin{aligned}
        \bigl(
            \bar{\rho}_{\rd}\,\hat{B}(T,q_{\rv},q_{\rc},q_{\rr})\,w
        \bigr)(t,x,y,z)
         &=
        -\int_0^z
        \Big[
            \hat{B}(T,q_{\rv},q_{\rc},q_{\rr})\,(v-b)\cdot\nablaH \bar{\rho}_{\rd}
            +\bar{\rho}_{\rd}\,
            \big(
                \del_t\hat{B}(T,q_{\rv},q_{\rc},q_{\rr})
                \\ &\quad+\divH(
                    \hat{B}(T,q_{\rv},q_{\rc},q_{\rr})\,v
                )
                -\hat{B}(T,q_{\rv},q_{\rc},q_{\rr})\divH b
            \big)
        \Big](\cdot,\eta)\,\rd\eta.
    \end{aligned}
\end{equation}
We emphasize that the boundary conditions \(w|_{z=0,1}=0\) are automatically satisfied by this representation whenever \(\bar{\rho}_{\rd}\) solves the averaged continuity equation \eqref{eq: avg cont moist}. In the following, to ease notation, we will write $$B = B(T, q_\rv, q_\rc, q_\rr)\ ,\ \bar{B } = \bar B(T, q_\rv, q_\rc, q_\rr)\ \text{ and }\ \hat{B } = \hat B(T, q_\rv, q_\rc, q_\rr) .$$
With these preparations, the system \eqref{eq:coupled moisture compr PE detailed} can be recast in the following form
\begin{equation}
    \left\{
    \begin{aligned}
        \del_t \bar{\rho}_{\rd} + \divH(\bar{\rho}_{\rd}\,b)
        &=0
        &&\text{ in }\T^2_\tau,\\
        \bar{\rho}_{\rd}\,\hat{B}\,Q_{\rm}
        \bigl[
            \del_t v + (u\cdot\nabla)v
        \bigr]
        -\rL_\rH v
        &=
        \bar{\rho}_{\rd}\,\hat{B}\,q_{\rr}V_{\rr}\del_z v -\nablaH\bigl(
            \bar{\rho}_{\rd}\hat{B}(1+q_{\rv})T
        \bigr)
        &&\text{ in }\Omega_\tau,\\
        Q_{\rm}\bigl[
            \del_t T + (u\cdot\nabla)T
        \bigr]
        -\Delta T
        &=
        q_{\rr}V_{\rr}\del_z T
        -(1+q_{\rv})T \div u\\
        &\quad
        -\Bigl[
            T\cdot \frac{1+q_{\rv}}{Q_{\rm}}(q_{\rvs}-q_{\rv})^+q_{\rr}
            -(q_{\rv}-q_{\rvs})q_{\rc}
            -(q_{\rv}-q_{\rvs})^+
        \Bigr]
        &&\text{ in }\Omega_\tau,\\
        \del_t q_{\rv} + (u\cdot\nabla)q_{\rv} - \Delta q_{\rv}
        &=
        T\cdot \frac{1+q_{\rv}}{Q_{\rm}}(q_{\rvs}-q_{\rv})^+q_{\rr}
        -(q_{\rv}-q_{\rvs})q_{\rc}
        -(q_{\rv}-q_{\rvs})^+
        &&\text{ in }\Omega_\tau,\\
        \del_t q_{\rc} + (u\cdot\nabla)q_{\rc} - \Delta q_{\rc}
        &=
        (q_{\rv}-q_{\rvs})q_{\rc}
        +(q_{\rv}-q_{\rvs})^+
        -(q_{\rc}-1)^+
        -q_{\rc}q_{\rr}
        &&\text{ in }\Omega_\tau,\\
        \del_t q_{\rr} + (u\cdot\nabla)q_{\rr} - \Delta q_{\rr}
        &=
        \frac{1}{\hat{B}}
        \del_z(
            \hat{B}\,q_{\rr}V_{\rr})
        +(q_{\rc}-1)^+ + q_{\rc}q_{\rr}
        -T\cdot \frac{1+q_{\rv}}{Q_{\rm}}(q_{\rvs}-q_{\rv})^+q_{\rr}
        &&\text{ in }\Omega_\tau,\\
        \bar{\rho}_{\rd}(0)
        &=
        \bar{\rho}_{\rd,0},
        \
        v(0)=v_0, \
        T(0)=T_0, \
        q_{j}(0)=q_{j,0} \text{ for } j \in \{ \rv, \rc, \rr \}.
    \end{aligned}
    \right.
    \label{eq:coupled moisture compr PE reformulated}
\end{equation}
In this formulation, the prognostic variables are the averaged dry air density \(\bar{\rho}_{\rd}\), the horizontal velocity \(v\), the temperature \(T\), and the mixing ratios \(q_{\rv}\), \(q_{\rc}\), and \(q_{\rr}\). The remaining quantities are determined diagnostically. Specifically,
the averaged flux is
\[
b
=
\int_0^1
(
    \hat{B}\,v
)(\cdot,\eta)\,\rd\eta,
\]
with $\hat B$ as in \eqref{eq: hat B moist}
and the vertical velocity \(w\) is recovered from \eqref{eq:W-moist}.

\section{Main results}\label{sec:main results}
\noindent
We are now in the position to formulate our main theorem concerning the global, strong well-posedness for initial data close to the constant steady state
\[
(\bar{\rho}_{\rd}^\ast,0,T^\ast,0,0,0)
\]
of the system \eqref{eq:coupled moisture compr PE detailed} subject to the boundary conditions \eqref{eq:bc fluid PE} and \eqref{eq:bc heat & moisture PE} and supplemented by the initial conditions \eqref{eq:init conds PE}. Here, it is assumed that $T^\ast>0$ is chosen in a way that $q_\rvs =0$. We emphasize that regimes where $q_\rvs$ has a gap to $0$ can also be treated, see \autoref{subsec: nonlinear estimates}. Before stating the main theorem, we formulate the assumptions on the initial data.

\begin{asu}\label{assu:data0 moist}
Let \(\tau>0\) and \(\bar{\rho}_{\rd}^\ast,T^\ast>0\). Assume that
\begin{equation*}
    (\bar{\rho}_{\rd,0}, v_0, T_0, q_{j,0})\in \rH^2(\T^2)
   \times \rH^2(\Omega;\R^2) \times \rH^2(\Omega)
   \times 
   \rH^2(\Omega) 
\end{equation*}
and
\begin{equation*}
    \del_z v_0|_{z=0,1}=
    \del_z T_0|_{z=0,1}=
    \del_z q_{j,0}|_{z=0,1}=0\ \text{ for } j\in\{\rv,\rc,\rr\}.
\end{equation*}
Moreover, assume that the first-order compatibility quantities
\(
( \dt \bar \rho_\rd ,
\del_t v,
\del_t T,
\del_t q_j)|_{t=0}
\)
which are given by the respective right-hand sides of \eqref{eq:coupled moisture compr PE reformulated} evaluated at the initial data, satisfy
\[
( \dt \bar \rho_\rd,
\del_t v,
\del_t T,
\del_t q_j)|_{t=0} \in \rH^1(\T^2) \times  \rH^1(\Omega;\R^2) \times \rH^1(\Omega) \times \rH^1(\Omega) \ \text{ for } j\in\{\rv,\rc,\rr\}. 
\]
Finally, for some $\eps>0$ suppose that
\begin{equation*}
 \begin{aligned}
       &\quad \|\bar{\rho}_{\rd,0}-\bar{\rho}_{\rd}^\ast\|_{\rH^2(\T^2)}
    +
    \|v_0\|_{\rH^2(\Omega)}
    +
    \|T_0-T^\ast\|_{\rH^2(\Omega)}
    +
    \sum_{j\in\{\rv,\rc,\rr\}}
    \|q_{j,0}\|_{\rH^2(\Omega)} \\&  +
        \|(\dt \bar \rho_\rd)|_{t=0}\|_{\rH^1(\T^2)}
        +
        \|(\dt v)|_{t=0}\|_{\rH^1(\Omega)}
        +
        \|(\dt T)|_{t=0}\|_{\rH^1(\Omega)}
        +
        \sum_{j\in\{\rv,\rc,\rr\}}
        \|(\dt q_j)|_{t=0}\|_{\rH^1(\Omega)}
    \le \eps^2.
 \end{aligned}
\end{equation*}
\end{asu}
\noindent
The main theorem of this article then reads as follows.
\begin{thm}[Global, strong well-posedness of \eqref{eq:coupled moisture compr PE detailed}] \mbox{} \\
    \label{thm: global WP}
    Let \(\tau>0\) and \(\bar{\rho}_{\rd}^\ast,T^\ast>0\). Assume that the initial data
    \[
    (\bar{\rho}_{\rd,0},v_0,T_0,q_{\rv,0},q_{\rc,0},q_{\rr,0})
    \]
    satisfy \autoref{assu:data0 moist}. Then there exists
    \[
    \eps_0=\eps_0(\tau,\bar{\rho}_{\rd}^\ast,T^\ast,\Omega)>0
    \]
    such that for every \(\eps\in(0,\eps_0)\) the reformulated system
    \eqref{eq:coupled moisture compr PE reformulated} subject to the boundary conditions
    \eqref{eq:bc fluid PE} and \eqref{eq:bc heat & moisture PE} admits a unique strong solution
    \[
    (\bar{\rho}_{\rd},v,T,q_{\rv},q_{\rc},q_{\rr})
    \]
    with
    \begin{equation*}
    \begin{aligned}
        \bar{\rho}_{\rd}
        &\in \rH^{1}(0,\tau;\rH^{2}(\T^2)) \cap \rH^2(0,\tau;\rL^2(\T^2)), \\
        v
        &\in \rH^{1}(0,\tau;\rH^2(\Omega;\R^2)) \cap \rL^2(0,\tau;\rH^{3}(\Omega;\R^2)) \cap \rH^2(0,\tau; \rL^2(\Omega;\R^2)), \\
        T
        &\in \rH^{1}(0,\tau;\rH^2(\Omega)) \cap \rL^2(0,\tau;\rH^{3}(\Omega))\cap \rH^2(0,\tau; \rL^2(\Omega)), \\
        q_j
        &\in \rH^{1}(0,\tau;\rH^2(\Omega)) \cap \rL^2(0,\tau;\rH^{3}(\Omega))\cap \rH^2(0,\tau; \rL^2(\Omega)) \ \text{ for } j\in\{\rv,\rc,\rr\}.
    \end{aligned}
    \end{equation*}
    With \(\hat{B}(T,q_{\rv},q_{\rc},q_{\rr})\) as in \eqref{eq: hat B moist}, define the diagnostic quantities
    \[
    \rho_{\rd}:=\bar{\rho}_{\rd}\,\hat{B}(T,q_{\rv},q_{\rc},q_{\rr}),
    \quad
    p:=\rho_{\rd}(1+q_{\rv})T,
    \ \text{ and }\ 
    w \text{ by } \eqref{eq:W-moist}.
    \]
    Then
    \[
    (\rho_{\rd},u=(v,w),p,T,q_{\rv},q_{\rc},q_{\rr})
    \]
    is a unique, strong solution of the full coupled moisture--compressible primitive equation system
    \eqref{eq:coupled moisture compr PE detailed}. Moreover, for all \((t,x,y,z)\in[0,\tau]\times\Omega\),
    \begin{equation*}
      \frac{\bar{\rho}_\rd^\ast \hat{B}(T^\ast)(z)}{2} \le \rho_\rd(t,x,y,z)
\le \frac{3\,\bar{\rho}_\rd^\ast \hat{B}(T^\ast)(z)}{2} \quad \text{ and } \quad
        \frac{T^\ast}{2}
        \le
        T(t,x,y,z)
        \le
        \frac{3\,T^\ast}{2}.
    \end{equation*}
\end{thm}

\section{Moist Lagrangian Coordinates}\label{sec: Lagrange moist}

\noindent
Throughout the rest of this article we normalize the gravity constant to be equal to $1$, that is $g=1$.
In order to circumvent the hyperbolic effects that arise in the newly derived continuity equation
\eqref{eq: avg cont moist}, it is natural to introduce characteristics associated with the flow field \(b\). 
Define the flow \(\rX\) as the solution of the differential equation
\begin{equation}
    \label{eq:flow moist}
    \left\{
    \begin{aligned}
        \dt \rX(t,x,y)
        &= b\bigl(t,\rX(t,x,y)\bigr),
        \qquad t>0,\\
        \rX(0,x,y)
        &=(x,y),
    \end{aligned}
    \right.
\end{equation}
for \((x,y)\in\T^2\). Note that \(\rX\) is a two-dimensional flow that depends on the temperature $T$ as well as the mixing ratios $q_j$ for $j \in \{\rv, \rc,\rr\}$ through the nonlinear and nonlocal functional $\hat{B}$. Solving \eqref{eq:flow moist} leads to
\begin{equation}
    \label{eq:diffeo X moist}
    \rX(t,x,y)
    =
    (x,y)^\top
    +
    \int_0^t
    b\bigl(s,\rX(s,x,y)\bigr)\rd s
\end{equation}
for \((x,y)\in\T^2\). Note that in contrast to the transform introduced in \cite{TZ:26}, the flow field $b$ also depends on the mixing ratios $q_j$ for $j \in \{\rv,\rc,\rr\}$.
In the following, we show that the flow \(\rX(t,\cdot)\colon\T^2\to\T^2\) is a well-defined \(\rC^1\)-diffeomorphism provided that the temperature $T$, the mixing ratios $q_j$, and the horizontal velocity $v$ are given and sufficiently regular. Precisely, for \(\tau>0\) and assume that
\[
v\in \E^v_1(0,\tau) \ \text{ and } \
T,\ q_{\rv},\ q_{\rc},\ q_{\rr}\in \E^h_1(0,\tau),
\]
where 
\begin{equation*}
    \label{eq: assu v T q spaces}
    \E^v_1(0,\tau)
    \coloneqq
    \rL^2(0,\tau;\rH^3(\Omega;\R^2))
    \cap
    \rH^1(0,\tau;\rH^2(\Omega;\R^2)) \cap
\rH^2(0,\tau;\rL^2(\Omega;\R^2))
\end{equation*}
and
\begin{equation*}
    \label{eq: assu scalars spaces}
    \E^h_1(0,\tau)
    \coloneqq
    \rL^2(0,\tau;\rH^3(\Omega))
    \cap
    \rH^1(0,\tau;\rH^2(\Omega))\cap
    \rH^2(0,\tau;\rL^2(\Omega)).
\end{equation*}
Moreover, let \(T^\ast>0\) and assume that
\begin{equation*}
    \label{eq: smallness v T q}
    \|v\|_{\E^v_1(0,\tau)}
    +
    \|T-T^\ast\|_{\E^h_1(0,\tau)}
    +
    \sum_{j\in\{\rv,\rc,\rr\}}
    \|q_j\|_{\E^h_1(0,\tau)}
    \le \eps
\end{equation*}
for some \(\eps>0\).
Using the embeddings
\[
\E^h_1(0,\tau)
\hookrightarrow
\rL^\infty(0,\tau;\rH^2(\Omega))
\hookrightarrow
\rL^\infty(0,\tau;\rL^\infty(\Omega)),
\]
there is a constant \(C_{\mathrm{emb}}(\Omega,\tau)>0\) such that
\begin{equation*}
    \|T-T^\ast\|_{\rL^\infty(0,\tau;\rL^\infty(\Omega))}
    +
    \sum_{j\in\{\rv,\rc,\rr\}}
    \|q_j\|_{\rL^\infty(0,\tau;\rL^\infty(\Omega))}
    \le
    C_{\mathrm{emb}}\,\eps.
\end{equation*}
Choosing
\(
C_{\mathrm{emb}}\,\eps
<
\min\{
    {T^\ast}/{2},
  1/2
\},
\)
we infer that
\begin{equation}
    \label{eq: nondegeneracy T q}
    \frac{T^\ast}{2}
    \le T(t,x,y,z)
    \le \frac{3T^\ast}{2}, \quad 
    \frac{1}{2}
    \le 1+q_{\rv}(t,x,y,z)
    \le \frac{3}{2}   \ \text{ and }\   \frac{1}{2}
    \le 1+ Q_{\rm}(t,x,y,z)
    \le \frac{3}{2}
\end{equation}
for all \((t,x,y,z)\in[0,\tau]\times\Omega\).
Since all maps related to the definition of $\hat{B}$ are smooth in the regime \eqref{eq: nondegeneracy T q}, the operator $(T,q_j) \mapsto \hat{B}(T,q_j)$ is locally Lipschitz from $\E^h_1(0,\tau)^4$ to $\E^h_1(0,\tau)$ and standard composition estimates yield that
\begin{equation}\label{eq: deltaB fund}
    \| \hat{B} - \hat{B}^\ast(z) \|_{\E^h_1(0,\tau)} \leq C \Big (  \| T -T^\ast \|_{\E^h_1(0,\tau)}+  \sum_{j\in\{\rv,\rc,\rr\}}
    \|q_j\|_{\E^h_1(0,\tau)} \Big )\leq C\eps,  
\end{equation}
where $\hat{B}^\ast(z)$ is the induced state for the functional $\hat{B}$ coming from the constant state $(T^\ast,0,0,0)$, that is,
\begin{equation*}
    \hat{B}^\ast(z) := \frac{\mathrm{exp}(-z/ T^\ast)}{T^\ast (1- \mathrm{exp}(-1/T^\ast)) }.
\end{equation*}
We note that the state $\hat{B}^\ast:=\hat{B}^\ast(z)$ depends on the vertical variable, which is induced by hydrostatic balance $\dz p  
=-\bar{\rho}_\rd \hat{B}Q_\rm$.  Furthermore, we obtain
\begin{equation*}
    \frac{2}{3 \, T^\ast} \, \frac{\mathrm{exp}(-2 / T^\ast)}{1- \mathrm{exp}(-2 / T^\ast)} \leq \hat{B}(\Theta(t,x,y,z)) \leq \frac{2}{T^\ast} \, \frac{1}{1- \mathrm{exp}(-2 / (3\,T^\ast))} \ \text{ for all }\ (t,x,y,z) \in [0,\tau] \times \Omega.
\end{equation*}
Next, we estimate the transport field
\(b\). Since \(\rH^2(\T^2)\) is a Banach algebra with respect to point-wise multiplication, we infer that
\begin{equation*}
    \begin{aligned}\label{eq: est b space}
        \| \nablaH b(t) \|_{\rH^{2}(\T^2)} &\leq \int_0^1 \| \nablaH \hat{B}(t,\cdot,\eta) \|_{\rH^{2}(\T^2)}\, \| v(t,\cdot,\eta) \|_{\rH^{2}(\T^2)} + \| \hat{B}(t,\cdot,\eta)\|_{\rH^{2}(\T^2)}\,  \| \nablaH v(t,\cdot,\eta) \|_{\rH^{2}(\T^2)} \rd \eta \\
        &\leq  C\| \hat{B}(t) \|_{\rH^{3}(\Omega)} \, \| v(t) \|_{\rH^{2}(\Omega)} + \| \hat{B}(t) \|_{\rH^{2}(\Omega)} \, \|v(t)\|_{\rH^{2}(\Omega)}
    \end{aligned}
\end{equation*}
and therefore
\begin{equation*}
   \begin{aligned}
        \| \nablaH b \|_{\rL^2(0,\tau;\rH^{2}(\T^2))} & \leq  \| \hat{B} \|_{\rL^2(0,\tau;\rH^{3}(\Omega))} \, \| v \|_{\rL^\infty(0,\tau;\rH^{2}(\Omega))} + \| \hat{B} \|_{\rL^\infty(0,\tau;\rH^{2}(\Omega))} \, \|v\|_{\rL^2(0,\tau;\rH^{3}(\Omega))} \\ &\leq C \| \hat{B} \|_{\E^h_1(0,\tau)} \, \| v \|_{\E^v_1{(0,\tau)}} 
        \\&\leq C \eps.
   \end{aligned}
\end{equation*}
Differentiating \eqref{eq:diffeo X moist} with respect to the horizontal variables and using standard estimates for Sobolev flows, we infer that there is a constant \(C>0\) such that
\begin{equation*}
    \label{eq:est of nablaH X - Id moist}
    \sup_{t\in[0,\tau]}
    \|\nablaH \rX(t,\cdot)-\mathrm{I}_2\|_{\rH^2(\T^2)}
    \le
    C\int_0^\tau
    \|\nablaH b(s,\cdot)\|_{\rH^2(\T^2)}\rd s
    \le
    C\tau^{\nicefrac12}\eps.
\end{equation*}
Since \(\rH^2(\T^2)\hookrightarrow\rL^\infty(\T^2)\), this yields
\begin{equation*}
    \label{eq:est of nablaHX - Id in Linfty moist}
    \sup_{t\in[0,\tau]}
    \|\nablaH \rX(t,\cdot)-\mathrm{I}_2\|_{\rL^\infty(\T^2)}
    \le
    C\tau^{\nicefrac12}\eps.
\end{equation*}
Hence, choosing
\[
\alpha
\coloneqq
\min\Bigl\{
    \frac{1}{2C\tau^{\nicefrac12}},
    \frac{T^\ast}{2C_{\mathrm{emb}}},
    \frac{1}{2C_{\mathrm{emb}}},
\Bigr\}
\]
and assuming \(\eps\in(0,\alpha)\), we obtain
\begin{equation}
    \label{eq:est of nablaHX - Id moist half}
    \sup_{t\in[0,\tau]}
    \|\nablaH \rX(t,\cdot)-\mathrm{I}_2\|_{\rL^\infty(\T^2)}
    \le
    \frac12.
\end{equation}
A Neumann series argument then guarantees invertibility of \(\nablaH \rX\). Denoting by \(\mathrm{Y}(t,\cdot)\) the inverse of \(\rX(t,\cdot)\), we find that
\[
\nablaH \mathrm{Y}(t,\rX(t,\cdot))
=[\nablaH \rX]^{-1}(t,\cdot).
\]
Writing
\begin{equation*}
    \label{eq: ex of nablaH Y moist}
    \mathrm{Z}(t,\cdot)
    \coloneqq
    [\nablaH \rX]^{-1}(t,\cdot)
    =
    \frac{1}{\mathrm{J}(t,\cdot)}
    \bigl(\Cof \nablaH \rX(t,\cdot)\bigr)^\top \ \text{ with } \ \mathrm{J}(t,\cdot) ;= \det\nablaH \rX(t,\cdot)
\end{equation*}
we infer from \eqref{eq:est of nablaHX - Id moist half} that
\[
\mathrm{J}\geq  C>0
\quad\text{on }(0,\tau)\times\T^2
\]
for some constant \(C>0\). Moreover, since
\(
\rL^\infty(0,\tau;\rH^2(\T^2))
\cap
\rH^1(0,\tau;\rH^2(\T^2))
\)
is a Banach algebra with respect to point-wise multiplication, we obtain
\begin{equation*}
    \label{eq: est det and cof moist}
    \begin{aligned}
        \|\mathrm{J}\|_{\rL^\infty(0,\tau;\rH^2(\T^2))}
        +
        \|\mathrm{J}\|_{\rH^1(0,\tau;\rH^2(\T^2))}
       +
        \|\Cof\nablaH \rX\|_{\rL^\infty(0,\tau;\rH^2(\T^2))}
        +
        \|\Cof\nablaH \rX\|_{\rH^1(0,\tau;\rH^2(\T^2))}
        \le C.
    \end{aligned}
\end{equation*}
Consequently,
\begin{equation*}
    \label{eq: est Z moist}
    \|\mathrm{Z}\|_{\rL^\infty(0,\tau;\rH^2(\T^2))}
    +
    \|\mathrm{Z}\|_{\rH^1(0,\tau;\rH^2(\T^2))}
    \le C,
\end{equation*}
and, in particular we obtain
\begin{equation}
    \label{eq: est Z minus I moist}
    \|\mathrm{Z}-\mathrm{I}_2\|_{\rL^\infty(0,\tau;\rH^2(\T^2))}
    +
    \|\mathrm{J}-1\|_{\rL^\infty(0,\tau;\rH^2(\T^2))}
    \le
    C\eps.
\end{equation}
Furthermore, for all \(j,k,l\in\{1,2\}\), there is a constant \(C>0\) such that
\begin{equation}
    \label{eq: est dZ moist}
    \Bigl\|
        \frac{\partial \mathrm{Z}_{l,j}}{\partial y_k}
    \Bigr\|_{\rL^\infty(0,\tau;\rH^1(\T^2))}
    \le
    C\eps.
\end{equation}
Finally, the identity
\[
    \dt \mathrm Z
    =
    -\mathrm Z(\dt\nablaH\rX)\mathrm Z .
\]
yields
\[
    \dt\mathrm Z
    =
    -\mathrm Z\nablaH b^\rL .
\]
Consequently, there is a
constant \(C>0\) such that
\begin{equation}
    \label{eq: est dtZ moist}
    \|\dt\mathrm Z\|_{\rL^2(0,\tau;\rH^2(\T^2))}
    +
    \|\dt\mathrm Z^\top\|_{\rL^2(0,\tau;\rH^2(\T^2))} +  \|\dt\mathrm J\|_{\rL^2(0,\tau;\rH^2(\T^2))}
    \le
    C\eps .
\end{equation}
The above properties of the Lagrangian transformation \(\rX\) are summarized in the following.

\begin{lem}\label{lem:ests of trafo moist}
Let \(\tau>0\) and assume that $(v,T,q_j) \in \E^v_1(0,\tau) \times \E^h_1(0,\tau) \times \E^h_1(0,\tau)$ for $j \in \{ \rv,\rc,\rr\}$.
Suppose that
\[
\|v\|_{\E^v_1(0,\tau)}
+
\|T-T^\ast\|_{\E^h_1(0,\tau)}
+
\sum_{j\in\{\rv,\rc,\rr\}}
\|q_j\|_{\E^h_1(0,\tau)}
\le
\eps
\]
for some \(\eps=\eps(\tau,T^\ast,\Omega)>0\) sufficiently small. Denote by \(\rX\) the flow associated with the field \(b\) as defined in \eqref{eq:flow moist}.

\begin{enumerate}[(a)]
    \item There is a constant \(C>0\) such that
    \begin{equation*}
        \sup_{t\in(0,\tau)}
        \|\nablaH \rX(t,\cdot)-\mathrm{I}_2\|_{\rL^{\infty}(\T^2)}
        +
        \sup_{t\in(0,\tau)}
        \|\nablaH \rX(t,\cdot)-\mathrm{I}_2\|_{\rH^2(\T^2)}
        \le
        C\eps.
    \end{equation*}
    In particular, if \(\eps\in(0,\alpha)\), then
    \begin{equation*}
        \|\nablaH \rX-\mathrm{I}_2\|_{\rL^\infty(0,\tau;\rW^{1,\infty}(\T^2))}
        \le
        \frac12,
    \end{equation*}
    so \(\nablaH \rX(t,\cdot)\) is invertible for all \(t\in[0,\tau]\), and \(\mathrm{Z}=[\nablaH \rX]^{-1}\) is well-defined.

    \item Similar estimates as in (a) are valid for \(\mathrm{Z}\) amd $\mathrm{J}$, namely
    \begin{equation*}
        \begin{aligned}
            \sup_{t\in(0,\tau)} \Big (
            \|\mathrm{Z}(t,\cdot)-\mathrm{I}_2\|_{\rH^2(\T^2)}
            +
            \|\mathrm{Z}^\top(t,\cdot)-\mathrm{I}_2\|_{\rH^2(\T^2)}+
            \|\mathrm{J}(t,\cdot)-1\|_{\rH^2(\T^2)} + 
            \|\mathrm{J}^\top(t,\cdot)-1\|_{\rH^2(\T^2)} \Big )
            \le
            C\eps,
        \end{aligned}
    \end{equation*}
    for some constant \(C>0\).
    \item For all \(j,k,l\in\{1,2\}\), there exists a constant \(C>0\) such that
    \begin{equation*}
        \Bigl\|
            \frac{\partial \mathrm{Z}_{l,j}}{\partial y_k}
        \Bigr\|_{\rL^\infty(0,\tau;\rH^1(\T^2))}
        \le
        C\eps.
    \end{equation*}
        \item The time derivatives of the transformation quantities satisfy
    \begin{equation*}
        \|\dt \mathrm Z\|_{\rL^2(0,\tau;\rH^2(\T^2))}
        +
        \|\dt \mathrm Z^\top\|_{\rL^2(0,\tau;\rH^2(\T^2))}
        +
        \|\dt \mathrm J\|_{\rL^2(0,\tau;\rH^2(\T^2))}
        \le
        C\eps ,
    \end{equation*}
   for some constant $C>0$.
\end{enumerate}
\end{lem}


\medskip 
\section{Moist Lagrangian formulation of the model}\label{sec Lagrange model}

\noindent 
With the estimate of the transformation at hand, we introduce the new unknowns following the characteristics \(\rX\). For \(\bar{\rho}_{\rd}^\ast>0\), \(T^\ast>0\), 
we define
\begin{equation*}
    \begin{aligned}
        \bar{\rho}_{\rd}^\rL(t,\cdot)
:=
\bar{\rho}_{\rd}(t,\rX(t,\cdot))-\bar{\rho}_{\rd}^\ast, \ v^\rL(t,\cdot,z)
&:=
v(t,\rX(t,\cdot),z), \ T^\rL(t,\cdot,z)
:=
T(t,\rX(t,\cdot),z)-T^\ast \text{ and}\\
q_j^\rL(t,\cdot,z)
&:=
q_j(t,\rX(t,\cdot),z)
\ \text{ for } 
j\in\{\rv,\rc,\rr\}.
    \end{aligned}
\end{equation*}
The remaining quantities are defined diagnostically. Specifically, the transformed functional \(\hat B^\rL\) is then defined by
\[
\hat B^\rL(t,\cdot,z)
:=
\hat B\bigl(T(t,\rX(t,\cdot),z),q_{\rv}(t,\rX(t,\cdot),z),q_{\rc}(t,\rX(t,\cdot),z),q_{\rr}(t,\rX(t,\cdot),z)\bigr)
=
\hat B^\ast+\delta \hat B( T^\rL,  q_\rv^\rL, q_\rc^\rL, q_\rr^\rL),
\]
where
\[
\delta \hat B ( T^\rL,  q_\rv^\rL, q_\rc^\rL, q_\rr^\rL)
:=
\hat B(T^\rL+T^\ast, q_\rv^\rL, q_\rc^\rL, q_\rr^\rL)-\hat B^\ast.
\]
In the following, we will also suppress the dependence of $\delta \hat B$ on the transformed temperature and mixing ratios and write
\begin{equation*}
    \delta \hat{B} = \delta \hat B ( T^\rL,  q_\rv^\rL, q_\rc^\rL, q_\rr^\rL).
\end{equation*}
As in \eqref{eq: deltaB fund}, we observe that in the regime \eqref{eq: nondegeneracy T q}, the operator
\(
(T,q_j)\mapsto \hat B(T,q_j)
\)
is locally Lipschitz from \(\E_1^h(0,\tau)^4\) to \(\E_1^h(0,\tau)\). Therefore, there exists a constant \(C>0\) such that
\[
\|\delta \hat B\|_{\E_1^h(0,\tau)}
\le
C\Bigl(
\|T^\rL\|_{\E_1^h(0,\tau)}
+
\sum_{j\in\{\rv,\rc,\rr\}}
\|q_j^\rL\|_{\E_1^h(0,\tau)}
\Bigr).
\]
In the following, we calculate the Fr\'echet derivative of \(\hat B\) explicitly and deduce the corresponding linearized operator \(D\hat B^\ast\), that is, \(D\hat B\) evaluated at the reference state \((T^\ast,0,0,0)\).
To this end, we set
\begin{equation*}
    A((T,q_{\rv},q_{\rc},q_{\rr})(z)
    :=
    \int_0^z
    \frac{1+q_{\rv}+q_{\rc}+q_{\rr}}{(1+q_{\rv})T}(\cdot,\eta)\rd\eta,
\end{equation*}
Let \((h_T,h_{\rv},h_{\rc},h_{\rr})\) be a direction. A direct differentiation yields
\begin{equation*}
    \begin{aligned}
        &\quad (D\hat B)(T,q_{\rv},q_{\rc},q_{\rr})[h_T,h_{\rv},h_{\rc},h_{\rr}](z)
        \\
        &=
        \frac{\exp(-A(z))}{\bar B}
        \Biggl(
            -\frac{h_T(z)}{(1+q_{\rv})T^2}
            -\frac{h_{\rv}(z)}{(1+q_{\rv})^2T}
          \\&\quad  -\frac{1}{(1+q_{\rv})T}
            \int_0^z
            \Biggl[
                -\frac{1+q_{\rv}+q_{\rc}+q_{\rr}}{(1+q_{\rv})T^2}h_T
                -\frac{q_{\rc}+q_{\rr}}{(1+q_{\rv})^2T}h_{\rv}
                +\frac{1}{(1+q_{\rv})T}(h_{\rc}+h_{\rr})
            \Biggr](\cdot,\eta)\,\rd\eta
        \Biggr)
        \\
        &\quad
        -\frac{B}{\bar B^2}
        \int_0^1
        \exp(-A(\xi))
        \Biggl(
            -\frac{h_T(\xi)}{(1+q_{\rv})T^2}
            -\frac{h_{\rv}(\xi)}{(1+q_{\rv})^2T}
           \\ &\quad -\frac{1}{(1+q_{\rv})T}
            \int_0^\xi
            \Biggl[
                -\frac{1+q_{\rv}+q_{\rc}+q_{\rr}}{(1+q_{\rv})T^2}h_T
                -\frac{q_{\rc}+q_{\rr}}{(1+q_{\rv})^2T}h_{\rv}
                +\frac{1}{(1+q_{\rv})T}(h_{\rc}+h_{\rr})
            \Biggr](\cdot,\eta)\rd\eta
        \Biggr)\rd\xi .
    \end{aligned}
\end{equation*}
Hence, \((D\hat B)(T,q_{\rv},q_{\rc},q_{\rr})\) is a sum of multiplication operators and \(z\)-integral operators.
Moreover, the linearized operator \(D\hat B^\ast=(D\hat B)(T^\ast,0,0,0)\) is explicitly given by
\begin{equation}\label{eq: DB lin moist}
    \begin{aligned}
       &\quad D\hat B^\ast[h_T,h_{\rv},h_{\rc},h_{\rr}](z)
        \\&=
        \hat B^\ast(z)
        \Biggl(
            -\frac{h_T(z)}{T^\ast}
            -h_{\rv}(z)
            +\frac{1}{(T^\ast)^2}\int_0^z h_T(\eta)\rd\eta
            -\frac{1}{T^\ast}\int_0^z (h_{\rc}(\eta)+h_{\rr}(\eta))\rd\eta
            -C[h_T,h_{\rv},h_{\rc},h_{\rr}]
        \Biggr),
    \end{aligned}
\end{equation}
where
\begin{equation*}
    \begin{aligned}
        C[h_T,h_{\rv},h_{\rc},h_{\rr}]
        &:=
        \int_0^1
        \hat B^\ast(\xi)
        \Biggl[
            -\frac{h_T(\xi)}{T^\ast}
            -h_{\rv}(\xi)
            +\frac{1}{(T^\ast)^2}\int_0^\xi h_T(\eta)\rd\eta
            -\frac{1}{T^\ast}\int_0^\xi (h_{\rc}(\eta)+h_{\rr}(\eta))\rd\eta
        \Biggr]\rd\xi .
    \end{aligned}
\end{equation*}
As before, we split the Fr\'echet derivative into its linearized part at the reference state and the nonlinear remainder by setting
\begin{equation*}
    \delta(D\hat B)(T^\rL,q_{\rv}^\rL,q_{\rc}^\rL,q_{\rr}^\rL)
    :=
    (D\hat B)(T^\ast+T^\rL,q_{\rv}^\rL,q_{\rc}^\rL,q_{\rr}^\rL)
    -
    D\hat B^\ast .
\end{equation*}
Notice that the maps
\[
(T,q_{\rv})
\mapsto
\frac{1}{(1+q_{\rv})T},
\quad
(T,q_{\rv},q_{\rc},q_{\rr})
\mapsto
\frac{1+q_{\rv}+q_{\rc}+q_{\rr}}{(1+q_{\rv})T} \ \text{ and } \ G\mapsto \int_0^z G(\cdot,\eta)\rd\eta,
\
 \exp(-G), \
 \frac{1}{G}
\]
are all locally Lipschitz on \(\E_1^h(0,\tau)^4\) and $\E_1^h(0,\tau)$ respectively, in the regime \eqref{eq: nondegeneracy T q}, while multiplication and integration are bounded operators. Consequently, the operator-valued map
\[
(T,q_{\rv},q_{\rc},q_{\rr})\mapsto (D\hat B)(T,q_{\rv},q_{\rc},q_{\rr})
\]
is locally Lipschitz from \(\E_1^h(0,\tau)^4\) into \(\mathcal L(\E_1^h(0,\tau)^4,\E_1^h(0,\tau))\). In particular, we obtain
\begin{equation*}
    \|\delta(D\hat B)\|_{\mathrm{op}}
    \le
    C\Bigl(
        \|T^\rL\|_{\E_1^h(0,\tau)}
        +
        \sum_{j\in\{\rv,\rc,\rr\}}
        \|q_j^\rL\|_{\E_1^h(0,\tau)}
    \Bigr).
\end{equation*}
Next, by the chain rule,
\begin{equation}\label{eq:hori deriv of hat B moist final}
    \begin{aligned}
        \nablaH \hat B^\rL
        &=
        \bigl(
            D\hat B^\ast
            +
            \delta(D\hat B)(T^\rL,q_{\rv}^\rL,q_{\rc}^\rL,q_{\rr}^\rL)
        \bigr)
        \bigl[
            \nablaH T^\rL,\nablaH q_{\rv}^\rL,\nablaH q_{\rc}^\rL,\nablaH q_{\rr}^\rL
        \bigr].
    \end{aligned}
\end{equation}
In particular, the transformed averaged velocity field can be written as
\begin{equation*}
    b^\rL(t,\cdot)
    :=
    b(t,\rX(t,\cdot))
    =
    \int_0^1 \hat B^\ast\,v^\rL(t,\cdot,z)\,\rd z
    +
    \int_0^1 \delta\hat B(t,\cdot,z)\,v^\rL(t,\cdot,z)\,\rd z.
\end{equation*}
Consequently, its horizontal divergence is given by
\begin{equation}\label{eq: divH b moist}
    \begin{aligned}
        \divH b^\rL
        &=
        \int_0^1
        \Big[
            \nablaH \hat B^\rL \cdot v^\rL
            +
            \hat B^\rL\,\divH v^\rL
        \Big](\cdot,z)\,\rd z
        \\
        &=
        \int_0^1
        \Big[
            \bigl(
                D\hat B^\ast
                +
                \delta(D\hat B)
            \bigr)
            \bigl[
                \nablaH T^\rL,\nablaH q_{\rv}^\rL,\nablaH q_{\rc}^\rL,\nablaH q_{\rr}^\rL
            \bigr]
            \cdot v^\rL
            +
            \bigl(
                \hat B^\ast+\delta\hat B
            \bigr)\,\divH v^\rL
        \Big](\cdot,z)\,\rd z.
    \end{aligned}
\end{equation}
Next, we transform the diagnostic formula for the vertical velocity. Using \eqref{eq:W-moist}, together with the above representation of \(\hat B^\rL\) , 
    \(\dt \hat B^\rL\)
   and \(\nablaH \hat B^\rL\) , we obtain
using the shorthand
\(
\tilde v^\rL := v^\rL-b^\rL
\) 
the transformed formula
\begin{equation}\label{eq:wL_rhobar_hatB_moist}
    \begin{aligned}
       &\quad \bigl((\bar{\rho}_{\rd}\hat{B}w)^\rL\bigr)(t,\cdot,z)
       \\
       &=
        -\int_0^z
        \Big[
            \hat{B}^\rL \tilde v^\rL \cdot \mathrm{Z}^\top \nablaH \bar{\rho}_{\rd}^\rL
            +
            (\bar{\rho}_{\rd}^\rL+\bar{\rho}_{\rd}^\ast)
            \hat{B}^\rL
            \nablaH \tilde v^\rL : \mathrm{Z}^\top
            +
            (\bar{\rho}_{\rd}^\rL+\bar{\rho}_{\rd}^\ast)
            \bigl(
                D\hat B^\ast
                +
                \delta(D\hat B)
            \bigr)
       \\
       &\quad
            \bigl[
                \del_t T^\rL + \mathrm{Z} \tilde v^\rL\cdot\nablaH T^\rL,
                \del_t q_{\rv}^\rL + \mathrm{Z} \tilde v^\rL\cdot\nablaH q_{\rv}^\rL,
                \del_t q_{\rc}^\rL + \mathrm{Z} \tilde v^\rL\cdot\nablaH q_{\rc}^\rL,
                \del_t q_{\rr}^\rL + \mathrm{Z} \tilde v^\rL\cdot\nablaH q_{\rr}^\rL
            \bigr]
        \Big](\cdot,\eta)\,\rd\eta.
    \end{aligned}
\end{equation}
We further split
\begin{equation}\label{eq:wL-moist-full}
    (\bar{\rho}_{\rd}\hat{B}w)^\rL=(\bar\rho_{\rd}^\rL+\bar\rho_{\rd}^\ast)\hat B^\rL w^\rL
    =
    J_1^{\mathrm m}+J_2^{\mathrm m},
\end{equation}
where \(J_1^{\mathrm m}\) collects the terms that are linear with respect to the solution and \(J_2^{\mathrm m}\) contains all higher order terms. More precisely,
\begin{equation}
    \label{eq:J1-moist}
    \begin{aligned}
        J_1^{\mathrm m}
        &:=
        -\bar\rho_{\rd}^\ast
        \int_0^z
        \big[
            D\hat B^\ast
            \bigl[
                \del_t T^\rL,\del_t q_{\rv}^\rL,\del_t q_{\rc}^\rL,\del_t q_{\rr}^\rL
            \bigr]
            +
            \hat B^\ast\,\divH v^\rL
        \big]
        (\cdot,\eta)\,\rd\eta
        \\
        &\quad
        +
        \bar\rho_{\rd}^\ast
        \Big(\int_0^z \hat B^\ast(\cdot,\eta)\,\rd\eta\Big)
        \Big(\int_0^1
            \big[\hat B^\ast\,\divH v^\rL\big](\cdot,\zeta)\,\rd\zeta
        \Big).
    \end{aligned}
\end{equation}
The remainder \(J_2^{\mathrm m}\) is then defined by
\[
J_2^{\mathrm m}
:=
\bigl((\bar\rho_{\rd}^\rL+\bar\rho_{\rd}^\ast)\hat B^\rL w^\rL\bigr)-J_1^{\mathrm m}.
\]
Finally, we write
\(
Q_{\rm}^\rL:=q_{\rv}^\rL+q_{\rc}^\rL+q_{\rr}^\rL,
\)
so that
\[
Q_{\rm}(t,\rX(t,\cdot),z)=1+Q_{\rm}^\rL(t,\cdot,z).
\]
To simplify notation, we introduce the transformed Laplacian \(\cL_1\) and the transformed horizontal \(\nablaH\divH\)-operator \(\cL_2\). For regular enough scalar functions \(f\), we define
\[
\cL_1 f
:=
\sum_{j,k,l=1}^2
\frac{\del^2 f}{\del y_k\del y_l} \mathrm Z_{k,j}\mathrm Z_{l,j}
+
\sum_{j,k,l=1}^2
\frac{\del f}{\del y_l}\frac{\del \mathrm Z_{l,j}}{\del y_k}\mathrm Z_{k,j}
+
\del_z^2 f,
\]
and hence
\begin{equation}\label{eq:lapla}
    (\cL_1-\Delta)f
    =
    \sum_{j,k,l=1}^2
    \frac{\del^2 f}{\del y_k\del y_l}(\mathrm Z_{k,j}-\delta_{k,j})\mathrm Z_{l,j}
    +
    \sum_{k,l=1}^2
    \frac{\del^2 f}{\del y_k\del y_l}(\mathrm Z_{l,k}-\delta_{l,k})
    +
    \sum_{j,k,l=1}^2
    \frac{\del f}{\del y_l}\frac{\del \mathrm Z_{l,j}}{\del y_k}\mathrm Z_{k,j}.
\end{equation}
For regular enough vector fields \(V=(V_1,V_2)\), we define
\[
\bigl((\cL_2-\nablaH\divH)V\bigr)_i
:=
\sum_{j,k,l=1}^2
\frac{\partial^2 V_j}{\partial y_k\partial y_l}(\mathrm Z_{k,j}-\delta_{k,j})\mathrm Z_{l,i}
+
\sum_{j,l=1}^2
\frac{\partial^2 V_j}{\partial y_j\partial y_l}(\mathrm Z_{l,i}-\delta_{l,i})
+
\sum_{j,k,l=1}^2
\frac{\partial V_j}{\partial y_k}\frac{\partial \mathrm Z_{k,j}}{\partial y_l}\mathrm Z_{l,i}.
\]
Moreover, we write
\[
V_{\rr}^\rL(t,\cdot,z)
:=
V_{\rr}(t,\rX(t,\cdot),z).
\]
With the above conventions, the transformed equations are almost in a suitable form. However, a delicate term which still contains a hidden linear contribution is the factor
\[
(1+q_{\rv})T\,\del_z w
\]
in the temperature equation. Therefore, we now isolate its linear part and incorporate it into the left-hand side of the transformed system.
Using hydrostatic balance
\(
\del_z p=-\bar\rho_{\rd}\hat B Q_{\rm}
\)
together with the equation of state
\(
p=\bar\rho_{\rd}\hat B(1+q_{\rv})T,
\)
we obtain the identity
\begin{equation}\label{eq:vertical pressure work identity}
    (1+q_{\rv})T\,\del_z w
    =
    \frac{(1+q_{\rv})T}{\bar\rho_{\rd}\hat B}\,\del_z(\bar\rho_{\rd}\hat B\,w)
    +
    \frac{Q_{\rm}+\del_z\bigl((1+q_{\rv})T\bigr)}{\bar\rho_{\rd}\hat B}\,
    (\bar\rho_{\rd}\hat B\,w).
\end{equation}
After transformation and using the decomposition
\(
(\bar\rho_{\rd}\hat B w)^\rL = J_1^{\mathrm m}+J_2^{\mathrm m},
\)
we infer
\begin{equation}\label{eq:vertical linear split}
    \begin{aligned}
        \bigl((1+q_{\rv})T\,\del_z w\bigr)^\rL
        &=
        \frac{(1+q_{\rv}^\rL)(T^\ast+T^\rL)}
        {(\bar\rho_{\rd}^\rL+\bar\rho_{\rd}^\ast)(\hat B^\ast+\delta\hat B)}
        \,\del_z(J_1^{\mathrm m}+J_2^{\mathrm m})
        +
        \frac{
            1+Q_\rm^\rL
            +
            \del_z\bigl((1+q_{\rv}^\rL)(T^\ast+T^\rL)\bigr)
        }
        {(\bar\rho_{\rd}^\rL+\bar\rho_{\rd}^\ast)(\hat B^\ast+\delta\hat B)}
        (J_1^{\mathrm m}+J_2^{\mathrm m})
    \end{aligned}
\end{equation}
Since the reference state is given by
\(
(\bar\rho_{\rd}^\ast,0,T^\ast,0,0,0),
\)
the corresponding frozen coefficients are
\[
\frac{T^\ast}{\bar\rho_{\rd}^\ast\hat B^\ast}
\ \text{ and }\
\frac{1}{\bar\rho_{\rd}^\ast\hat B^\ast}.
\]
Hence, by adding and subtracting these reference-state coefficients, we obtain the exact decomposition
\begin{equation}\label{eq:vertical linear split final}
    \begin{aligned}
        \bigl((1+q_{\rv})T\,\del_z w\bigr)^\rL
        &=
        \frac{T^\ast}{\bar\rho_{\rd}^\ast\hat B^\ast}\,\del_z J_1^{\mathrm m}
        +
        \frac{1}{\bar\rho_{\rd}^\ast\hat B^\ast}\,J_1^{\mathrm m}
        +
        \mathcal N,
    \end{aligned}
\end{equation}
where
\begin{equation}\label{eq:def-Nw}
    \begin{aligned}
        \mathcal N
        &:=
        \frac{
            \bar\rho_{\rd}^\ast\hat B^\ast T^\ast q_{\rv}^\rL
            +
            \bar\rho_{\rd}^\ast\hat B^\ast T^\rL
            +
            \bar\rho_{\rd}^\ast\hat B^\ast q_{\rv}^\rL T^\rL
            -
            T^\ast\bar\rho_{\rd}^\rL\hat B^\ast
            -
            T^\ast\bar\rho_{\rd}^\rL\delta\hat B
            -
            T^\ast\bar\rho_{\rd}^\ast\delta\hat B
        }
        {
            \bar\rho_{\rd}^\ast\hat B^\ast
            (\bar\rho_{\rd}^\rL+\bar\rho_{\rd}^\ast)
            (\hat B^\ast+\delta\hat B)
        }
        \,\del_z J_1^{\mathrm m}
        \\
        &\quad
        +
        \frac{
            \bar\rho_{\rd}^\ast\hat B^\ast Q_{\rm}^\rL
            +
            \bar\rho_{\rd}^\ast\hat B^\ast\del_z T^\rL
            +
            \bar\rho_{\rd}^\ast\hat B^\ast q_{\rv}^\rL\del_z T^\rL
            +
            \bar\rho_{\rd}^\ast\hat B^\ast T^\ast\del_z q_{\rv}^\rL
            +
            \bar\rho_{\rd}^\ast\hat B^\ast T^\rL\del_z q_{\rv}^\rL
            -
            \bar\rho_{\rd}^\rL\hat B^\ast
            -
            \bar\rho_{\rd}^\rL\delta\hat B
            -
            \bar\rho_{\rd}^\ast\delta\hat B
        }
        {
            \bar\rho_{\rd}^\ast\hat B^\ast
            (\bar\rho_{\rd}^\rL+\bar\rho_{\rd}^\ast)
            (\hat B^\ast+\delta\hat B)
        }
        \,J_1^{\mathrm m}
        \\
        &\quad
        +
        \frac{
            T^\ast
            +
            T^\rL
            +
            T^\ast q_{\rv}^\rL
            +
            q_{\rv}^\rL T^\rL
        }
        {
            (\bar\rho_{\rd}^\rL+\bar\rho_{\rd}^\ast)
            (\hat B^\ast+\delta\hat B)
        }
        \,\del_z J_2^{\mathrm m}
        \\
        &\quad
        +
        \frac{
            1
            +
            Q_{\rm}^\rL
            +
            \del_z T^\rL
            +
            q_{\rv}^\rL\del_z T^\rL
            +
            T^\ast\del_z q_{\rv}^\rL
            +
            T^\rL\del_z q_{\rv}^\rL
        }
        {
            (\bar\rho_{\rd}^\rL+\bar\rho_{\rd}^\ast)
            (\hat B^\ast+\delta\hat B)
        }
        \,J_2^{\mathrm m}.
    \end{aligned}
\end{equation}
Using the explicit formula for \(J^\rm_1\) and  $\int_0^z \hat{B}^\ast(\eta)\,d\eta
=
\frac{
1-e^{-z/T^\ast}
}{
1-e^{-1/T^\ast}
}$
, we infer
\begin{equation}\label{eq:vertical linear operator identity}
    \begin{aligned}
       & \quad \frac{1}{\bar\rho_{\rd}^\ast \hat B^\ast}J^\rm_1
        +
        \frac{T^\ast}{\bar\rho_{\rd}^\ast \hat B^\ast}\del_z J^\rm_1
       \\ &=
        -\frac{T^\ast}{\hat B^\ast}
        D\hat B^\ast
        \bigl[
            \del_t T^\rL,\del_t q_{\rv}^\rL,\del_t q_{\rc}^\rL,\del_t q_{\rr}^\rL
        \bigr]
        -
        \frac{1}{\hat B^\ast}
        \int_0^z
        D\hat B^\ast
        \bigl[
            \del_t T^\rL,\del_t q_{\rv}^\rL,\del_t q_{\rc}^\rL,\del_t q_{\rr}^\rL
        \bigr](\cdot,\eta)\,\rd\eta
        \\
        &\quad
        -
        \frac{
            \int_0^z \hat B^\ast(\cdot,\eta)\,\divH v^\rL(\cdot,\eta)\,\rd\eta
        }{
            \hat B^\ast
        }
        - T^\ast \divH v^\rL
        +
        T^\ast e^{z/T^\ast}
        \int_0^1
        \big[
            \hat B^\ast\,\divH v^\rL
        \big](\cdot,\eta)\,\rd\eta .
    \end{aligned}
\end{equation}
Using the explicit formula for \(J^\rm_1\), the identity \eqref{eq:vertical linear operator identity} can be decomposed by linearity of \(D\hat B^\ast\) as
\begin{equation}\label{eq:vertical linear operator split}
    \begin{aligned}
        \frac{1}{\bar\rho_{\rd}^\ast \hat B^\ast}J^\rm_1
        +
        \frac{T^\ast}{\bar\rho_{\rd}^\ast \hat B^\ast}\del_z J^\rm_1
        &=
        \mathcal G_T[\del_t T^\rL]
        +
        \sum_{j\in\{\rv,\rc,\rr\}} \mathcal G_j[\del_t q_j^\rL]
        -
        \frac{
            \int_0^z \hat B^\ast(\cdot,\eta)\,\divH v^\rL(\cdot,\eta)\,\rd\eta
        }{
            \hat B^\ast
        }
        -T^\ast \divH v^\rL
        \\
        &\quad
        +
        T^\ast e^{z/T^\ast}
        \int_0^1
        \big[
            \hat B^\ast\,\divH v^\rL
        \big](\cdot,\eta)\,\rd\eta ,
    \end{aligned}
\end{equation}
where the contribution corresponding to the temperature variable is given explicitly by
\begin{equation}\label{eq:def-GT}
    \mathcal G_T[h](z)
    :=
    h(z)
    -
    \frac{\exp\bigl((z-1)/T^\ast\bigr)}
    {T^\ast(1-\exp(-1/T^\ast))}
    \int_0^1 h(\eta)\,\rd\eta ,
\end{equation}
and, for \(j\in\{\rv,\rc,\rr\}\), the corresponding moisture contributions are defined by
\begin{equation*}
    \mathcal G_{j}[h](z)
    :=
    -\frac{T^\ast}{\hat B^\ast}D\hat B^\ast[0,h_j](z)
    -
    \frac{1}{\hat B^\ast}
    \int_0^z D\hat B^\ast[0,h_j](\eta)\,\rd\eta,
\end{equation*}
with $h_j$ being a three dimensional vector with only $0$'s and the entry $h$ in the $j$-th position.
Collecting the identities derived above, the transformed system takes the form
\begin{equation}\label{eq: full moist CPE Lagrange final}
    \left\{
    \begin{aligned}
        \del_t \bar\rho_{\rd}^\rL
        +\bar\rho_{\rd}^\ast
        \int_0^1
        \big[
            \hat B^\ast\,\divH v^\rL
        \big](\cdot,\eta)\,\rd\eta
        &= G_{\rho},
        &&\text{ in }\T^2_\tau,\\
          \del_t v^\rL
        -\frac{\rL_\rH}{\bar\rho_{\rd}^\ast\hat B^\ast} v^\rL
        +\frac{T^\ast}{\bar\rho_{\rd}^\ast}\nablaH \bar\rho_{\rd}^\rL
        +\nablaH T^\rL+  \frac{T^\ast}{\hat B^\ast}
    D\hat B^\ast
    \bigl[
        \nablaH T^\rL,0,0,0
    \bigr]
        &= G_v,
        &&\text{ in }\Omega_\tau,\\
        \mathcal H_T[\del_t T^\rL]
        -\Delta T^\rL
        -\frac{
            \int_0^z \hat B^\ast(\cdot,\eta)\,\divH v^\rL(\cdot,\eta)\,\rd\eta
        }{
            \hat B^\ast
        }
        +T^\ast e^{z/T^\ast}
        \int_0^1
        \big[
            \hat B^\ast\,\divH v^\rL
        \big](\cdot,\eta)\,\rd\eta
        &= G_T,
        &&\text{ in }\Omega_\tau,\\
        \del_t q_{\rv}^\rL-\Delta q_{\rv}^\rL
        &= G_{\rv},
        &&\text{ in }\Omega_\tau,\\        \del_t q_{\rc}^\rL-\Delta q_{\rc}^\rL
        &= G_{\rc},
        &&\text{ in }\Omega_\tau,\\
        \del_t q_{\rr}^\rL-\Delta q_{\rr}^\rL
        &= G_{\rr},
        &&\text{ in }\Omega_\tau,\\
        \bar\rho_{\rd}^\rL(0)=\bar\rho_{\rd,0}-\bar\rho_{\rd}^\ast,
        \
        v^\rL(0)=v_0,
        \
        T^\rL(0)=T_0-T^\ast,
        \
        q_j^\rL(0)=q_{j,0}\ \text{ for } j\in\{\rv,\rc,\rr\},
    \end{aligned}
    \right.
\end{equation}
where the operator $\mathcal{H}_T$ is given by
\begin{equation*}
    \mathcal{H}_T[h] :=  h + \mathcal{G}_T [h] =   2 h
    -
    \frac{\exp\bigl((z-1)/T^\ast\bigr)}
    {T^\ast(1-\exp(-1/T^\ast))}
    \int_0^1 h(\cdot,\eta)\rd\eta .
\end{equation*}
Here the right-hand sides are given by
\begin{equation*}
    \begin{aligned}
        G_{\rho}
        &:=
        -\bar\rho_{\rd}^\rL\,\divH b^\rL   -(\bar\rho_{\rd}^\rL+\bar\rho_{\rd}^\ast)\,
        \nablaH b^\rL:(\mathrm Z^\top-\mathrm I_2)
        \\
        &\quad
        -\bar\rho_{\rd}^\ast
        \int_0^1
        \Big[
            \delta\hat B\,\divH v^\rL
            +
            \bigl(
                D\hat B^\ast+\delta(D\hat B)
            \bigr)
            \bigl[
                \nablaH (T^\rL, q_{\rv}^\rL, q_{\rc}^\rL, q_{\rr}^\rL)
            \bigr]
            \cdot v^\rL
        \Big](\cdot,\eta)\rd\eta,
    \end{aligned}
\end{equation*}
and
\begin{equation*}
    \begin{aligned}
        G_v
        &:=
    -\Big(
        \frac{\bar\rho_{\rd}^\rL \,\delta\hat B}
        {\bar\rho_{\rd}^\ast \,\hat B^\ast}
        +
        \frac{\bar\rho_{\rd}^\rL}{\bar\rho_{\rd}^\ast}
        +
        \frac{\delta\hat B}{\hat B^\ast}
        +
        Q_{\rm}^\rL
        +
        \frac{\bar\rho_{\rd}^\rL}{\bar\rho_{\rd}^\ast}Q_{\rm}^\rL
        +
        \frac{\delta\hat B}{\hat B^\ast}Q_{\rm}^\rL
        +
        \frac{\bar\rho_{\rd}^\rL \,\delta\hat B}
        {\bar\rho_{\rd}^\ast \,\hat B^\ast}Q_{\rm}^\rL
    \Big)
    \dt v^\rL   -
\frac{(\bar \rho_\rd \hat{B} w)^\rL(1+ Q_\rm^\rL) 
        }{
            \bar\rho_{\rd}^\ast\hat B^\ast
        }   \del_z v^\rL    \\&\quad
        -
\frac{
            (\bar\rho_{\rd}^\rL+\bar\rho_{\rd}^\ast)\hat B^\rL(1+Q_{\rm}^\rL)
        }{
            \bar\rho_{\rd}^\ast\hat B^\ast
        }
            \mathrm Z\tilde v^\rL\cdot\nablaH v^\rL        
        +\frac{\mu}{\bar\rho_{\rd}^\ast\hat B^\ast}
        (\cL_1-\Delta)v^\rL  
         \\
        &\quad +
        \frac{\mu+\lambda}{\bar\rho_{\rd}^\ast\hat B^\ast}
        (\cL_2-\nablaH\divH)v^\rL + \frac{
            (\bar\rho_{\rd}^\rL+\bar\rho_{\rd}^\ast)\hat B^\rL q_{\rr}^\rL V_{\rr}^\rL
        }{
            \bar\rho_{\rd}^\ast\hat B^\ast
        }
        \,\del_z v^\rL 
            \\
        &\quad
        -\frac{1}{\bar\rho_{\rd}^\ast\hat B^\ast}
        \Big[
            (\mathrm Z^\top-\mathrm I_2)\nablaH
            \big(
                (\bar\rho_{\rd}^\rL+\bar\rho_{\rd}^\ast)
                \hat B^\rL
                (1+q_{\rv}^\rL)
                (T^\ast+T^\rL)
            \big)  -T^\ast \nablaH q_{\rv}^\rL
        -\frac{T^\ast}{\hat B^\ast}
        D\hat B^\ast
        \bigl[
            \nablaH (0, q_{\rv}^\rL, q_{\rc}^\rL, q_{\rr}^\rL)
        \bigr]
            \\
            &\quad
            +
            \big[
                \hat B^\ast T^\rL
                +
                T^\ast\delta\hat B
                +
                \delta\hat B T^\rL
                +
                \hat B^\ast T^\ast q_{\rv}^\rL
                +
                T^\ast\delta\hat B q_{\rv}^\rL
                +
                \hat B^\ast q_{\rv}^\rL T^\rL
                +
                \delta\hat B q_{\rv}^\rL T^\rL
            \big]
            \nablaH\bar\rho_{\rd}^\rL 
            \\
            &\quad
            +
            \big[
                \hat B^\ast\bar\rho_{\rd}^\rL
                +
                \bar\rho_{\rd}^\rL\delta\hat B
                +
                \bar\rho_{\rd}^\ast\delta\hat B
                +
                \bar\rho_{\rd}^\ast\hat B^\ast q_{\rv}^\rL
                +
                \bar\rho_{\rd}^\rL\hat B^\ast q_{\rv}^\rL
                +
                \bar\rho_{\rd}^\ast\delta\hat B q_{\rv}^\rL
                +
                \bar\rho_{\rd}^\rL\delta\hat B q_{\rv}^\rL
            \big]
            \nablaH T^\rL
            \\
            &\quad
            +
            \big[
                T^\ast\hat B^\ast\bar\rho_{\rd}^\rL
                +
                \hat B^\ast\bar\rho_{\rd}^\ast T^\rL
                +
                \hat B^\ast\bar\rho_{\rd}^\rL T^\rL
                +
                T^\ast\bar\rho_{\rd}^\ast\delta\hat B
                +
                T^\ast\bar\rho_{\rd}^\rL\delta\hat B
                +
                \bar\rho_{\rd}^\ast\delta\hat B T^\rL
                +
                \bar\rho_{\rd}^\rL\delta\hat B T^\rL
            \big]
            \nablaH q_{\rv}^\rL
            \\
            &\quad
            +
            \big[
                T^\ast\bar\rho_{\rd}^\rL
                +
                \bar\rho_{\rd}^\ast T^\rL
                +
                \bar\rho_{\rd}^\rL T^\rL
                +
                \bar\rho_{\rd}^\ast T^\ast q_{\rv}^\rL
                +
                T^\ast\bar\rho_{\rd}^\rL q_{\rv}^\rL
                +
                \bar\rho_{\rd}^\ast q_{\rv}^\rL T^\rL
                +
                \bar\rho_{\rd}^\rL q_{\rv}^\rL T^\rL
            \big]
            D\hat B^\ast
            \big[
                \nablaH(T^\rL,q_{\rv}^\rL,q_{\rc}^\rL,q_{\rr}^\rL)
            \big]
            \\
            &\quad
            +
            \big[
                \bar\rho_{\rd}^\ast T^\ast
                +
                T^\ast\bar\rho_{\rd}^\rL
                +
                \bar\rho_{\rd}^\ast T^\rL
                +
                \bar\rho_{\rd}^\rL T^\rL
                +
                \bar\rho_{\rd}^\ast T^\ast q_{\rv}^\rL
                +
                T^\ast\bar\rho_{\rd}^\rL q_{\rv}^\rL
                +
                \bar\rho_{\rd}^\ast q_{\rv}^\rL T^\rL
                +
                \bar\rho_{\rd}^\rL q_{\rv}^\rL T^\rL
            \big]
            \delta(D\hat B)
            \big[
                \nablaH(T^\rL,q_{\rv}^\rL,q_{\rc}^\rL,q_{\rr}^\rL)
            \big]
        \Big]
    \end{aligned}
\end{equation*}
as well as 
\begin{equation*}
    \begin{aligned}
        G_T
        &:=
        -Q_{\rm}^\rL\,\del_t T^\rL
        -(1+Q_{\rm}^\rL)
        \Big[
            \mathrm Z\tilde v^\rL\cdot\nablaH T^\rL
            +
            w^\rL\del_z T^\rL
        \Big]
        +(\cL_1-\Delta)T^\rL
        +
        q_{\rr}^\rL V_{\rr}^\rL\,\del_z(T^\ast+T^\rL)         -\mathcal N
        \\
        &\quad
        -
            (1+q_{\rv}^\rL)(T^\ast+T^\rL)\,\nablaH v^\rL:(\mathrm Z^\top - \mathrm{I}_2) - q_\rv^\rL ( T^\ast + T^\rL) \divH v^\rL -T^\rL \divH v^\rL \\ &\quad
        +\frac{T^\ast   D\hat B^\ast
        \bigl[
            0,\del_t (q_{\rv}^\rL, q_{\rc}^\rL, q_{\rr}^\rL)
        \bigr]}{\hat B^\ast}
        +\frac{1}{\hat B^\ast}
        \int_0^z   D\hat B^\ast
        \bigl[
            0,\del_t (q_{\rv}^\rL, q_{\rc}^\rL, q_{\rr}^\rL)
        \bigr]
       (\cdot,\eta)\rd\eta
        \\
        &\quad
        -
            (T^\ast+T^\rL)\frac{1+q_{\rv}^\rL}{1+Q_{\rm}^\rL}(q_{\rvs}^\rL-q_{\rv}^\rL)^+q_{\rr}^\rL
            +(q_{\rv}^\rL-q_{\rvs}^\rL)q_{\rc}^\rL
            +(q_{\rv}^\rL-q_{\rvs}^\rL)^+,
    \end{aligned}
\end{equation*}
where $\mathcal N$ is given by \eqref{eq:def-Nw} and
\begin{equation*}
    \begin{aligned}
        G_{\rv}
        &:=
        (\cL_1-\Delta)q_{\rv}^\rL
        -
        \mathrm Z\tilde v^\rL\cdot\nablaH q_{\rv}^\rL
        -
        w^\rL\del_z q_{\rv}^\rL
        \\
        &\quad
        +
        (T^\ast+T^\rL)\frac{1+q_{\rv}^\rL}{1+Q_{\rm}^\rL}(q_{\rvs}^\rL-q_{\rv}^\rL)^+q_{\rr}^\rL
        -(q_{\rv}^\rL-q_{\rvs}^\rL)q_{\rc}^\rL
        -(q_{\rv}^\rL-q_{\rvs}^\rL)^+,
        \\
        G_{\rc}
        &:=
        (\cL_1-\Delta)q_{\rc}^\rL
        -
        \mathrm Z\tilde v^\rL\cdot\nablaH q_{\rc}^\rL
        -
        w^\rL\del_z q_{\rc}^\rL
        \\
        &\quad
        +
        (q_{\rv}^\rL-q_{\rvs}^\rL)q_{\rc}^\rL
        +(q_{\rv}^\rL-q_{\rvs}^\rL)^+
        -(q_{\rc}^\rL-1)^+
        -q_{\rc}^\rL q_{\rr}^\rL,
        \\
        G_{\rr}
        &:=
        (\cL_1-\Delta)q_{\rr}^\rL
        -
        \mathrm Z\tilde v^\rL\cdot\nablaH q_{\rr}^\rL
        -
        w^\rL\del_z q_{\rr}^\rL
        +
        \frac{1}{\hat B^\rL}
        \del_z\bigl(\hat B^\rL q_{\rr}^\rL V_{\rr}^\rL\bigr)
        \\
        &\quad
        +
        (q_{\rc}^\rL-1)^+
        +q_{\rc}^\rL q_{\rr}^\rL
        -
        (T^\ast+T^\rL)\frac{1+q_{\rv}^\rL}{1+Q_{\rm}^\rL}(q_{\rvs}^\rL-q_{\rv}^\rL)^+q_{\rr}^\rL .
    \end{aligned}
\end{equation*}
Finally, the boundary conditions are unchanged by the transformation, since the latter acts only in the horizontal directions.

\section{Linear Theory}\label{sec: local}
\noindent
In this section, we study the fully linearized system associated with \eqref{eq: full moist CPE Lagrange final}. At this stage, we regard the right-hand sides as given forcing terms and focus on the corresponding linear problem. Since the averaged density is now treated again as an independent unknown, the linearized system consists of six evolution equations for the variables
\[
(\xi,V,T,Q_{\rv},Q_{\rc},Q_{\rr}),
\]
where \(\xi\) denotes the linearized counterpart of \(\bar\rho_{\rd}^\rL\).
To formulate the system in a concise way, we introduce some notation that will be used throughout this section.
Define the functions \(\alpha\) and \(\beta\) and the integral operators \(\mathcal I_z\) by
\begin{equation*}
    \alpha(z):=\frac{1}{\bar\rho_{\rd}^\ast\hat B^\ast(z)},
    \quad
    \beta(z):=\frac{\exp((z-1)/T^\ast)}{T^\ast(1-\exp(-1/T^\ast))} \ \text{ and } \  \mathcal I_z f:=\int_0^z [\hat B^\ast f](\cdot,\eta)\,\rd\eta \ \text{ for } z \in [0,1],
\end{equation*}
where \(f\) is regular enough function.
Next, we set
\begin{equation*}
    \mathcal A T
    :=
    \nablaH T
    +{T^\ast}   \bar{\rho}^\ast_\rd  \, \alpha\,
    D\hat B^\ast
    \bigl[
        \nablaH (T,0,0,0)
    \bigr].
\end{equation*}

With this notation, we consider the fully linearized problem corresponding to \eqref{eq: full moist CPE Lagrange final}, given by
\begin{equation}
    \left\{
    \begin{aligned}
        \del_t \xi
        +\bar\rho_{\rd}^\ast \mathcal I_1(\divH V)
        &= g_1,
        &&\text{ in }\T^2_\tau,\\
        \del_t V
        -\alpha\, \rL_\rH V
        +\frac{T^\ast}{\bar\rho_{\rd}^\ast}\nablaH \xi
        +\mathcal A T
        &= g_2,
        &&\text{ in }\Omega_\tau,\\
        \mathcal H_T[\del_t T]
        -\Delta T
        -\bar{\rho}^\ast_\rd  \, \alpha\,  \mathcal I_z(\divH V)
        +T^\ast e^{z/T^\ast}\mathcal I_1(\divH V)
        &= g_3,
        &&\text{ in }\Omega_\tau,\\
        \del_t Q_{\rv}-\Delta Q_{\rv}
        &= g_4,
        &&\text{ in }\Omega_\tau,\\
        \del_t Q_{\rc}-\Delta Q_{\rc}
        &= g_5,
        &&\text{ in }\Omega_\tau,\\
        \del_t Q_{\rr}-\Delta Q_{\rr}
        &= g_6,
        &&\text{ in }\Omega_\tau,\\
        \xi(0)=\xi_0,\quad
        V(0)=V_0,\quad
        T(0)=T_0,\quad
        Q_j(0)=Q_{j,0} \ \text{ for } j\in\{\rv,\rc,\rr\}.
    \end{aligned}
    \right.
    \label{eq: full moist CPE Lagrange linear}
\end{equation}
with prescribed forcing terms \(g_1,\dots,g_6\). The system is supplemented with the boundary conditions
\begin{equation}\label{eq: bc linear moist}
    \partial_z V|_{z=0,1}=
    \partial_z T|_{z=0,1}=
    \partial_z Q_j|_{z=0,1}=0  \ \text{ for }
     j\in\{\rv,\rc,\rr\}.
\end{equation}
The next part of this section is devoted to proving a well-posedness result concerning the linearized system \eqref{eq: full moist CPE Lagrange linear} subject to the boundary conditions \eqref{eq: bc linear moist}. In view of the regularity class we aim at, we introduce the corresponding forcing and solution spaces directly.
For the density component, we define
\begin{equation*}
    \mathbb E^\rho_0(0,\tau)
    :=
    \rL^2(0,\tau;\rH^2(\T^2))
    \cap
    \rH^1(0,\tau;\rL^2(\T^2)).
\end{equation*}
For the remaining components, we set
\begin{equation*}
    \mathbb E^v_0(0,\tau)
    :=
    \rL^2(0,\tau;\rH^1(\Omega;\R^2))\cap \rH^1(0,\tau;\rL^2(\Omega;\R^2))
    \ \text{ and } \
    \mathbb E^h_0(0,\tau)
    :=
    \rL^2(0,\tau;\rH^1(\Omega))\cap \rH^1(0,\tau;\rL^2(\Omega)).
\end{equation*}
Accordingly, the full forcing space is given by
\begin{equation*}
    \mathbb E_0(0,\tau)
    :=
    \mathbb E^\rho_0(0,\tau)
    \times
    \mathbb E^v_0(0,\tau)
    \times
    \mathbb E^h_0(0,\tau)^4.
\end{equation*}
We further define
\begin{equation*}
    \mathbb E^\rho_1(0,\tau)
    :=
    \rH^1(0,\tau;\rH^2(\T^2))
    \cap
    \rH^2(0,\tau;\rL^2(\T^2))
\end{equation*}
and
for the remaining components, we set
\begin{equation*}
    \mathbb E^v_1(0,\tau)
    :=
    \rH^1(0,\tau;\rH^2(\Omega;\R^2))
    \cap
    \rL^2(0,\tau;\rH^3(\Omega;\R^2))
    \cap
    \rH^2(0,\tau;\rL^2(\Omega;\R^2))
\end{equation*}
as well as 
\begin{equation*}
    \mathbb E^h_1(0,\tau)
    :=
    \rH^1(0,\tau;\rH^2(\Omega))
    \cap
    \rL^2(0,\tau;\rH^3(\Omega))
    \cap
    \rH^2(0,\tau;\rL^2(\Omega)).
\end{equation*}
Accordingly, the full solution space is given by
\begin{equation*}
    \mathbb E_1(0,\tau)
    :=
    \mathbb E^\rho_1(0,\tau)
    \times
    \mathbb E^v_1(0,\tau)
    \times
    \mathbb E^h_1(0,\tau)^4.
\end{equation*}
We are now in a position to state the main result of this section.
\begin{lem}
    \label{lem: linear max reg}
    Let \(\tau>0\), \(\bar\rho_{\rd}^\ast>0\), and \(T^\ast>0\). We assume that the initial data satisfy
\begin{equation*}
   ( \xi_0, V_0, T_0, Q_{j,0})\in \rH^2(\T^2) \times 
 \rH^2(\Omega;\R^2) \times 
 \rH^2(\Omega)^4 \ \text{ for }
  j\in\{\rv,\rc,\rr\},
\end{equation*}
together with the boundary compatibility conditions
\begin{equation*}
    \del_z V_0|_{z=0,1}=
    \del_z T_0|_{z=0,1}=
    \del_z Q_{j,0}|_{z=0,1}=0 \ \text{ for }
 j\in\{\rv,\rc,\rr\}.
\end{equation*}
Moreover, we assume that the first-order compatibility quantities
\(
    (\partial_t \xi,\partial_t V,\partial_t T,\partial_t Q_j)|_{t=0}
\)
computed from the respective right-hand sides of
\eqref{eq: full moist CPE Lagrange linear} by inserting the initial data and
the forcing terms at time \(t=0\), satisfy
\begin{equation*}
   (\partial_t \xi,\partial_t V,\partial_t T,\partial_t Q_j)|_{t=0}
   \in
   \rH^1(\T^2)
   \times
   \rH^1(\Omega;\R^2)
   \times
   \rH^1(\Omega)^4
   \quad
   \text{for } j\in\{\rv,\rc,\rr\}.
\end{equation*}
    Concerning the forcing terms $ (g_1,g_2,g_3,g_4,g_5,g_6)$, we suppose that 
    \[
        (g_1,g_2,g_3,g_4,g_5,g_6)\in \mathbb E_0(0,\tau).
    \]
    Then there exists a unique, global and strong solution
    \begin{equation*}       (\xi,V,T,Q_{\rv},Q_{\rc},Q_{\rr})\in \mathbb E_1(0,\tau)
    \end{equation*}
    to \eqref{eq: full moist CPE Lagrange linear}--\eqref{eq: bc linear moist}. Moreover, there is a constant \(C=C(\tau,\bar\rho_{\rd}^\ast,T^\ast,\Omega)>0\) such that
    \begin{equation*}
        \begin{aligned}
        &\quad\|(\xi,V,T,Q_{\rv},Q_{\rc},Q_{\rr})\|_{\mathbb E_1(0,\tau)}
      \\  &\le
        C\Big(
        \|(g_1,g_2,g_3,g_4,g_5,g_6)\|_{\mathbb E_0(0,\tau)}
 +
        \|\xi_0\|_{\rH^2(\T^2)}
        +
        \|V_0\|_{\rH^2(\Omega)}
        +
        \|T_0\|_{\rH^2(\Omega)}
        +
        \sum_{j\in\{\rv,\rc,\rr\}}
        \|Q_{j,0}\|_{\rH^2(\Omega)}
   \\&     \quad +
        \|(\dt \xi)|_{t=0}\|_{\rH^1(\T^2)}
        +
        \|(\dt V)|_{t=0}\|_{\rH^1(\Omega)}
        +
        \|(\dt T)|_{t=0}\|_{\rH^1(\Omega)}
        +
        \sum_{j\in\{\rv,\rc,\rr\}}
        \|(\dt Q_j)|_{t=0}\|_{\rH^1(\Omega)}
        \Big).
        \end{aligned}
    \end{equation*}
\end{lem}
\begin{proof}
We divide the proof into three steps.
First we isolate the nonlocal time-derivative operator appearing
in the temperature equation. More precisely, we show that the operator
\(\mathcal H_T\) is boundedly invertible on the relevant Sobolev scale and that,
after applying \(\mathcal H_T^{-1}\), the resulting temperature operator is an
\(\mathcal R\)-sectorial perturbation of the Neumann Laplacian. This is the only
point where the special nonlocal structure of the temperature equation enters.
In the second step we introduce the full operator matrix associated with
\eqref{eq: full moist CPE Lagrange linear} and prove maximal \(\rL^2\)-regularity.
Lastly, we derive the additional regularity stated in the lemma, 
assuming that the right-hand sides possess the corresponding time regularity
and that the first-order compatibility conditions are satisfied.
\begin{step}{\emph{The nonlocal temperature operator $\mathcal{H}_T$}}\label{step1}\mbox{}\\
We start with the nonlocal operator occurring in the temperature equation.
Define
\[
    (\mathcal P f)(\cdot,z)
    :=
    \beta(z)\int_0^1 f(\cdot,\eta)\,\rd\eta,
    \qquad
    f\in \rH^1(\Omega),
\]
where
\[
    \beta(z)
    =
    \frac{\exp((z-1)/T^\ast)}
    {T^\ast(1-\exp(-1/T^\ast))}.
\]
Then \(\beta\in \rC^\infty([0,1])\) and
\[
    \int_0^1\beta(z)\,\rd z=1.
\]
Thus \(\mathcal P\) is a projection. Indeed,
\[
\begin{aligned}
    (\mathcal P^2 f)(\cdot,z)
    =
    \beta(z)
    \int_0^1
    (\mathcal P f)(\cdot,\eta)\,\rd\eta    =
    \beta(z)
    \int_0^1
    \beta(\eta)
    \left(
        \int_0^1 f(\cdot,\zeta)\,\rd\zeta
    \right)
    \rd\eta                                  =
    \beta(z)
    \int_0^1 f(\cdot,\zeta)\,\rd\zeta
    =
    (\mathcal P f)(\cdot,z).
\end{aligned}
\]
Moreover, \(\mathcal P\in\mathcal L(\rH^1(\Omega))\). This follows from the
smoothness of \(\beta\) and the boundedness of vertical integration on
\(\rH^1(\Omega)\). 
With this notation, the nonlocal time operator is
\[
    \mathcal H_T
    =
    2\,\mathrm{Id}-\mathcal P.
\]
Since \(\mathcal P^2=\mathcal P\), we have
\[
    (2\,\mathrm{Id}-\mathcal P)^{-1}
    =
    \frac12(\mathrm{Id}+\mathcal P).
\]
Hence
\begin{equation}\label{eq: inverse HT moist}
    \mathcal H_T^{-1}
    =
    \frac12(\mathrm{Id}+\mathcal P).
\end{equation}
In particular, \(\mathcal H_T^{-1}\) is bounded on \(\rH^1(\Omega)\).
Furthermore, the spectrum of \(\mathcal H_T^{-1}\) is contained in
\[
    \left\{\frac12,1\right\}.
\]
Indeed, \(\mathcal H_T^{-1}\) acts as the identity on the range of
\(\mathcal P\) and as multiplication by \(1/2\) on the kernel of \(\mathcal P\).
Consequently, for every bounded holomorphic function \(m\) defined on a domain
containing \(\{1/2,1\}\), the holomorphic functional calculus gives
\[
    m(\mathcal H_T^{-1})
    =
    m(1)\mathcal P
    +
    m(1/2)(\mathrm{Id}-\mathcal P).
\]
Thus \(\mathcal H_T^{-1}\) admits a bounded \(\Hinfty\)-calculus on
\(\rH^1(\Omega)\) with angle \(0\). In particular,
\(\mathcal H_T^{-1}\) is \(\mathcal R\)-sectorial on \(\rH^1(\Omega)\) with
\[
    \Phi^{\mathcal R}_{\mathcal H_T^{-1}}=0.
\]
We now apply \(\mathcal H_T^{-1}\) to the temperature equation in
\eqref{eq: full moist CPE Lagrange linear}. After adding and subtracting a
positive shift \(\omega>0\), we obtain
\begin{equation}\label{eq: Temp with HT moist}
\begin{aligned}
    \partial_t T
    +
    \mathcal H_T^{-1}(-\Delta+\omega)T
    -
    \omega\mathcal H_T^{-1}T
    =
    \mathcal H_T^{-1}
    \big (
        g_3
        +
       \bar{\rho}^\ast_\rd  \, \alpha\,\mathcal I_z(\divH V)
        -
        T^\ast e^{z/T^\ast}\mathcal I_1(\divH V)
    \big ).
\end{aligned}
\end{equation}
Our goal is to show that, after a further shift if necessary,
\[
    \mathcal H_T^{-1}(-\Delta+\omega)
\]
is \(\mathcal R\)-sectorial on \(\rH^1(\Omega)\) with angle strictly smaller
than \(\pi/2\). We view it as a product of the two, in general non-commuting,
operators
\[
    A:=\mathcal H_T^{-1},
    \qquad
    B:=-\Delta+\omega.
\]
Here
\[
    B\colon \rH^1(\Omega)\to \rH^1(\Omega),
    \qquad
    \rD(B)=\rH^3_{\rN}(\Omega),
\]
is the Neumann Laplacian shifted by \(\omega\). It is well known that, for
\(\omega>0\) sufficiently large, \(B\) is \(\mathcal R\)-sectorial on
\(\rH^1(\Omega)\) with angle strictly smaller than \(\pi/2\). This follows from
the standard Neumann realization of the Laplacian on the layer, together with
localization and the periodic extension argument in the horizontal variables, as explained on p.~1082 in \cite{HK:16}.
Define
\[
    S:=AB=\mathcal H_T^{-1}(-\Delta+\omega),
    \qquad
    \rD(S):=\rH^3_{\rN}(\Omega).
\]
To verify \(\mathcal R\)-sectoriality of \(S\), we consider the commutator expression
\[
    Z_{A,B}(\lambda,\mu)
    :=
    \big[
        A(\mu+B)^{-1}
        -
        (\mu+B)^{-1}A
    \big](\lambda+A)^{-1}.
\]
We need an estimate of the form
\begin{equation}\label{eq: commutator moist temp}
      \| Z_{A,B}(\lambda,\mu) \|
      \le
      \frac{C}
      {(1+|\lambda|)^{1-a}|\mu|^{1+b}},
      \qquad
      (\lambda,\mu)\in
      \Sigma_\pi\times\Sigma_{\pi-\Phi_B^{\mathcal R}},
\end{equation}
with \(a,b\ge0\) and \(a+b<1\). Since \(A=\mathcal H_T^{-1}\) is bounded and
boundedly invertible on \(\rH^1(\Omega)\), and since
\(\sigma(A)\subset\{1/2,1\}\), there exists \(C_A>0\) such that
\[
    \|(\lambda+A)^{-1}\|
    \le
    \frac{C_A}{1+|\lambda|}
    \qquad
    \text{for all } \lambda\in\Sigma_\pi.
\]
Moreover, the \(\mathcal R\)-sectoriality of \(B\) yields
\[
    \|(\mu+B)^{-1}\|
    \le
    \frac{C_B}{|\mu|}
    \qquad
    \text{for all }
    \mu\in\Sigma_{\pi-\Phi_B^{\mathcal R}}.
\]
Writing \(R_\mu=(\mu+B)^{-1}\), we obtain the crude commutator estimate
\[
    \|AR_\mu-R_\mu A\|
    \le
    2\|A\|\,\|R_\mu\|
    \le
    \frac{C}{|\mu|}.
\]
Together with the bound for \((\lambda+A)^{-1}\), this gives
\[
    \|Z_{A,B}(\lambda,\mu)\|
    \le
    \frac{C}{(1+|\lambda|)|\mu|}.
\]
Thus \eqref{eq: commutator moist temp} holds with \(a=b=0\). Therefore
\cite[Theorem~1.1]{Weber:98} applies and yields sectoriality of \(S\), up to a
shift. Since \(\rH^1(\Omega)\) is a Hilbert space, sectoriality and
\(\mathcal R\)-sectoriality coincide; see, for instance,
\cite[Chapter~4]{PS:16}.
Finally, the lower-order term
\(
    -\omega\mathcal H_T^{-1}T
\)
in \eqref{eq: Temp with HT moist} is bounded on \(\rH^1(\Omega)\). Hence it is
handled by the standard perturbation theorem for \(\mathcal R\)-sectorial
operators. Consequently, there exists \(\omega_1\ge0\) such that the shifted
temperature operator
\[
    \mathcal H_T^{-1}(-\Delta+\omega)
    -
    \omega\mathcal H_T^{-1}
    +
    \omega_1
\]
is \(\mathcal R\)-sectorial on \(\rH^1(\Omega)\) with angle strictly smaller
than \(\pi/2\). This proves the maximal-regularity property for the principal
temperature block.
\end{step}
\begin{step}{\emph{The full operator matrix}}\label{step2}\mbox{}\\
We now consider the full abstract Cauchy problem associated with
\eqref{eq: full moist CPE Lagrange linear}. In view of \autoref{step1}, we apply
\(\mathcal H_T^{-1}\) to the temperature equation and write the system in the
form
\[
    \partial_t \mathbf u+\mathbb A\mathbf u=\mathbf g \ \text{ with } \   \mathbf u
    :=
    (\xi,V,T,Q_{\rv},Q_{\rc},Q_{\rr})^\top.
\]
Define the ground space and the domain by
\[
    \rX_0
    :=
    \rH^2(\T^2)
    \times
    \rH^1(\Omega;\R^2)
    \times
    \rH^1(\Omega)^4 \ \text{ and } \  \rX_1
    :=
    \rH^2(\T^2)
    \times
    \rH^3_{\rN}(\Omega;\R^2)
    \times
    \rH^3_{\rN}(\Omega)^4.
\]
Moreover, after applying \(\mathcal H_T^{-1}\), the right-hand side is understood
as
\[
    \mathbf g
    :=
    \bigl(
        g_1,
        g_2,
        \mathcal H_T^{-1}g_3,
        g_4,
        g_5,
        g_6
    \bigr)^\top.
\]
Since \(\mathcal H_T^{-1}\in\mathcal L(\rH^1(\Omega))\), this modification does
not change the forcing space.
We define
\[
    \mathbb A\colon \rX_1\subset\rX_0\to\rX_0
\]
by the operator matrix
\[
\mathbb A
=
\begin{pmatrix}
0
&
\bar\rho_{\rd}^\ast\,\mathcal I_1(\divH)
&
0
&
0
&
0
&
0
\\ 
\frac{T^\ast}{\bar\rho_{\rd}^\ast}\nablaH
&
-\alpha\, \rL_\rH
&
\mathcal A_T
&
0
&
0
&
0
\\ 
0
&
\mathcal C
&
\mathcal H_T^{-1}(-\Delta+\omega)
    -
    \omega\mathcal H_T^{-1}
&
0
&
0
&
0
\\ 
0
&
0
&
0
&
-\Delta
&
0
&
0
\\ 
0
&
0
&
0
&
0
&
-\Delta
&
0
\\ 
0
&
0
&
0
&
0
&
0
&
-\Delta
\end{pmatrix}.
\]
Here,
the operators
\(\mathcal A_T,\mathcal A_{\rv},\mathcal A_{\rc},\mathcal A_{\rr}\) are defined
through
\[
\begin{aligned}
    \mathcal A_T T
    +
    \mathcal A_{\rv}Q_{\rv}
    +
    \mathcal A_{\rc}Q_{\rc}
    +
    \mathcal A_{\rr}Q_{\rr}:=
    \mathcal{A}(T,Q_\rv,Q_\rc,Q_\rr)
\end{aligned}
\]
and the coupling from the velocity equation into the temperature equation is
\[
    \mathcal C V
    :=
    \mathcal H_T^{-1}
    \big(
        -  \bar{\rho}^\ast_\rd  \, \alpha\,\mathcal I_z(\divH V)
        +
        T^\ast e^{z/T^\ast}\mathcal I_1(\divH V)
    \big).
\]
Now split the full operator matrix $\mathbb A$ into two parts
\(
    \mathbb A=\mathbb A_0+\mathbb B,
\)
where the principal part is
\[
\mathbb A_0
=
\begin{pmatrix}
0
&
\bar\rho_{\rd}^\ast\,\mathcal I_1(\divH)
&
0
&
0
&
0
&
0
\\ 
0
&
-\mu\alpha\Delta-(\mu+\lambda)\alpha\nablaH\divH
&
0
&
0
&
0
&
0
\\ 
0
&
0
&
\mathcal H_T^{-1}(-\Delta+\omega)
    -
    \omega\mathcal H_T^{-1}
&
0
&
0
&
0
\\ 
0
&
0
&
0
&
-\Delta
&
0
&
0
\\ 
0
&
0
&
0
&
0
&
-\Delta
&
0
\\
0
&
0
&
0
&
0
&
0
&
-\Delta
\end{pmatrix},
\]
and the perturbation is
\[
\mathbb B
=
\begin{pmatrix}
0
&
0
&
0
&
0
&
0
&
0
\\ 
\frac{T^\ast}{\bar\rho_{\rd}^\ast}\nablaH
&
0
&
\mathcal A_T
&
0
&
0
&
0
\\ 
0
&
\mathcal C
&
0
&
0
&
0
&
0
\\ 
0
&
0
&
0
&
0
&
0
&
0
\\ 
0
&
0
&
0
&
0
&
0
&
0
\\ 
0
&
0
&
0
&
0
&
0
&
0
\end{pmatrix}.
\]
First, we discuss the principal part. The operator
\[
  -\alpha\, \rL_\rH=  -\mu\alpha\Delta-(\mu+\lambda)\alpha\nablaH\divH,
    \quad
    \rD(-\alpha\, \rL_\rH)=\rH^3_{\rN}(\Omega;\R^2),
\]
is normally elliptic on \(\rH^1(\Omega;\R^2)\) and since $\alpha(z)$
is smooth and bounded away from zero, see also the discussion in \cite[Section 3]{HIRZ:25} and \cite{TZ:26}. Hence, there exists a shift \(\omega_V\ge0\) such that
\(
    \alpha\, \rL_\rH+\omega_V
\)
is \(\mathcal R\)-sectorial on \(\rH^1(\Omega;\R^2)\) with angle strictly less
than \(\pi/2\).
By \autoref{step1}, there exists \(\omega_T\ge0\) such that
\(
    \mathcal H_T^{-1}(-\Delta+\omega)
    -
    \omega\mathcal H_T^{-1}+\omega_T
\)
is \(\mathcal R\)-sectorial on \(\rH^1(\Omega)\) with angle strictly less than
\(\pi/2\). Moreover, the shifted Neumann Laplacian $-\Delta + \omega$ with domain $ \rD(-\Delta + \omega)=\rH^3_{\rN}(\Omega)$
is \(\mathcal R\)-sectorial on \(\rH^1(\Omega)\)
with angle strictly less than \(\pi/2\), see \cite{DHP:03}.
Concerning the upper triangular density coupling
\[
    K V
    :=
    \bar\rho_{\rd}^\ast\mathcal I_1(\divH V).
\]
For \(V\in\rH^3_{\rN}(\Omega;\R^2)\), Jensen's inequality and the smoothness of
\(\hat B^\ast\) give
\[
\begin{aligned}
    \|KV\|_{\rH^2(\T^2)}
    &\le
    C
    \left\|
        \int_0^1
        \hat B^\ast(\eta)\divH V(\cdot,\eta)\,\rd\eta
    \right\|_{\rH^2(\T^2)}
    \\
    &\le
    C\|V\|_{\rH^3(\Omega)}.
\end{aligned}
\]
Hence \(K\) is bounded from \(\rH^3_{\rN}(\Omega;\R^2)\) into
\(\rH^2(\T^2)\). Therefore the first two diagonal blocks of
\(\mathbb A_0\) form an upper triangular operator matrix. Since the diagonal
entries are \(\mathcal R\)-sectorial after a shift and since the off-diagonal
entry \(K\) is bounded with respect to the graph norm of the velocity operator,
the explicit resolvent formula for triangular operator matrices yields that,
after increasing the shift if necessary,
\[
    \mathbb A_0+\omega_0
\]
is \(\mathcal R\)-sectorial on \(\rX_0\) with angle strictly smaller than
\(\pi/2\).
We now estimate the perturbation \(\mathbb B\). Let
\[
    \mathbf u
    =
    (\xi,V,T,Q_{\rv},Q_{\rc},Q_{\rr})^\top
    \in \rX_1.
\]
First, since \(\xi\) is independent of \(z\),
\[
    \left\|
        \frac{T^\ast}{\bar\rho_{\rd}^\ast}\nablaH\xi
    \right\|_{\rH^1(\Omega)}
    \le
    C\|\xi\|_{\rH^2(\T^2)}.
\]
Next, by the definition of
\(
    \mathcal AT
\)
and the boundedness of \(D\hat B^\ast\), we have
\[
\begin{aligned}
    \left\|
        \mathcal A T
    \right\|_{\rH^1(\Omega)}
\le
    C
        \|T\|_{\rH^2(\Omega)}
.
\end{aligned}
\]
Finally, using the boundedness of \(\mathcal H_T^{-1}\) on
\(\rH^1(\Omega)\), the smoothness of \(\hat B^\ast\), and Jensen's inequality,
we infer
\[
\begin{aligned}
    \|\mathcal C V\|_{\rH^1(\Omega)}
    \le
    C
    \left\|
        \mathcal I_z(\divH V)
    \right\|_{\rH^1(\Omega)}
    +
    C
    \left\|
        \mathcal I_1(\divH V)
    \right\|_{\rH^1(\Omega)}
    \le
    C\|V\|_{\rH^2(\Omega)}.
\end{aligned}
\]
Consequently, interpolation yields
\[
\begin{aligned}
    \|\mathbb B\mathbf u\|_{\rX_0}
    &\le
    C\big(
        \|\xi\|_{\rH^2(\T^2)}
        +\|V\|_{\rH^2(\Omega)}
        +\|T\|_{\rH^2(\Omega)}
    \big) \\ &\le
    \varepsilon
    \big(
        \|V\|_{\rH^3(\Omega)}
        +
        \|T\|_{\rH^3(\Omega)}
     \big )
+
    C(\varepsilon)
\big (
        \|\xi\|_{\rH^2(\T^2)}
        +
        \|V\|_{\rH^1(\Omega)}
        +
        \|T\|_{\rH^1(\Omega)}
    \big ) \\ & \le
    \varepsilon
    \|\mathbb A_0\mathbf u\|_{\rX_0}
    +
    C(\varepsilon)\|\mathbf u\|_{\rX_0}.
\end{aligned}
\]
Thus \(\mathbb B\) is relatively \(\mathbb A_0\)-bounded with relative bound
zero. By \cite[Proposition~4.3]{DHP:03}, there
exists \(\omega_\ast\ge0\) such that
\(
    \mathbb A+\omega_\ast
\)
is \(\mathcal R\)-sectorial on \(\rX_0\) with
\(
    \Phi^{\mathcal R}_{\mathbb A+\omega_\ast}<\pi/2.
\)
Hence \(\mathbb A\) has maximal \(\rL^2\)-regularity on \(\rX_0\). 
\end{step}
\begin{step}{\emph{Higher regularity}}\mbox{}\\
It remains to derive the additional time regularity stated in the lemma. From
the maximal-regularity result obtained in the previous step we already know that
\[
    \mathbf u
    =
    (\xi,V,T,Q_{\rv},Q_{\rc},Q_{\rr})
    \in
    \rL^2(0,\tau;\rX_1)
    \cap
    \rH^1(0,\tau;\rX_0).
\]
Thus only the improved regularity
\[
    \xi
    \in
    \rH^2(0,\tau;\rL^2(\T^2))
\]
and
\[
    V,T,Q_j
    \in
    \rH^1(0,\tau;\rH^2(\Omega))
    \cap
    \rH^2(0,\tau;\rL^2(\Omega))
\]
has still to be shown. We differentiate the relevant equations with respect to
time and set
\[
    \zeta:=\partial_t\xi,
    \quad
    W:=\partial_tV,
    \quad
    S:=\partial_tT \ \text{ and }\ 
    R_j:=\partial_tQ_j
    \quad\text{for } j\in\{\rv,\rc,\rr\}.
\]
Since all coefficients in the linearized system are independent of time, the
time derivatives satisfy
\begin{equation}\label{eq: differentiated linear system moist}
    \left\{
    \begin{aligned}
        \partial_t\zeta
        +\bar\rho_{\rd}^\ast\mathcal I_1(\divH W)
        &=
        \partial_t g_1,
        \\
        \partial_t W
        -\alpha\rL_{\rH}W
        +\mathcal A S
        &=
        \partial_t g_2
        -
        \frac{T^\ast}{\bar\rho_{\rd}^\ast}\nablaH \zeta,
        \\
        \mathcal H_T[\partial_t S]
        -\Delta S
        -\bar\rho_{\rd}^\ast\alpha\,\mathcal I_z(\divH W)
        +T^\ast e^{z/T^\ast}\mathcal I_1(\divH W)
        &=
        \partial_t g_3,
        \\
        \partial_t R_{\rv}-\Delta R_{\rv}
        &=
        \partial_t g_4,
        \\
        \partial_t R_{\rc}-\Delta R_{\rc}
        &=
        \partial_t g_5,
        \\
        \partial_t R_{\rr}-\Delta R_{\rr}
        &=
        \partial_t g_6.
    \end{aligned}
    \right.
\end{equation}
The boundary conditions for \(W,S\), and \(R_j\) are again homogeneous Neumann
conditions, since the original unknowns satisfy homogeneous Neumann boundary
conditions for all times.
We first use the regularity already obtained in \autoref{step2}. Since
\[
    \zeta=\partial_t\xi
    \in
    \rL^2(0,\tau;\rH^2(\T^2)),
\]
we have
\[
    \nablaH\zeta
    \in
    \rL^2(0,\tau;\rH^1(\T^2))
    \hookrightarrow
    \rL^2((0,\tau)\times\Omega;\R^2).
\]
Consequently,
\[
    \partial_t g_2
    -
    \frac{T^\ast}{\bar\rho_{\rd}^\ast}\nablaH \zeta
    \in
    \rL^2((0,\tau)\times\Omega;\R^2).
\]
Moreover,
\[
    \partial_tg_3,\partial_tg_j
    \in
   \rL^2((0,\tau) \times \Omega)
    \quad
    \text{for } j\in\{\rv,\rc,\rr\}.
\]
By the compatibility assumption, the initial data
\[
   (W(0),S(0),R_{\rv}(0),R_{\rc}(0),R_{\rr}(0))
\]
computed from \eqref{eq: full moist CPE Lagrange linear} satisfy
\[
   (W(0),S(0),R_{\rv}(0),R_{\rc}(0),R_{\rr}(0))
   \in
   \rH^1(\Omega;\R^2)
   \times
   \rH^1(\Omega)^4 .
\]
We may therefore apply the same lower-order maximal-regularity argument as in
\autoref{step2}, now only to the parabolic subsystem for
\[
    (W,S,R_{\rv},R_{\rc},R_{\rr}),
\]
with \(\zeta\) treated as a given function. 
It remains to improve the time regularity of \(\xi\). From the first equation
in \eqref{eq: differentiated linear system moist} we have
\[
    \partial_t\zeta
    =
    \partial_t g_1
    -
    \bar\rho_{\rd}^\ast\mathcal I_1(\divH W).
\]
Since
\[
    \partial_tg_1\in\rL^2(0,\tau;\rL^2(\T^2)) \ \text{ and }\    W\in\rL^2(0,\tau;\rH^2(\Omega;\R^2)),
\]
we obtain
\[
    \partial_t\zeta
    \in
    \rL^2(0,\tau;\rL^2(\T^2)).
\]
Since \(\zeta=\partial_t\xi\), this gives
\[
    \xi
    \in
    \rH^2(0,\tau;\rL^2(\T^2)).
\]
Combining this with the regularity from \autoref{step2} yields the claimed
regularity.
\end{step}
\end{proof}
\section{Nonlinear estimates}\label{subsec: nonlinear estimates}
\noindent
We now collect the nonlinear estimates needed for the fixed point argument.
In contrast to the dry case treated in \cite{TZ:26}, the right-hand sides in the
temperature and moisture equations contain contributions which are only of first
order in the perturbation variables. More precisely, the temperature equation
contains a linear contribution involving \(\partial_t Q_j\), while the
phase-transition terms contain the order-one switch
\(
    (q_{\rv}^\rL-q_{\rvs}^\rL)^+ .
\)
This creates a genuine difficulty, since these terms cannot be estimated
quadratically in the solution norm in a direct way.
Without loss of generality, we assume in the following that the saturation
mixing ratio vanishes, that is
\[
    q_{\rvs}=0 .
\]
This can be achieved by choosing the reference temperature \(T^\ast\) in the
cutoff regime of the saturation mixing ratio; see also the discussion in
\cite{HKLT:17}. 
The switch term is then incorporated into the left-hand side of the vapor
equation, leading to an auxiliary equation of the form
\[
    \partial_t q_j-\Delta q_j+q_j^+=h .
\]
The corresponding a priori estimate is the key ingredient for recovering
quadratic-order bounds for the remaining order-one terms.
Before dealing with the switch term, we first prove estimates for the genuinely
quadratic nonlinear terms. To this end, we remove the order-one contributions
from the velocity, temperature and moisture equations.
We introduce the linear moisture-time contribution in the temperature equation
by
\[
    \mathcal Q_T[\partial_t Q]
    :=
    \frac{T^\ast}{\hat B^\ast}
    D\hat B^\ast
    \bigl[
        0,\partial_t(q_{\rv}^\rL,q_{\rc}^\rL,q_{\rr}^\rL)
    \bigr]
    +
    \frac{1}{\hat B^\ast}
    \int_0^z
    D\hat B^\ast
    \bigl[
        0,\partial_t(q_{\rv}^\rL,q_{\rc}^\rL,q_{\rr}^\rL)
    \bigr](\cdot,\eta)\,\rd\eta .
\]
In the cutoff regime \(q_{\rvs}=0\), the phase-transition term reduces to
\((q_{\rv}^\rL)^+\). We therefore define 
\begin{equation*}
    \widetilde{G}_v = G_v+  T^\ast \nablaH q_{\rv}^\rL
        +\frac{T^\ast}{\hat B^\ast}
        D\hat B^\ast
        \bigl[
            \nablaH (0, q_{\rv}^\rL, q_{\rc}^\rL, q_{\rr}^\rL)
        \bigr]
\end{equation*}
\[
    \widetilde G_T
    :=
    G_T-\mathcal Q_T[\partial_t Q]-(q_{\rv}^\rL)^+,
\]
and
\[
    \widetilde G_{\rv}
    :=
    G_{\rv}+(q_{\rv}^\rL)^+,
    \quad
    \widetilde G_{\rc}
    :=
    G_{\rc}-(q_{\rv}^\rL)^+ \ \text{ and }\
    \widetilde G_{\rr}
    :=
    G_{\rr}.
\]
Here \(G_{\rr}\) is unchanged. Indeed, the only potentially non-quadratic
positive-part contribution in \(G_{\rr}\) is \((q_{\rc}^\rL-1)^+\), which
vanishes identically on sufficiently small balls in the solution space.
The following lemma precisely estimates these remainders.
\begin{lem}\label{lem: est nonlinear moist}
    Let \(\tau>0\), \(\bar\rho_{\rd}^\ast>0\), and \(T^\ast>0\). There exists
    \(\eps_0=\eps_0(\tau,\bar\rho_{\rd}^\ast,T^\ast,\Omega)>0\) such that for
    every \(\eps\in(0,\eps_0)\) the following holds. Assume that
    \[
        (
            \bar\rho_{\rd}^\rL,
            v^\rL,
            T^\rL,
            q_{\rv}^\rL,
            q_{\rc}^\rL,
            q_{\rr}^\rL
        )
        \in \mathbb B_\eps(0,\tau), 
    \]
    where 
    \begin{equation*}
          \mathbb B_\eps(0,\tau):=   \big\{
        (
            \bar\rho_{\rd}^\rL,
            v^\rL,
            T^\rL,
            q_{\rv}^\rL,
            q_{\rc}^\rL,
            q_{\rr}^\rL
        )\in\mathbb E_1(0,\tau)
        \colon
        \|(
            \bar\rho_{\rd}^\rL,
            v^\rL,
            T^\rL,
            q_{\rv}^\rL,
            q_{\rc}^\rL,
            q_{\rr}^\rL
        )\|_{\mathbb E_1(0,\tau)}\le \eps
        \big\}
    \end{equation*}
    and suppose that $V_r$ is smooth with $\| V_r\|_{\E_1^v(0,\tau)} \leq \eps$
    Then there exists a constant
    \(C=C(\tau,\bar\rho_{\rd}^\ast,T^\ast,\Omega)>0\) such that
    \[
        \big\|
        \big(
            G_\rho,
           \widetilde G_v,
            \widetilde G_T,
            \widetilde G_{\rv},
            \widetilde G_{\rc},
            \widetilde G_{\rr}
        \big)
        (
            \bar\rho_{\rd}^\rL,
            v^\rL,
            T^\rL,
            q_{\rv}^\rL,
            q_{\rc}^\rL,
            q_{\rr}^\rL
        )
        \big \|_{\mathbb E_0(0,\tau)}
        \le
        C\eps^2 .
    \]
\end{lem}
\begin{proof}
We begin by collecting a few auxiliary estimates that will be used repeatedly
below. Let
\[
    U^\rL
    :=
    (T^\rL,q_{\rv}^\rL,q_{\rc}^\rL,q_{\rr}^\rL).
\]
By the discussion in \autoref{sec Lagrange model}, the smoothness of the functional
\(\hat B\) around the reference state implies
\[
    \| \delta\hat B \|_{\mathbb E^h_1(0,\tau)}
    \leq C\eps .
\]
Moreover, the corresponding perturbation of the Fréchet derivative satisfies
\[
    \| \delta(D\hat B) \|_{\mathrm{op}}
    \leq C\eps .
\]
Here \(\|\cdot\|_{\mathrm{op}}\) denotes the operator norm in the Sobolev spaces
used below. Throughout the proof, we will repeatedly use the estimates on the transform that have been established in \autoref{lem:ests of trafo moist}.
We next estimate the averaged horizontal velocity \(b^\rL\).
Using the boundedness of vertical integration and the algebra property of
\(\rH^2(\T^2)\), we obtain
\[
    \| b^\rL \|_{\rH^1_\tau\rH^2(\T^2)}
    +
    \| b^\rL \|_{\rL^2_\tau\rH^3(\T^2)} + \| b^\rL \|_{\rH^2_\tau\rL^2(\T^2)}
    \leq C\eps .
\]
We now turn to the vertical velocity contribution. Recall that
\[
    \bigl((\bar\rho_{\rd}^\rL+\bar\rho_{\rd}^\ast)
    \hat B^\rL w^\rL\bigr)
    =
    J_1^{\mathrm m}+J_2^{\mathrm m}.
\]
We first estimate \(J_1^{\mathrm m}\). By the explicit formula for \(J_1^{\mathrm m}\), the boundedness of vertical
integration, and the smoothness of \(\hat B^\ast\), we obtain
\[
\begin{aligned}
    \|J_1^{\mathrm m}\|_{\rL^2_\tau\rH^2(\Omega)}
    +
    \|\del_zJ_1^{\mathrm m}\|_{\rL^2_\tau\rH^2(\Omega)}
    \leq
    C\big(
        \|\del_t T^\rL\|_{\rL^2_\tau\rH^2(\Omega)}
        +
        \sum_{j\in\{\rv,\rc,\rr\}}
        \|\del_t q_j^\rL\|_{\rL^2_\tau\rH^2(\Omega)}
        +
        \|v^\rL\|_{\rL^2_\tau\rH^3(\Omega)}
    \big)
    \leq C\eps .
\end{aligned}
\]
Moreover, differentiating \(J_1^{\mathrm m}\) with respect to time gives
\[
\begin{aligned}
    \|\del_tJ_1^{\mathrm m}\|_{\rL^2_\tau\rL^2(\Omega)}
    +
    \|\del_t\del_zJ_1^{\mathrm m}\|_{\rL^2_\tau\rL^2(\Omega)}
    \leq
    C\big(
        \|\del_{tt} T^\rL\|_{\rL^2_\tau\rL^2(\Omega)}
        +
        \sum_{j\in\{\rv,\rc,\rr\}}
        \|\del_{tt} q_j^\rL\|_{\rL^2_\tau\rL^2(\Omega)}
        +
        \|\del_t v^\rL\|_{\rL^2_\tau\rH^2(\Omega)}
    \big)
    \leq C\eps .
\end{aligned}
\]
Consequently,
\[
    \|J_1^{\mathrm m}\|_{\rL^2_\tau\rH^2(\Omega)}
    +
    \|J_1^{\mathrm m}\|_{\rH^1_\tau\rL^2(\Omega)}
    +
    \|\del_zJ_1^{\mathrm m}\|_{\rL^2_\tau\rH^2(\Omega)}
    +
    \|\del_zJ_1^{\mathrm m}\|_{\rH^1_\tau\rL^2(\Omega)}
    \leq C\eps .
\]
Next, to obtain an estimate of the remainder $J_2^\rm$, we calculate it explicitly, yielding
\begin{equation}\label{eq:J2-moist}
    \begin{aligned}
        J_2^{\mathrm m}
        &:=
        -\int_0^z
        \Big[
            (\hat B^\ast+\delta\hat B)
            \big(
                \widetilde v^\rL\cdot \mathrm Z^\top\nablaH\bar\rho_{\rd}^\rL
                +
                (\bar\rho_{\rd}^\rL+\bar\rho_{\rd}^\ast)
                \nablaH\widetilde v^\rL:(\mathrm Z^\top-\mathrm I_2)
            \big)
            \\
            &\quad
            +
            \hat B^\ast\bar\rho_{\rd}^\rL
            \divH\widetilde v^\rL
            +
            \delta\hat B\,\bar\rho_{\rd}^\rL
            \divH\widetilde v^\rL
            +
            \delta\hat B\,\bar\rho_{\rd}^\ast
            \divH\widetilde v^\rL
            \\
            &\quad
            -
            \hat B^\ast\bar\rho_{\rd}^\ast
            \int_0^1
            \Big[
                \bigl(D\hat B^\ast+\delta(D\hat B)\bigr)
                [\nablaH U^\rL]\cdot v^\rL
                +
                \delta\hat B\,\divH v^\rL
            \Big](\cdot,\zeta)\,\rd\zeta
            \\
            &\quad
            +
            \bar\rho_{\rd}^\rL
            \bigl(D\hat B^\ast+\delta(D\hat B)\bigr)
            \Big[
                \dt U^\rL
                +
                \mathrm Z\widetilde v^\rL\cdot\nablaH U^\rL
            \Big]
            +
            \bar\rho_{\rd}^\ast
            \delta(D\hat B)
            [\dt U^\rL]
            \\
            &\quad
            +
            \bar\rho_{\rd}^\ast
            \bigl(D\hat B^\ast+\delta(D\hat B)\bigr)
            \Big[
                \mathrm Z\widetilde v^\rL\cdot\nablaH U^\rL
            \Big]
        \Big](\cdot,\eta)\,\rd\eta .
    \end{aligned}
\end{equation}
We now prove the corresponding \(\rH^1_\tau\rL^2\)-estimate.
We discuss the terms in \(\dt J_2^{\mathrm m}\) which are representative for
the estimate. Differentiating the first contribution gives
\[
\begin{aligned}
    \dt
    \Big[
        (\hat B^\ast+\delta\hat B)
        \widetilde v^\rL\cdot \mathrm Z^\top\nablaH\bar\rho_{\rd}^\rL
    \Big]
    &=
    \dt\delta\hat B\,
    \widetilde v^\rL\cdot \mathrm Z^\top\nablaH\bar\rho_{\rd}^\rL
    +
    (\hat B^\ast+\delta\hat B)
    \dt\widetilde v^\rL\cdot \mathrm Z^\top\nablaH\bar\rho_{\rd}^\rL
    \\
    &\quad
    +
    (\hat B^\ast+\delta\hat B)
    \widetilde v^\rL\cdot \dt\mathrm Z^\top\nablaH\bar\rho_{\rd}^\rL
    +
    (\hat B^\ast+\delta\hat B)
    \widetilde v^\rL\cdot \mathrm Z^\top\nablaH\dt\bar\rho_{\rd}^\rL .
\end{aligned}
\]
Hence
\[
\begin{aligned}
    \big\|
    \dt
    \Big[
        (\hat B^\ast+\delta\hat B)
        \widetilde v^\rL\cdot \mathrm Z^\top\nablaH\bar\rho_{\rd}^\rL
    \Big]
    \big\|_{\rL^2_\tau\rL^2(\Omega)}
    &\leq
    C \big (
    \|\dt\delta\hat B\|_{\rL^2_\tau\rH^2}
    \|\widetilde v^\rL\|_{\rL^\infty_\tau\rH^2}
    \|\bar\rho_{\rd}^\rL\|_{\rL^\infty_\tau\rH^2}
    +
    \|\dt\widetilde v^\rL\|_{\rL^2_\tau\rH^2}
    \|\bar\rho_{\rd}^\rL\|_{\rL^\infty_\tau\rH^2}
    \\
    &\quad
    +
    \|\dt\mathrm Z\|_{\rL^2_\tau\rH^2}
    \|\widetilde v^\rL\|_{\rL^\infty_\tau\rH^2}
    \|\bar\rho_{\rd}^\rL\|_{\rL^\infty_\tau\rH^2}
    +
    \|\widetilde v^\rL\|_{\rL^\infty_\tau\rH^2}
    \|\dt\bar\rho_{\rd}^\rL\|_{\rL^2_\tau\rH^2} \big )
    \\
    &\leq C\eps^2 .
\end{aligned}
\]
Similarly
\[
\begin{aligned}
  &\quad  \big\|
    \dt
    \Big[
        (\hat B^\ast+\delta\hat B)
        (\bar\rho_{\rd}^\rL+\bar\rho_{\rd}^\ast)
        \nablaH\widetilde v^\rL:(\mathrm Z^\top-\mathrm I_2)
    \Big]
    \big\|_{\rL^2_\tau\rL^2(\Omega)}
   \\ &\leq
    C \big (
    \|\dt\delta\hat B\|_{\rL^2_\tau\rH^2}
    \|\nablaH\widetilde v^\rL\|_{\rL^\infty_\tau\rH^1}
    \|\mathrm Z^\top-\mathrm I_2\|_{\rL^\infty_\tau\rH^2}
    +
    \|\dt\bar\rho_{\rd}^\rL\|_{\rL^2_\tau\rH^2}
    \|\nablaH\widetilde v^\rL\|_{\rL^\infty_\tau\rH^1}
    \|\mathrm Z^\top-\mathrm I_2\|_{\rL^\infty_\tau\rH^2}
    \\
    &\quad
    +
    \|\dt\widetilde v^\rL\|_{\rL^2_\tau\rH^2}
    \|\mathrm Z^\top-\mathrm I_2\|_{\rL^\infty_\tau\rH^2}
    +
    \|\nablaH\widetilde v^\rL\|_{\rL^\infty_\tau\rH^1}
    \|\dt\mathrm Z\|_{\rL^2_\tau\rH^2} \big )
    \\
    &\leq C\eps^2 .
\end{aligned}
\]
All other contributions contain either \(\dt\delta\hat B\),
\(\dt\bar\rho_{\rd}^\rL\), or \(\dt\mathrm Z\), and are estimated in the same
way.
Next, the divergence terms give
\[
\begin{aligned}
    \dt
    \bigl [
        \hat B^\ast\bar\rho_{\rd}^\rL\divH\widetilde v^\rL
    \bigr ]
    &=
    \hat B^\ast\dt\bar\rho_{\rd}^\rL\divH\widetilde v^\rL
    +
    \hat B^\ast\bar\rho_{\rd}^\rL\divH\dt\widetilde v^\rL ,
\end{aligned}
\]
and therefore
\[
\begin{aligned}
    \big \|
        \dt
        \bigl [
            \hat B^\ast\bar\rho_{\rd}^\rL\divH\widetilde v^\rL
        \bigr ]
    \big\|_{\rL^2_\tau\rL^2(\Omega)}
    \leq
    C \big (
    \|\dt\bar\rho_{\rd}^\rL\|_{\rL^2_\tau\rH^2}
    \|\widetilde v^\rL\|_{\rL^\infty_\tau\rH^2}
    +
    \|\bar\rho_{\rd}^\rL\|_{\rL^\infty_\tau\rH^2}
    \|\dt\widetilde v^\rL\|_{\rL^2_\tau\rH^2}\big ) \leq C\eps^2 .
\end{aligned}
\]
The two terms containing \(\delta\hat B\divH\widetilde v^\rL\) are estimated
identically.
We now consider the averaged contribution
\[
    \int_0^1
    \Big[
        \bigl(D\hat B^\ast+\delta(D\hat B)\bigr)
        [\nablaH U^\rL]\cdot v^\rL
        +
        \delta\hat B\,\divH v^\rL
    \Big](\cdot,\zeta)\,\rd\zeta .
\]
For the frozen part we obtain
\[
\begin{aligned}
    \dt
    \big [
        D\hat B^\ast[\nablaH U^\rL]\cdot v^\rL
    \big ]
    =
    D\hat B^\ast[\nablaH\dt U^\rL]\cdot v^\rL
    +
    D\hat B^\ast[\nablaH U^\rL]\cdot\dt v^\rL ,
\end{aligned}
\]
and hence
\[
\begin{aligned}
    \big \|
        \dt
        \big [
            D\hat B^\ast[\nablaH U^\rL]\cdot v^\rL
        \big ]
    \big \|_{\rL^2_\tau\rL^2(\Omega)}
    \leq
    \big (
    \|\dt U^\rL\|_{\rL^2_\tau\rH^2}
    \|v^\rL\|_{\rL^\infty_\tau\rH^2}
    +
    \|U^\rL\|_{\rL^\infty_\tau\rH^2}
    \|\dt v^\rL\|_{\rL^2_\tau\rH^2} \big )
    \leq C\eps^2 .
\end{aligned}
\]
The corresponding term with \(\delta(D\hat B)\) is controlled by the operator
estimate for \(\delta(D\hat B)\), namely
\[
\begin{aligned}
    &\quad\big \|
        \delta(D\hat B)[\nablaH U^\rL]\cdot v^\rL
    \big\|_{\rH^1_\tau\rL^2(\Omega)}
   \\& \leq
    C \big (
    \|\delta(D\hat B)[\nablaH U^\rL]\|_{\rH^1_\tau\rL^2(\Omega)}
    \|v^\rL\|_{\rL^\infty_\tau\rH^2(\Omega)}
    +
    \|\delta(D\hat B)[\nablaH U^\rL]\|_{\rL^\infty_\tau\rH^1(\Omega)}
    \|\dt v^\rL\|_{\rL^2_\tau\rH^1(\Omega)} \big )
\leq C\eps^2 .
\end{aligned}
\]
Also,
\[
\begin{aligned}
    \big \|
        \dt(\delta\hat B\,\divH v^\rL)
    \big \|_{\rL^2_\tau\rL^2(\Omega)}
    \leq
    C \big (
    \|\dt\delta\hat B\|_{\rL^2_\tau\rH^2}
    \|v^\rL\|_{\rL^\infty_\tau\rH^2}
    +
    \|\delta\hat B\|_{\rL^\infty_\tau\rH^2}
    \|\dt v^\rL\|_{\rL^2_\tau\rH^2}
    \leq C\eps^2 \big ) .
\end{aligned}
\]
It remains to discuss the terms involving \(\dt U^\rL\) and the advective
quantity
\(
    \mathrm Z\widetilde v^\rL\cdot\nablaH U^\rL .
\)
For the first one we use the product rule:
\[
\begin{aligned}
    \dt
    \big[
        \bar\rho_{\rd}^\rL D\hat B^\ast[\dt U^\rL]
    \big]
    &=
    \dt\bar\rho_{\rd}^\rL D\hat B^\ast[\dt U^\rL]
    +
    \bar\rho_{\rd}^\rL D\hat B^\ast[\dt^2 U^\rL].
\end{aligned}
\]
Thus
\[
\begin{aligned}
    \big \|
        \dt
        \big[
            \bar\rho_{\rd}^\rL D\hat B^\ast[\dt U^\rL]
        \big]
    \big\|_{\rL^2_\tau\rL^2(\Omega)}
    \leq
    C \big (
    \|\dt\bar\rho_{\rd}^\rL\|_{\rL^2_\tau\rH^2}
    \|\dt U^\rL\|_{\rL^\infty_\tau\rH^1}
    +
    \|\bar\rho_{\rd}^\rL\|_{\rL^\infty_\tau\rH^2}
    \|\dt^2 U^\rL\|_{\rL^2_\tau\rL^2} \big )
\leq C\eps^2 .
\end{aligned}
\]
The term
\(
    \bar\rho_{\rd}^\ast\delta(D\hat B)[\dt U^\rL]
\)
is estimated directly by the operator estimate for \(\delta(D\hat B)\), that is,
\[
    \big \|
        \delta(D\hat B)[\dt U^\rL]
    \big \|_{\rH^1_\tau\rL^2(\Omega)}
    \leq C\eps
    \|\dt U^\rL\|_{\rH^1_\tau\rL^2(\Omega)}
    \leq C\eps^2 .
\]
Finally, set
\(
    A^\rL
    :=
    \mathrm Z\widetilde v^\rL\cdot\nablaH U^\rL .
\)
Then
\[
    \dt A^\rL
    =
    \dt\mathrm Z\,\widetilde v^\rL\cdot\nablaH U^\rL
    +
    \mathrm Z\dt\widetilde v^\rL\cdot\nablaH U^\rL
    +
    \mathrm Z\widetilde v^\rL\cdot\nablaH\dt U^\rL .
\]
Consequently,
\[
\begin{aligned}
   &\quad \|\dt A^\rL\|_{\rL^2_\tau\rL^2(\Omega)}
   \\& \leq
    C \big 
    \|\dt\mathrm Z\|_{\rL^2_\tau\rH^2}
    \|\widetilde v^\rL\|_{\rL^\infty_\tau\rH^2}
    \|U^\rL\|_{\rL^\infty_\tau\rH^2}
    +
    \|\dt\widetilde v^\rL\|_{\rL^2_\tau\rH^2}
    \|U^\rL\|_{\rL^\infty_\tau\rH^2}
    +
    \|\widetilde v^\rL\|_{\rL^\infty_\tau\rH^2}
    \|\dt U^\rL\|_{\rL^2_\tau\rH^2} \big )
  \leq C\eps^2 .
\end{aligned}
\]
It follows that
\[
    \big \|
        \dt
        \big [
            (\bar\rho_{\rd}^\rL+\bar\rho_{\rd}^\ast)
            \bigl(D\hat B^\ast+\delta(D\hat B)\bigr)[A^\rL]
        \big ]
    \big\|_{\rL^2_\tau\rL^2(\Omega)}
    \leq C\eps^2 .
\]
Combining the preceding estimates and using once more the boundedness of
vertical integration on \(\rL^2(\Omega)\), we conclude that
\[
    \|\dt J_2^{\mathrm m}\|_{\rL^2_\tau\rL^2(\Omega)}
    \leq C\eps^2 .
\]
We now prove the corresponding \(\rL^2_\tau\rH^1\)-estimate. We discuss the
terms in \(J_2^{\mathrm m}\) which are representative for the estimate. For
the first contribution, the spatial product rule gives
\[
\begin{aligned}
    \nabla
    \Big[
        (\hat B^\ast+\delta\hat B)
        \widetilde v^\rL\cdot \mathrm Z^\top\nablaH\bar\rho_{\rd}^\rL
    \Big]
    &=
    \nabla\delta\hat B\,
    \widetilde v^\rL\cdot \mathrm Z^\top\nablaH\bar\rho_{\rd}^\rL
    +
    (\hat B^\ast+\delta\hat B)
    \nabla\widetilde v^\rL\cdot \mathrm Z^\top\nablaH\bar\rho_{\rd}^\rL
    \\
    &\quad
    +
    (\hat B^\ast+\delta\hat B)
    \widetilde v^\rL\cdot \nabla\mathrm Z^\top\nablaH\bar\rho_{\rd}^\rL
    +
    (\hat B^\ast+\delta\hat B)
    \widetilde v^\rL\cdot \mathrm Z^\top\nabla\nablaH\bar\rho_{\rd}^\rL .
\end{aligned}
\]
Hence
\[
\begin{aligned}
    \big\|
    \nabla
    \Big[
        (\hat B^\ast+\delta\hat B)
        \widetilde v^\rL\cdot \mathrm Z^\top\nablaH\bar\rho_{\rd}^\rL
    \Big]
    \big\|_{\rL^2_\tau\rL^2(\Omega)}
    &\leq
    C \big (
    \|\delta\hat B\|_{\rL^\infty_\tau\rH^2}
    \|\widetilde v^\rL\|_{\rL^\infty_\tau\rH^2}
    \|\bar\rho_{\rd}^\rL\|_{\rL^2_\tau\rH^3}
    +
    \|\widetilde v^\rL\|_{\rL^2_\tau\rH^2}
    \|\bar\rho_{\rd}^\rL\|_{\rL^\infty_\tau\rH^2}
    \\
    &\quad
    +
    \|\mathrm Z\|_{\rL^\infty_\tau\rH^2}
    \|\widetilde v^\rL\|_{\rL^\infty_\tau\rH^2}
    \|\bar\rho_{\rd}^\rL\|_{\rL^2_\tau\rH^2}
    +
    \|\widetilde v^\rL\|_{\rL^\infty_\tau\rH^2}
    \|\bar\rho_{\rd}^\rL\|_{\rL^2_\tau\rH^2} \big )
    \\
    &\leq C\eps^2 .
\end{aligned}
\]
Together with the corresponding \(\rL^2_\tau\rL^2\)-estimate, this yields
\[
    \big\|
        (\hat B^\ast+\delta\hat B)
        \widetilde v^\rL\cdot \mathrm Z^\top\nablaH\bar\rho_{\rd}^\rL
    \big\|_{\rL^2_\tau\rH^1(\Omega)}
    \leq C\eps^2 .
\]
Similarly, for the second transport contribution we obtain
\[
\begin{aligned}
    &\quad
    \big\|
    \nabla
    \Big[
        (\hat B^\ast+\delta\hat B)
        (\bar\rho_{\rd}^\rL+\bar\rho_{\rd}^\ast)
        \nablaH\widetilde v^\rL:(\mathrm Z^\top-\mathrm I_2)
    \Big]
    \big\|_{\rL^2_\tau\rL^2(\Omega)}
    \\
    &\leq
    C \big (
    \|\delta\hat B\|_{\rL^\infty_\tau\rH^2}
    \|\nablaH\widetilde v^\rL\|_{\rL^2_\tau\rH^2}
    \|\mathrm Z^\top-\mathrm I_2\|_{\rL^\infty_\tau\rH^2}
    +
    \|\bar\rho_{\rd}^\rL\|_{\rL^\infty_\tau\rH^2}
    \|\nablaH\widetilde v^\rL\|_{\rL^2_\tau\rH^2}
    \|\mathrm Z^\top-\mathrm I_2\|_{\rL^\infty_\tau\rH^2}
    \\
    &\quad
+
    \|\nablaH\widetilde v^\rL\|_{\rL^2_\tau\rH^2}
    \|\mathrm Z^\top-\mathrm I_2\|_{\rL^\infty_\tau\rH^2}
    +
    \|\nablaH\widetilde v^\rL\|_{\rL^\infty_\tau\rH^1}
    \|\mathrm Z^\top-\mathrm I_2\|_{\rL^2_\tau\rH^3} \big )
    \\
    &\leq C\eps^2 .
\end{aligned}
\]
Thus
\[
    \big\|
        (\hat B^\ast+\delta\hat B)
        (\bar\rho_{\rd}^\rL+\bar\rho_{\rd}^\ast)
        \nablaH\widetilde v^\rL:(\mathrm Z^\top-\mathrm I_2)
    \big\|_{\rL^2_\tau\rH^1(\Omega)}
    \leq C\eps^2 .
\]
Next, the divergence terms are estimated in the same way. For instance,
\[
\begin{aligned}
    \nabla
    \bigl[
        \hat B^\ast\bar\rho_{\rd}^\rL\divH\widetilde v^\rL
    \bigr]
    =
    \hat B^\ast\nabla\bar\rho_{\rd}^\rL\divH\widetilde v^\rL
    +
    \hat B^\ast\bar\rho_{\rd}^\rL\nabla\divH\widetilde v^\rL ,
\end{aligned}
\]
and therefore
\[
\begin{aligned}
    \big \|
        \nabla
        \bigl[
            \hat B^\ast\bar\rho_{\rd}^\rL\divH\widetilde v^\rL
        \bigr]
    \big\|_{\rL^2_\tau\rL^2(\Omega)}
    \leq
    C \big (
    \|\bar\rho_{\rd}^\rL\|_{\rL^\infty_\tau\rH^2}
    \|\widetilde v^\rL\|_{\rL^2_\tau\rH^3}
    +
    \|\bar\rho_{\rd}^\rL\|_{\rL^\infty_\tau\rH^2}
    \|\widetilde v^\rL\|_{\rL^\infty_\tau\rH^2}
    \leq C\eps^2 \big ).
\end{aligned}
\]
The two terms containing \(\delta\hat B\divH\widetilde v^\rL\) are controlled
identically, since
\[
\begin{aligned}
    \big\|
        \delta\hat B\,\bar\rho_{\rd}^\rL\divH\widetilde v^\rL
    \big\|_{\rL^2_\tau\rH^1(\Omega)}
    +
    \big\|
        \delta\hat B\,\bar\rho_{\rd}^\ast\divH\widetilde v^\rL
    \big\|_{\rL^2_\tau\rH^1(\Omega)}
    \leq C\eps^2 .
\end{aligned}
\]
We now consider the averaged contribution
\[
    \int_0^1
    \Big[
        \bigl(D\hat B^\ast+\delta(D\hat B)\bigr)
        [\nablaH U^\rL]\cdot v^\rL
        +
        \delta\hat B\,\divH v^\rL
    \Big](\cdot,\zeta)\,\rd\zeta .
\]
For the frozen part, the spatial product rule gives
\[
\begin{aligned}
    \nabla
    \big [
        D\hat B^\ast[\nablaH U^\rL]\cdot v^\rL
    \big ]
    =
    D\hat B^\ast[\nabla\nablaH U^\rL]\cdot v^\rL
    +
    D\hat B^\ast[\nablaH U^\rL]\cdot\nabla v^\rL ,
\end{aligned}
\]
and hence
\[
\begin{aligned}
    \big \|
        \nabla
        \big [
            D\hat B^\ast[\nablaH U^\rL]\cdot v^\rL
        \big ]
    \big \|_{\rL^2_\tau\rL^2(\Omega)}
    \leq
    C \big (
    \|U^\rL\|_{\rL^2_\tau\rH^3}
    \|v^\rL\|_{\rL^\infty_\tau\rH^2}
   +
    \|U^\rL\|_{\rL^\infty_\tau\rH^2}
    \|v^\rL\|_{\rL^2_\tau\rH^3}
    \leq C\eps^2 \big ).
\end{aligned}
\]
The corresponding term with \(\delta(D\hat B)\) is controlled by the operator
estimate for \(\delta(D\hat B)\), namely
\[
\begin{aligned}
    &\quad
    \big \|
        \delta(D\hat B)[\nablaH U^\rL]\cdot v^\rL
    \big\|_{\rL^2_\tau\rH^1(\Omega)}
    \\
    &\leq
    C \big (
    \|\delta(D\hat B)[\nablaH U^\rL]\|_{\rL^2_\tau\rH^1(\Omega)}
    \|v^\rL\|_{\rL^\infty_\tau\rH^2(\Omega)}
    +
    \|\delta(D\hat B)[\nablaH U^\rL]\|_{\rL^\infty_\tau\rH^1(\Omega)}
    \|v^\rL\|_{\rL^2_\tau\rH^2(\Omega)} \big )
    \\
    &\leq C\eps^2 .
\end{aligned}
\]
Moreover,
\[
\begin{aligned}
    \big \|
        \delta\hat B\,\divH v^\rL
    \big \|_{\rL^2_\tau\rH^1(\Omega)}
    &\leq
    C \big (
    \|\delta\hat B\|_{\rL^\infty_\tau\rH^2}
    \|v^\rL\|_{\rL^2_\tau\rH^3}
    +
    \|\delta\hat B\|_{\rL^2_\tau\rH^3}
    \|v^\rL\|_{\rL^\infty_\tau\rH^2} \big )
    \\
    &\leq C\eps^2 .
\end{aligned}
\]
It remains to discuss the terms involving \(\dt U^\rL\) and the advective
quantity
For the first one, we use the spatial product estimate
\[
\begin{aligned}
    \big \|
        \bar\rho_{\rd}^\rL D\hat B^\ast[\dt U^\rL]
    \big \|_{\rL^2_\tau\rH^1(\Omega)}
    \leq
    C
    \|\bar\rho_{\rd}^\rL\|_{\rL^\infty_\tau\rH^2(\Omega)}
    \|\dt U^\rL\|_{\rL^2_\tau\rH^1(\Omega)}
    \leq C\eps^2 .
\end{aligned}
\]
Similarly,
\[
\begin{aligned}
    \big \|
        \bar\rho_{\rd}^\rL\delta(D\hat B)[\dt U^\rL]
    \big \|_{\rL^2_\tau\rH^1(\Omega)}
    +
    \big \|
        \bar\rho_{\rd}^\ast\delta(D\hat B)[\dt U^\rL]
    \big \|_{\rL^2_\tau\rH^1(\Omega)}
    \leq C\eps^2 .
\end{aligned}
\]
Finally, wit $A^\rL$ as above, we have
\[
\begin{aligned}
    \nabla A^\rL
    =
    \nabla\mathrm Z\,\widetilde v^\rL\cdot\nablaH U^\rL
    +
    \mathrm Z\nabla\widetilde v^\rL\cdot\nablaH U^\rL
    +
    \mathrm Z\widetilde v^\rL\cdot\nabla\nablaH U^\rL .
\end{aligned}
\]
Consequently,
\[
\begin{aligned}
   \|A^\rL\|_{\rL^2_\tau\rH^1(\Omega)}
   \leq
    C \big (
    \|\mathrm Z\|_{\rL^\infty_\tau\rH^2}
    \|\widetilde v^\rL\|_{\rL^\infty_\tau\rH^2}
    \|U^\rL\|_{\rL^2_\tau\rH^3}
    +
    \|\mathrm Z\|_{\rL^\infty_\tau\rH^2}
    \|\widetilde v^\rL\|_{\rL^2_\tau\rH^3}
    \|U^\rL\|_{\rL^\infty_\tau\rH^2} \big )
\leq C\eps^2 .
\end{aligned}
\]
It follows that
\[
\begin{aligned}
    \big \|
        (\bar\rho_{\rd}^\rL+\bar\rho_{\rd}^\ast)
        \bigl(D\hat B^\ast+\delta(D\hat B)\bigr)[A^\rL]
    \big \|_{\rL^2_\tau\rH^1(\Omega)}
    \leq C\eps^2 .
\end{aligned}
\]
Combining the preceding estimates and using the boundedness of vertical
integration on \(\rH^1(\Omega)\), we conclude that
\[
    \|J_2^{\mathrm m}\|_{\rL^2_\tau\rH^1(\Omega)}
    \leq C\eps^2 .
\]
Consequently,
\[
    \|J_2^{\mathrm m}\|_{\rL^2_\tau\rH^1(\Omega)}
    +
    \|J_2^{\mathrm m}\|_{\rH^1_\tau\rL^2(\Omega)}
    +
    \|\del_zJ_2^{\mathrm m}\|_{\rL^2_\tau\rH^1(\Omega)}
    +
    \|\del_zJ_2^{\mathrm m}\|_{\rH^1_\tau\rL^2(\Omega)}
    \leq C\eps^2 .
\]
Combining the estimates of $J_1^\rm$ and $J_2^\rm$, we obtain
\[
    \big \|
        \bigl((\bar\rho_{\rd}^\rL+\bar\rho_{\rd}^\ast)
        \hat B^\rL w^\rL\bigr)
    \big \|_{\rL^2_\tau\rH^1(\Omega)\cap\rH^1_\tau\rL^2(\Omega) \cap \rL^\infty_\tau \rL^3(\Omega)}
    +
    \big \|
        \del_z
        \bigl((\bar\rho_{\rd}^\rL+\bar\rho_{\rd}^\ast)
        \hat B^\rL w^\rL\bigr)
    \big \|_{\rL^2_\tau\rH^1(\Omega)\cap\rH^1_\tau\rL^2(\Omega) \cap \rL^\infty_\tau \rL^3(\Omega)}
    \leq C\eps ,
\]
using the embedding
\begin{equation*}
    \rL^2_\tau\rH^1(\Omega)\cap\rH^1_\tau\rL^2(\Omega) \hookrightarrow \rL^\infty_\tau \rL^3(\Omega).
\end{equation*}
We now estimate \(G_\rho\). 
Using the boundedness
of vertical integration and the algebra property of \(\rH^2\), we obtain
\[
\begin{aligned}
    &\quad
    \big \|
    \int_0^1
    \Big[
        \delta\hat B\,\divH v^\rL
        +
        \bigl(D\hat B^\ast+\delta(D\hat B)\bigr)
        [\nablaH U^\rL]\cdot v^\rL
    \Big](\cdot,\eta)\,\rd\eta
    \big \|_{\rL^2_\tau\rH^2(\T^2)}
    \\
    &\leq
    C
    \|\delta\hat B\|_{\rL^\infty_\tau\rH^2(\Omega)}
    \|\divH v^\rL\|_{\rL^2_\tau\rH^2(\Omega)}
    \\
    &\quad
    +
    C
    \Big(
        \|D\hat B^\ast[\nablaH U^\rL]\|_{\rL^2_\tau\rH^2(\Omega)}
        +
        \|\delta(D\hat B)\|_{\mathrm{op}}
        \|U^\rL\|_{\rL^2_\tau\rH^3(\Omega)}
    \Big)
    \|v^\rL\|_{\rL^\infty_\tau\rH^2(\Omega)}
    \\
    &\leq C\eps^2 .
\end{aligned}
\]
For the remaining part, we get
\[
\begin{aligned}
    &\quad
    \big \|
        \bar\rho_{\rd}^\rL \divH b^\rL
        +
        (\bar\rho_{\rd}^\rL+\bar\rho_{\rd}^\ast)
        \nablaH b^\rL:(\mathrm Z^\top-\mathrm I_2)
    \big \|_{\rL^2_\tau\rH^2(\T^2)}
    \\
    &\leq
    C
    \|\bar\rho_{\rd}^\rL\|_{\rL^\infty_\tau\rH^2(\T^2)}
    \|\divH b^\rL\|_{\rL^2_\tau\rH^2(\T^2)}
    +
    C
    \|\bar\rho_{\rd}^\rL+\bar\rho_{\rd}^\ast\|_{\rL^\infty_\tau\rH^2(\T^2)}
    \|\nablaH b^\rL\|_{\rL^2_\tau\rH^2(\T^2)}
    \|\mathrm Z^\top-\mathrm I_2\|_{\rL^\infty_\tau\rH^2(\T^2)}
    \\
    &\leq C\eps^2 .
\end{aligned}
\]
We next estimate the \(\rH^1_\tau\rL^2\)-part. By the product rule, the
boundedness of vertical integration and the estimates for
\(\delta\hat B\) and \(\delta(D\hat B)\), we obtain
\[
\begin{aligned}
    &\quad
    \big \|
    \dt
    \int_0^1
    \Big[
        \delta\hat B\,\divH v^\rL
        +
        \bigl(D\hat B^\ast+\delta(D\hat B)\bigr)
        [\nablaH U^\rL]\cdot v^\rL
    \Big](\cdot,\eta)\,\rd\eta
    \big \|_{\rL^2_\tau\rL^2(\T^2)}
    \\
    &\leq
    C
    \|\dt\delta\hat B\|_{\rL^2_\tau\rH^2(\Omega)}
    \|v^\rL\|_{\rL^\infty_\tau\rH^2(\Omega)}
    +
    C
    \|\delta\hat B\|_{\rL^\infty_\tau\rH^2(\Omega)}
    \|\dt v^\rL\|_{\rL^2_\tau\rH^2(\Omega)}
    \\
    &\quad
    +
    C
    \|\dt U^\rL\|_{\rL^2_\tau\rH^2(\Omega)}
    \|v^\rL\|_{\rL^\infty_\tau\rH^2(\Omega)}
    +
    C
    \|U^\rL\|_{\rL^\infty_\tau\rH^2(\Omega)}
    \|\dt v^\rL\|_{\rL^2_\tau\rH^2(\Omega)}
    \\
    &\quad
    +
    C
    \big \|\delta(D\hat B)[\nablaH U^\rL]\big \|_{\rH^1_\tau\rL^2(\Omega)}
    \|v^\rL\|_{\rL^\infty_\tau\rH^2(\Omega)}
    +
    C
    \big \|\delta(D\hat B)[\nablaH U^\rL]\big \|_{\rL^\infty_\tau\rH^1(\Omega)}
    \|\dt v^\rL\|_{\rL^2_\tau\rH^1(\Omega)}
    \\
    &\leq C\eps^2 .
\end{aligned}
\]
For the remaining part, again using the product, we get
\[
\begin{aligned}
    &\quad
    \big \|
        \dt
        \Big[
            \bar\rho_{\rd}^\rL \divH b^\rL
            +
            (\bar\rho_{\rd}^\rL+\bar\rho_{\rd}^\ast)
            \nablaH b^\rL:(\mathrm Z^\top-\mathrm I_2)
        \Big]
    \big \|_{\rL^2_\tau\rL^2(\T^2)}
    \\
    &\leq
    C
    \|\dt\bar\rho_{\rd}^\rL\|_{\rL^2_\tau\rH^2(\T^2)}
    \|b^\rL\|_{\rL^\infty_\tau\rH^2(\T^2)}
    +
    C
    \|\bar\rho_{\rd}^\rL\|_{\rL^\infty_\tau\rH^2(\T^2)}
    \|\dt b^\rL\|_{\rL^2_\tau\rH^2(\T^2)}
    \\
    &\quad
    +
    C
    \|\dt\bar\rho_{\rd}^\rL\|_{\rL^2_\tau\rH^2(\T^2)}
    \|\nablaH b^\rL\|_{\rL^\infty_\tau\rH^1(\T^2)}
    \|\mathrm Z^\top-\mathrm I_2\|_{\rL^\infty_\tau\rH^2(\T^2)}
    \\
    &\quad
    +
    C
    \|\bar\rho_{\rd}^\rL+\bar\rho_{\rd}^\ast\|_{\rL^\infty_\tau\rH^2(\T^2)}
    \|\dt b^\rL\|_{\rL^2_\tau\rH^2(\T^2)}
    \|\mathrm Z^\top-\mathrm I_2\|_{\rL^\infty_\tau\rH^2(\T^2)}
    \\
    &\quad
    +
    C
    \|\bar\rho_{\rd}^\rL+\bar\rho_{\rd}^\ast\|_{\rL^\infty_\tau\rH^2(\T^2)}
    \|\nablaH b^\rL\|_{\rL^\infty_\tau\rH^1(\T^2)}
    \|\dt\mathrm Z\|_{\rL^2_\tau\rH^2(\T^2)}
    \\
    &\leq C\eps^2 .
\end{aligned}
\]
Combining this with the previous \(\rL^2_\tau\rH^2\)-estimate yields
\[
    \|G_\rho\|_{\rL^2_\tau\rH^2(\T^2)\cap\rH^1_\tau\rL^2(\T^2)}
    \leq C\eps^2 .
\]
Next, we estimate the nonlinear term $G_v$. We start with the contribution containing \(\dt v^\rL\). We only treat one
representative term, namely
\[
    \frac{\bar\rho_{\rd}^\rL}{\bar\rho_{\rd}^\ast}\dt v^\rL .
\]
Using the product rule,
\[
    \nabla\bigl(\bar\rho_{\rd}^\rL\dt v^\rL\bigr)
    =
    \nablaH\bar\rho_{\rd}^\rL\,\dt v^\rL
    +
    \bar\rho_{\rd}^\rL\nabla\dt v^\rL .
\]
Therefore,
\[
\begin{aligned}
    \big \|
        \nabla
        \bigg(
            \frac{\bar\rho_{\rd}^\rL}{\bar\rho_{\rd}^\ast}
            \dt v^\rL
        \bigg)
    \big \|_{\rL^2_\tau\rL^2(\Omega)}
    &\leq
    C
    \|\nablaH\bar\rho_{\rd}^\rL\|_{\rL^\infty_\tau\rL^6(\T^2)}
    \|\dt v^\rL\|_{\rL^2_\tau\rL^3(\Omega)}
    +
    C
    \|\bar\rho_{\rd}^\rL\|_{\rL^\infty_\tau\rL^\infty(\T^2)}
    \|\nabla\dt v^\rL\|_{\rL^2_\tau\rL^2(\Omega)}
    \\
    &\leq
    C
    \|\bar\rho_{\rd}^\rL\|_{\rL^\infty_\tau\rH^2(\T^2)}
    \|\dt v^\rL\|_{\rL^2_\tau\rH^1(\Omega)}
    \\
    &\leq C\eps^2 .
\end{aligned}
\]
Together with the lower-order estimate
\[
\begin{aligned}
    \big \|
        \frac{\bar\rho_{\rd}^\rL}{\bar\rho_{\rd}^\ast}
        \dt v^\rL
    \big \|_{\rL^2_\tau\rL^2(\Omega)}
    \leq
    C
    \|\bar\rho_{\rd}^\rL\|_{\rL^\infty_\tau\rL^\infty(\T^2)}
    \|\dt v^\rL\|_{\rL^2_\tau\rL^2(\Omega)}
    \leq C\eps^2 ,
\end{aligned}
\]
we obtain
\[
    \big \|
        \frac{\bar\rho_{\rd}^\rL}{\bar\rho_{\rd}^\ast}
        \dt v^\rL
    \big \|_{\rL^2_\tau\rH^1(\Omega)}
    \leq C\eps^2 .
\]
For the horizontal transport term, using the product rule, we have
\[
\begin{aligned}
    &\quad
    \big \|
        \nabla
        \big[
            \frac{
                (\bar\rho_{\rd}^\rL+\bar\rho_{\rd}^\ast)\hat B^\rL(1+Q_{\rm}^\rL)
            }{
                \bar\rho_{\rd}^\ast\hat B^\ast
            }
            \mathrm Z\tilde v^\rL\cdot\nablaH v^\rL
        \big]
    \big \|_{\rL^2_\tau\rL^2(\Omega)}
    \\
    &\leq
    C
    \big \|
        \nabla
        \big[
            \frac{
                (\bar\rho_{\rd}^\rL+\bar\rho_{\rd}^\ast)\hat B^\rL(1+Q_{\rm}^\rL)
            }{
                \bar\rho_{\rd}^\ast\hat B^\ast
            }
        \big]
    \big \|_{\rL^\infty_\tau\rH^1(\Omega)}
    \|\mathrm Z\|_{\rL^\infty_\tau\rH^2(\T^2)}
    \|\tilde v^\rL\|_{\rL^\infty_\tau\rH^2(\Omega)}
    \|\nablaH v^\rL\|_{\rL^2_\tau\rH^1(\Omega)}
    \\
    &\quad
    +
    C
    \big \|
        \frac{
            (\bar\rho_{\rd}^\rL+\bar\rho_{\rd}^\ast)\hat B^\rL(1+Q_{\rm}^\rL)
        }{
            \bar\rho_{\rd}^\ast\hat B^\ast
        }
    \big \|_{\rL^\infty_\tau\rH^2(\Omega)}
    \|\nabla\mathrm Z\|_{\rL^\infty_\tau\rH^1(\T^2)}
    \|\tilde v^\rL\|_{\rL^\infty_\tau\rH^2(\Omega)}
    \|\nablaH v^\rL\|_{\rL^2_\tau\rH^1(\Omega)}
    \\
    &\quad
    +
    C
    \big \|
        \frac{
            (\bar\rho_{\rd}^\rL+\bar\rho_{\rd}^\ast)\hat B^\rL(1+Q_{\rm}^\rL)
        }{
            \bar\rho_{\rd}^\ast\hat B^\ast
        }
    \big \|_{\rL^\infty_\tau\rH^2(\Omega)}
    \|\mathrm Z\|_{\rL^\infty_\tau\rH^2(\T^2)}
    \|\nabla\tilde v^\rL\|_{\rL^\infty_\tau\rH^1(\Omega)}
    \|\nablaH v^\rL\|_{\rL^2_\tau\rH^1(\Omega)}
    \\
    &\quad
    +
    C
    \big \|
        \frac{
            (\bar\rho_{\rd}^\rL+\bar\rho_{\rd}^\ast)\hat B^\rL(1+Q_{\rm}^\rL)
        }{
            \bar\rho_{\rd}^\ast\hat B^\ast
        }
    \big \|_{\rL^\infty_\tau\rH^2(\Omega)}
    \|\mathrm Z\|_{\rL^\infty_\tau\rH^2(\T^2)}
    \|\tilde v^\rL\|_{\rL^\infty_\tau\rH^2(\Omega)}
    \|\nabla\nablaH v^\rL\|_{\rL^2_\tau\rL^2(\Omega)}
    \\
    &\leq C\eps^2 .
\end{aligned}
\]
Together with the lower-order estimate
\[
\begin{aligned}
    &\quad
    \big \|
        \frac{
            (\bar\rho_{\rd}^\rL+\bar\rho_{\rd}^\ast)\hat B^\rL(1+Q_{\rm}^\rL)
        }{
            \bar\rho_{\rd}^\ast\hat B^\ast
        }
        \mathrm Z\tilde v^\rL\cdot\nablaH v^\rL
    \big \|_{\rL^2_\tau\rL^2(\Omega)}
    \\
    &\leq
    C
    \big \|
        \frac{
            (\bar\rho_{\rd}^\rL+\bar\rho_{\rd}^\ast)\hat B^\rL(1+Q_{\rm}^\rL)
        }{
            \bar\rho_{\rd}^\ast\hat B^\ast
        }
    \big \|_{\rL^\infty_\tau\rH^2(\Omega)}
    \|\mathrm Z\|_{\rL^\infty_\tau\rH^2(\T^2)}
    \|\tilde v^\rL\|_{\rL^\infty_\tau\rH^2(\Omega)}
    \|\nablaH v^\rL\|_{\rL^2_\tau\rL^2(\Omega)}
    \leq C\eps^2 ,
\end{aligned}
\]
we obtain
\[
\begin{aligned}
    &\quad
    \big \|
        \frac{
            (\bar\rho_{\rd}^\rL+\bar\rho_{\rd}^\ast)\hat B^\rL(1+Q_{\rm}^\rL)
        }{
            \bar\rho_{\rd}^\ast\hat B^\ast
        }
        \mathrm Z\tilde v^\rL\cdot\nablaH v^\rL
    \big \|_{\rL^2_\tau\rH^1(\Omega)}
    \leq C\eps^2 .
\end{aligned}
\]
For the vertical transport term, we have estimating anisotropically
\[
\begin{aligned}
    \big \|
        \nabla
        \big[
            \frac{
                (\bar\rho_{\rd}\hat B w)^\rL(1+Q_{\rm}^\rL)
            }{
                \bar\rho_{\rd}^\ast\hat B^\ast
            }
            \del_z v^\rL
        \big]
    \big \|_{\rL^2_\tau\rL^2(\Omega)}
    &\leq
    C \big (
    \big \|
        \nabla
        \big[
            \frac{
                1+Q_{\rm}^\rL
            }{
                \bar\rho_{\rd}^\ast\hat B^\ast
            }
        \big]
    \big \|_{\rL^\infty_\tau\rH^1(\Omega)}
    \|(\bar\rho_{\rd}\hat B w)^\rL\|_{\rL^2_\tau\rH^1(\Omega)}
    \|\del_z v^\rL\|_{\rL^\infty_\tau\rH^1(\Omega)}
    \\
    &\quad
    +
    \big \|
        \frac{
            1+Q_{\rm}^\rL
        }{
            \bar\rho_{\rd}^\ast\hat B^\ast
        }
    \big \|_{\rL^\infty_\tau\rH^2(\Omega)}
    \|\nabla(\bar\rho_{\rd}\hat B w)^\rL\|_{\rL^2_\tau\rL^2_\rH \rL^\infty_z}
    \|\del_z v^\rL\|_{\rL^\infty_\tau \rL^\infty_\rH \rL^2_z}
    \\
    &\quad
    +
    \big \|
        \frac{
            1+Q_{\rm}^\rL
        }{
            \bar\rho_{\rd}^\ast\hat B^\ast
        }
    \big \|_{\rL^\infty_\tau\rH^2(\Omega)}
    \|(\bar\rho_{\rd}\hat B w)^\rL\|_{\rL^2_\tau\rH^1(\Omega)}
    \|\nabla\del_z v^\rL\|_{\rL^\infty_\tau\rL^2(\Omega)} \big )
    \\
    &\leq C\eps^2,
\end{aligned}
\]
where we used the embedding
\begin{equation*}
    \rH^2_\tau \rL^2 \cap \rL^2_\tau \rH^3 \hookrightarrow \rL^\infty_\tau \rH^{\frac{9}{4}} \hookrightarrow \rL^\infty_\tau \rL^\infty_\rH \rH^1_z
\end{equation*}
as well as 
\begin{equation*}
     \|\nabla(\bar\rho_{\rd}\hat B w)^\rL\|_{\rL^2_\tau\rL^2_\rH \rL^\infty_z} \leq C \| \dz (\bar\rho_{\rd}\hat B w)^\rL\|_{\rL^2_\tau \rH^1(\Omega)} \leq C\eps.
\end{equation*}
The lower-order term is estimated in the same way. Indeed,
\[
\begin{aligned}
    \big \|
        \frac{
            (\bar\rho_{\rd}\hat B w)^\rL(1+Q_{\rm}^\rL)
        }{
            \bar\rho_{\rd}^\ast\hat B^\ast
        }
        \del_z v^\rL
    \big \|_{\rL^2_\tau\rL^2(\Omega)}
    \leq
    C
    \big \|
        \frac{
            1+Q_{\rm}^\rL
        }{
            \bar\rho_{\rd}^\ast\hat B^\ast
        }
    \big \|_{\rL^\infty_\tau\rH^2(\Omega)}
    \|(\bar\rho_{\rd}\hat B w)^\rL\|_{\rL^2_\tau\rH^1(\Omega)}
    \|\del_z v^\rL\|_{\rL^\infty_\tau\rH^1(\Omega)}
\leq C\eps^2 .
\end{aligned}
\]
Thus,
\[
\begin{aligned}
    &\quad
    \big \|
        \frac{
            (\bar\rho_{\rd}\hat B w)^\rL(1+Q_{\rm}^\rL)
        }{
            \bar\rho_{\rd}^\ast\hat B^\ast
        }
        \del_z v^\rL
    \big \|_{\rL^2_\tau\rH^1(\Omega)}
    \leq C\eps^2 .
\end{aligned}
\]
We now estimate the transformed viscous terms. 
Using the product rule, we obtain
\[
\begin{aligned}
    \nabla
    \bigl((\cL_1-\Delta)v^\rL\bigr)_i
    &=
    \sum_{j,k,l}
    \nabla
    \frac{\partial^2 v_i^\rL}{\partial y_k\partial y_l}
    \bigl(\mathrm Z_{k,j}-\delta_{k,j}\bigr)\mathrm Z_{l,j}
    +
    \sum_{j,k,l}
    \frac{\partial^2 v_i^\rL}{\partial y_k\partial y_l}
    \nabla\mathrm Z_{k,j}\,\mathrm Z_{l,j}
    +
    \sum_{j,k,l}
    \frac{\partial^2 v_i^\rL}{\partial y_k\partial y_l}
    \bigl(\mathrm Z_{k,j}-\delta_{k,j}\bigr)
    \nabla\mathrm Z_{l,j}
    \\
    &\quad
    +
    \sum_{k,l}
    \nabla
    \frac{\partial^2 v_i^\rL}{\partial y_k\partial y_l}
    \bigl(\mathrm Z_{l,k}-\delta_{l,k}\bigr)
    +
    \sum_{k,l}
    \frac{\partial^2 v_i^\rL}{\partial y_k\partial y_l}
    \nabla\mathrm Z_{l,k}
    +
    \sum_{j,k,l}
    \nabla\mathrm Z_{l,j}
    \frac{\partial v_i^\rL}{\partial y_k}
    \frac{\partial \mathrm Z_{k,j}}{\partial y_l}
  \\
    &\quad   +
    \sum_{j,k,l}
    \mathrm Z_{l,j}
    \nabla
    \frac{\partial v_i^\rL}{\partial y_k}
    \frac{\partial \mathrm Z_{k,j}}{\partial y_l}
    +
    \sum_{j,k,l}
    \mathrm Z_{l,j}
    \frac{\partial v_i^\rL}{\partial y_k}
    \nabla
    \frac{\partial \mathrm Z_{k,j}}{\partial y_l}.
\end{aligned}
\]
Therefore,
\[
\begin{aligned}
    &\quad
    \big \|
        \nabla
        \bigl((\cL_1-\Delta)v^\rL\bigr)_i
    \big \|_{\rL^2_\tau\rL^2(\Omega)}
    \\
    &\leq C \big(
    \|v^\rL\|_{\rL^2_\tau\rH^3(\Omega)}
    \|\mathrm Z-\mathrm I_2\|_{\rL^\infty_\tau\rH^2(\T^2)}
    \|\mathrm Z\|_{\rL^\infty_\tau\rH^2(\T^2)}
    +
    \|v^\rL\|_{\rL^2_\tau\rH^3(\Omega)}
    \|\nablaH\mathrm Z\|_{\rL^\infty_\tau\rH^1(\T^2)}
    \|\mathrm Z\|_{\rL^\infty_\tau\rH^2(\T^2)}
    \\
    &\quad
    +
    \|v^\rL\|_{\rL^2_\tau\rH^3(\Omega)}
    \|\mathrm Z-\mathrm I_2\|_{\rL^\infty_\tau\rH^2(\T^2)}
    \|\nablaH\mathrm Z\|_{\rL^\infty_\tau\rH^1(\T^2)}
    +
    \|v^\rL\|_{\rL^2_\tau\rH^3(\Omega)}
    \|\mathrm Z-\mathrm I_2\|_{\rL^\infty_\tau\rH^2(\T^2)}
    \\
    &\quad
    +
    \|v^\rL\|_{\rL^2_\tau\rH^3(\Omega)}
    \|\nablaH\mathrm Z\|_{\rL^\infty_\tau\rH^1(\T^2)}
    +
    \|\mathrm Z\|_{\rL^\infty_\tau\rH^2(\T^2)}
    \|v^\rL\|_{\rL^2_\tau\rH^3(\Omega)}
    \|\nablaH\mathrm Z\|_{\rL^\infty_\tau\rH^1(\T^2)} \big )
    \\
    &\leq C\eps^2 .
\end{aligned}
\]
The lower-order \(\rL^2_\tau\rL^2\)-part is estimated in the same way. Hence
\[
\begin{aligned}
    \big \|
        \frac{\mu}{\bar\rho_{\rd}^\ast\hat B^\ast}
        (\cL_1-\Delta)v^\rL
    \big \|_{\rL^2_\tau\rH^1(\Omega)}
    \leq C\eps^2 .
\end{aligned}
\]
We only write the estimate for the \(\cL_1-\Delta\) term. The term
\(\cL_2-\nablaH\divH\) has the same structure and is estimated in the same
way. The transformed viscous operators occurring in the temperature and
moisture equations are treated identically, and we omit them here.
We now consider the pressure contribution. We only estimate one representative
term from the nonlinear pressure remainder, namely
\[
    \frac{1}{\bar\rho_{\rd}^\ast\hat B^\ast}
    \nablaH
    \big(
        \bar\rho_{\rd}^\rL \hat B^\ast T^\rL
    \big).
\]
The other terms in the pressure remainder are estimated in the same way.
Using the product rule and since \(\hat B^\ast\) is a fixed smooth multiplier, we have 
\[
\begin{aligned}
    &\quad
    \big \|
        \nabla
        \big[
            \frac{1}{\bar\rho_{\rd}^\ast\hat B^\ast}
            \nablaH
            \big(
                \bar\rho_{\rd}^\rL \hat B^\ast T^\rL
            \big)
        \big]
    \big \|_{\rL^2_\tau\rL^2(\Omega)}
    \\
    &\leq
    C\big ( 
    \|\nabla\nablaH\bar\rho_{\rd}^\rL\|_{\rL^\infty_\tau\rL^2(\T^2)}
    \|T^\rL\|_{\rL^2_\tau\rL^\infty(\Omega)}
    +
    \|\nablaH\bar\rho_{\rd}^\rL\|_{\rL^\infty_\tau\rL^6(\T^2)}
    \|T^\rL\|_{\rL^2_\tau\rL^3(\Omega)}
    +
    \|\nablaH\bar\rho_{\rd}^\rL\|_{\rL^\infty_\tau\rL^6(\T^2)}
    \|\nabla T^\rL\|_{\rL^2_\tau\rL^3(\Omega)}
    \\
    &\quad
    +
    \|\nablaH\bar\rho_{\rd}^\rL\|_{\rL^\infty_\tau\rL^6(\T^2)}
    \|\nablaH T^\rL\|_{\rL^2_\tau\rL^3(\Omega)}
    +
    \|\bar\rho_{\rd}^\rL\|_{\rL^\infty_\tau\rL^\infty(\T^2)}
    \|T^\rL\|_{\rL^2_\tau\rH^2(\Omega)}
    +
    \|\bar\rho_{\rd}^\rL\|_{\rL^\infty_\tau\rH^2(\T^2)}
    \|T^\rL\|_{\rL^2_\tau\rH^2(\Omega)} \big )
    \\
    &\leq C\eps^2 .
\end{aligned}
\]
The lower-order \(\rL^2_\tau\rL^2\)-part is estimated in the same way. Therefore,
\[
\begin{aligned}
    \big \|
        \frac{1}{\bar\rho_{\rd}^\ast\hat B^\ast}
        \nablaH
        \big(
            \bar\rho_{\rd}^\rL \hat B^\ast T^\rL
        \big)
    \big \|_{\rL^2_\tau\rH^1(\Omega)}
    \leq C\eps^2 .
\end{aligned}
\]
All other quadratic pressure terms, as well as the geometric pressure term
containing \(\mathrm Z^\top-\mathrm I_2\), are treated by the same product-rule
estimate.
We now estimate the \(\rH^1_\tau\rL^2\)-part. We again start with the
contribution containing \(\dt v^\rL\). As before, we only treat one
representative term, namely
\[
    \frac{\bar\rho_{\rd}^\rL}{\bar\rho_{\rd}^\ast}\dt v^\rL .
\]
Using the product rule,
\[
    \dt
    \bigg(
        \frac{\bar\rho_{\rd}^\rL}{\bar\rho_{\rd}^\ast}
        \dt v^\rL
    \bigg)
    =
    \frac{\dt\bar\rho_{\rd}^\rL}{\bar\rho_{\rd}^\ast}\dt v^\rL
    +
    \frac{\bar\rho_{\rd}^\rL}{\bar\rho_{\rd}^\ast}\dt^2 v^\rL .
\]
Therefore,
\[
\begin{aligned}
    \big \|
        \dt
        \big[
            \frac{\bar\rho_{\rd}^\rL}{\bar\rho_{\rd}^\ast}
            \dt v^\rL
        \big]
    \big \|_{\rL^2_\tau\rL^2(\Omega)}
    &\leq
    C \big (
    \|\dt\bar\rho_{\rd}^\rL\|_{\rL^2_\tau\rL^6(\T^2)}
    \|\dt v^\rL\|_{\rL^\infty_\tau\rL^3(\Omega)}
    +
    \|\bar\rho_{\rd}^\rL\|_{\rL^\infty_\tau\rL^\infty(\T^2)}
    \|\dt^2 v^\rL\|_{\rL^2_\tau\rL^2(\Omega)} \big )
    \\
    &\leq
    C \big (
    \|\dt\bar\rho_{\rd}^\rL\|_{\rL^2_\tau\rH^1(\T^2)}
    \|\dt v^\rL\|_{\rL^\infty_\tau\rH^1(\Omega)}
    +
    \|\bar\rho_{\rd}^\rL\|_{\rL^\infty_\tau\rH^2(\T^2)}
    \|\dt^2 v^\rL\|_{\rL^2_\tau\rL^2(\Omega)} \big )
    \\
    &\leq C\eps^2 .
\end{aligned}
\]
All remaining terms in the coefficient of \(\dt v^\rL\) are treated in the
same way. 
For the horizontal transport term, using the product rule, we have

\[
\begin{aligned}
    &\quad
    \big \|
        \dt
        \bigg[
            \frac{
                (\bar\rho_{\rd}^\rL+\bar\rho_{\rd}^\ast)\hat B^\rL(1+Q_{\rm}^\rL)
            }{
                \bar\rho_{\rd}^\ast\hat B^\ast
            }
            \mathrm Z\tilde v^\rL\cdot\nablaH v^\rL
        \bigg]
    \big \|_{\rL^2_\tau\rL^2(\Omega)}
    \\
    &\leq
    C \big (
    \big \|
        \dt
        \bigg[
            \frac{
                (\bar\rho_{\rd}^\rL+\bar\rho_{\rd}^\ast)\hat B^\rL(1+Q_{\rm}^\rL)
            }{
                \bar\rho_{\rd}^\ast\hat B^\ast
            }
        \bigg]
    \big \|_{\rL^2_\tau\rH^1(\Omega)}
    \|\mathrm Z\|_{\rL^\infty_\tau\rH^2(\T^2)}
    \|\tilde v^\rL\|_{\rL^\infty_\tau\rH^2(\Omega)}
    \|\nablaH v^\rL\|_{\rL^\infty_\tau\rH^1(\Omega)}
    \\
    &\quad
    +
    \big \|
        \frac{
            (\bar\rho_{\rd}^\rL+\bar\rho_{\rd}^\ast)\hat B^\rL(1+Q_{\rm}^\rL)
        }{
            \bar\rho_{\rd}^\ast\hat B^\ast
        }
    \big \|_{\rL^\infty_\tau\rH^2(\Omega)}
    \|\dt\mathrm Z\|_{\rL^2_\tau\rH^1(\T^2)}
    \|\tilde v^\rL\|_{\rL^\infty_\tau\rH^2(\Omega)}
    \|\nablaH v^\rL\|_{\rL^\infty_\tau\rH^1(\Omega)}
    \\
    &\quad
    +
    \big \|
        \frac{
            (\bar\rho_{\rd}^\rL+\bar\rho_{\rd}^\ast)\hat B^\rL(1+Q_{\rm}^\rL)
        }{
            \bar\rho_{\rd}^\ast\hat B^\ast
        }
    \big \|_{\rL^\infty_\tau\rH^2(\Omega)}
    \|\mathrm Z\|_{\rL^\infty_\tau\rH^2(\T^2)}
    \|\dt\tilde v^\rL\|_{\rL^2_\tau\rH^1(\Omega)}
    \|\nablaH v^\rL\|_{\rL^\infty_\tau\rH^1(\Omega)}
    \\
    &\quad
    +
    \big \|
        \frac{
            (\bar\rho_{\rd}^\rL+\bar\rho_{\rd}^\ast)\hat B^\rL(1+Q_{\rm}^\rL)
        }{
            \bar\rho_{\rd}^\ast\hat B^\ast
        }
    \big \|_{\rL^\infty_\tau\rH^2(\Omega)}
    \|\mathrm Z\|_{\rL^\infty_\tau\rH^2(\T^2)}
    \|\tilde v^\rL\|_{\rL^\infty_\tau\rH^2(\Omega)}
    \|\nablaH\dt v^\rL\|_{\rL^2_\tau\rH^1(\Omega)} \big )
    \\
    &\leq C\eps^2 .
\end{aligned}
\]
Together with the previous \(\rL^2_\tau\rL^2\)-estimate, we obtain
\[
\begin{aligned}
    &\quad
    \big \|
        \frac{
            (\bar\rho_{\rd}^\rL+\bar\rho_{\rd}^\ast)\hat B^\rL(1+Q_{\rm}^\rL)
        }{
            \bar\rho_{\rd}^\ast\hat B^\ast
        }
        \mathrm Z\tilde v^\rL\cdot\nablaH v^\rL
    \big \|_{\rH^1_\tau\rL^2(\Omega)}
    \leq C\eps^2 .
\end{aligned}
\]
For the vertical transport term, we have
\[
\begin{aligned}
    \dt
    \bigg[
        \frac{
            (\bar\rho_{\rd}\hat B w)^\rL(1+Q_{\rm}^\rL)
        }{
            \bar\rho_{\rd}^\ast\hat B^\ast
        }
        \del_z v^\rL
    \bigg]
&=
    \dt
    \bigg[
        \frac{
            1+Q_{\rm}^\rL
        }{
            \bar\rho_{\rd}^\ast\hat B^\ast
        }
    \bigg]
    (\bar\rho_{\rd}\hat B w)^\rL
    \del_z v^\rL
    +
    \frac{
        1+Q_{\rm}^\rL
    }{
        \bar\rho_{\rd}^\ast\hat B^\ast
    }
    \dt(\bar\rho_{\rd}\hat B w)^\rL
    \del_z v^\rL
\\ &\quad    +
    \frac{
        1+Q_{\rm}^\rL
    }{
        \bar\rho_{\rd}^\ast\hat B^\ast
    }
    (\bar\rho_{\rd}\hat B w)^\rL
    \del_z\dt v^\rL .
\end{aligned}
\]
Therefore, estimating anisotropically, we obtain 
\[
\begin{aligned}
    \big \|
        \dt
        \bigg[
            \frac{
                (\bar\rho_{\rd}\hat B w)^\rL(1+Q_{\rm}^\rL)
            }{
                \bar\rho_{\rd}^\ast\hat B^\ast
            }
            \del_z v^\rL
        \bigg]
    \big \|_{\rL^2_\tau\rL^2(\Omega)}
    &\leq
    C \big (
    \big \|
        \dt
        \bigg[
            \frac{
                1+Q_{\rm}^\rL
            }{
                \bar\rho_{\rd}^\ast\hat B^\ast
            }
        \bigg]
    \big \|_{\rL^2_\tau\rH^2(\Omega)}
    \|(\bar\rho_{\rd}\hat B w)^\rL\|_{\rL^\infty_\tau\rL^3(\Omega)}
    \|\del_z v^\rL\|_{\rL^\infty_\tau\rH^1(\Omega)}
    \\
    &\quad
    +
    \big \|
        \frac{
            1+Q_{\rm}^\rL
        }{
            \bar\rho_{\rd}^\ast\hat B^\ast
        }
    \big \|_{\rL^\infty_\tau\rH^2(\Omega)}
    \|\dt(\bar\rho_{\rd}\hat B w)^\rL\|_{\rL^2_\tau \rL^2_\rH \rL^\infty_z}
    \|\del_z v^\rL\|_{\rL^\infty_\tau \rL^\infty_\rH \rL^2_z}
    \\
    &\quad
    +
    \big \|
        \frac{
            1+Q_{\rm}^\rL
        }{
            \bar\rho_{\rd}^\ast\hat B^\ast
        }
    \big \|_{\rL^\infty_\tau\rH^2(\Omega)}
    \|(\bar\rho_{\rd}\hat B w)^\rL\|_{\rL^\infty_\tau\rL^3(\Omega)}
    \|\del_z\dt v^\rL\|_{\rL^2_\tau\rH^1(\Omega)} \big )
    \\
    &\leq C\eps^2,
\end{aligned}
\]
where we used the embedding
\begin{equation*}
    \rH^2_\tau \rL^2 (\Omega)\cap \rL^2_\tau \rH^3 \hookrightarrow \rL^\infty_\tau \rH^{\frac{9}{4}}(\Omega) \hookrightarrow \rL^\infty_\tau \rL^\infty_\rH \rH^1_z
\end{equation*}
as well as 

\begin{equation*}
     \|\dt(\bar\rho_{\rd}\hat B w)^\rL\|_{\rL^2_\tau\rL^2_\rH \rL^\infty_z} \leq C \| \dz (\bar\rho_{\rd}\hat B w)^\rL\|_{\rH^1_\tau \rL^2(\Omega)} \leq C\eps.
\end{equation*}
Together with the previous \(\rL^2_\tau\rL^2\)-estimate, this gives
\[
\begin{aligned}
    &\quad
    \big \|
        \frac{
            (\bar\rho_{\rd}\hat B w)^\rL(1+Q_{\rm}^\rL)
        }{
            \bar\rho_{\rd}^\ast\hat B^\ast
        }
        \del_z v^\rL
    \big \|_{\rH^1_\tau\rL^2(\Omega)}
    \leq C\eps^2 .
\end{aligned}
\]
Consequently,
\[
\begin{aligned}
    &\quad
    \big \|
        \frac{
            (\bar\rho_{\rd}^\rL+\bar\rho_{\rd}^\ast)\hat B^\rL(1+Q_{\rm}^\rL)
        }{
            \bar\rho_{\rd}^\ast\hat B^\ast
        }
        \mathrm Z\tilde v^\rL\cdot\nablaH v^\rL
        +
        \frac{
            (\bar\rho_{\rd}\hat B w)^\rL(1+Q_{\rm}^\rL)
        }{
            \bar\rho_{\rd}^\ast\hat B^\ast
        }
        \del_z v^\rL
    \big \|_{\rH^1_\tau\rL^2(\Omega)}
    \leq C\eps^2 .
\end{aligned}
\]
We next estimate the transformed viscous terms. As before, we only write the
estimate for \(\cL_1-\Delta\). The term \(\cL_2-\nablaH\divH\) has the same
structure and is estimated in the same way. The transformed viscous operators
in the temperature and moisture equations are treated identically.

Using the product rule, we obtain
\[
\begin{aligned}
    &\quad
    \big \|
        \dt
        \bigl((\cL_1-\Delta)v^\rL\bigr)_i
    \big \|_{\rL^2_\tau\rL^2(\Omega)}
    \\
    &\leq
    C\big (
    \|\dt v^\rL\|_{\rL^2_\tau\rH^2(\Omega)}
    \|\mathrm Z-\mathrm I_2\|_{\rL^\infty_\tau\rH^2(\T^2)}
    \|\mathrm Z\|_{\rL^\infty_\tau\rH^2(\T^2)}
    +
    \|v^\rL\|_{\rL^\infty_\tau\rH^2(\Omega)}
    \|\dt\mathrm Z\|_{\rL^2_\tau\rH^2(\T^2)}
    \|\mathrm Z\|_{\rL^\infty_\tau\rH^2(\T^2)}
    \\
    &\quad
    +
    \|v^\rL\|_{\rL^\infty_\tau\rH^2(\Omega)}
    \|\mathrm Z-\mathrm I_2\|_{\rL^\infty_\tau\rH^2(\T^2)}
    \|\dt\mathrm Z\|_{\rL^2_\tau\rH^2(\T^2)}
    +
    \|\dt v^\rL\|_{\rL^2_\tau\rH^2(\Omega)}
    \|\mathrm Z-\mathrm I_2\|_{\rL^\infty_\tau\rH^2(\T^2)}
    \\
    &\quad +
    \|v^\rL\|_{\rL^\infty_\tau\rH^2(\Omega)}
    \|\dt\mathrm Z\|_{\rL^2_\tau\rH^2(\T^2)}
    +
    \|\dt\mathrm Z\|_{\rL^2_\tau\rH^2(\T^2)}
    \|v^\rL\|_{\rL^\infty_\tau\rH^2(\Omega)}
    \|\nablaH\mathrm Z\|_{\rL^\infty_\tau\rH^1(\T^2)}
    \\
    &\quad
    +
    \|\mathrm Z\|_{\rL^\infty_\tau\rH^2(\T^2)}
    \|\dt v^\rL\|_{\rL^2_\tau\rH^2(\Omega)}
    \|\nablaH\mathrm Z\|_{\rL^\infty_\tau\rH^1(\T^2)}
    +
    \|\mathrm Z\|_{\rL^\infty_\tau\rH^2(\T^2)}
    \|v^\rL\|_{\rL^\infty_\tau\rH^2(\Omega)}
    \|\nablaH\dt\mathrm Z\|_{\rL^2_\tau\rH^1(\T^2)}
    \big )
    \\
    &\leq C\eps^2 .
\end{aligned}
\]
Together with the previous \(\rL^2_\tau\rL^2\)-estimate, we get
\[
\begin{aligned}
    \big \|
        \frac{\mu}{\bar\rho_{\rd}^\ast\hat B^\ast}
        (\cL_1-\Delta)v^\rL
    \big \|_{\rH^1_\tau\rL^2(\Omega)}
    \leq C\eps^2 .
\end{aligned}
\]
As explained above, the same argument gives
\[
\begin{aligned}
    &\quad
    \big \|
        \frac{\mu}{\bar\rho_{\rd}^\ast\hat B^\ast}
        (\cL_1-\Delta)v^\rL
    \big \|_{\rH^1_\tau\rL^2(\Omega)}
    +
    \big \|
        \frac{\mu+\lambda}{\bar\rho_{\rd}^\ast\hat B^\ast}
        (\cL_2-\nablaH\divH)v^\rL
    \big \|_{\rH^1_\tau\rL^2(\Omega)}
    \leq C\eps^2 .
\end{aligned}
\]
We finally consider the pressure contribution. Again, we estimate only the
representative term
\[
    \frac{1}{\bar\rho_{\rd}^\ast\hat B^\ast}
    \nablaH
    \big(
        \bar\rho_{\rd}^\rL \hat B^\ast T^\rL
    \big).
\]
Using the product rule and the fact that \(\hat B^\ast\) is a fixed smooth
multiplier, we have
\[
\begin{aligned}
    \dt
    \big[
        \frac{1}{\bar\rho_{\rd}^\ast\hat B^\ast}
        \nablaH
        \big(
            \bar\rho_{\rd}^\rL \hat B^\ast T^\rL
        \big)
    \big]
    =
    \frac{1}{\bar\rho_{\rd}^\ast\hat B^\ast}
    \nablaH
    \big(
        \dt\bar\rho_{\rd}^\rL \hat B^\ast T^\rL
    \big)
    +
    \frac{1}{\bar\rho_{\rd}^\ast\hat B^\ast}
    \nablaH
    \big(
        \bar\rho_{\rd}^\rL \hat B^\ast \dt T^\rL
    \big).
\end{aligned}
\]
Therefore,
\[
\begin{aligned}
    &\quad
    \big \|
        \dt
        \big[
            \frac{1}{\bar\rho_{\rd}^\ast\hat B^\ast}
            \nablaH
            \big(
                \bar\rho_{\rd}^\rL \hat B^\ast T^\rL
            \big)
        \big]
    \big \|_{\rL^2_\tau\rL^2(\Omega)}
    \\
    &\leq
    C\big (
    \|\nablaH\dt\bar\rho_{\rd}^\rL\|_{\rL^2_\tau\rL^2(\T^2)}
    \|T^\rL\|_{\rL^\infty_\tau\rL^\infty(\Omega)}
    +
    \|\dt\bar\rho_{\rd}^\rL\|_{\rL^2_\tau\rL^6(\T^2)}
    \|T^\rL\|_{\rL^\infty_\tau\rL^3(\Omega)}
    \\
    &\quad
    +
    \|\dt\bar\rho_{\rd}^\rL\|_{\rL^2_\tau\rL^6(\T^2)}
    \|\nablaH T^\rL\|_{\rL^\infty_\tau\rL^3(\Omega)}
    +
    \|\nablaH\bar\rho_{\rd}^\rL\|_{\rL^\infty_\tau\rL^6(\T^2)}
    \|\dt T^\rL\|_{\rL^2_\tau\rL^3(\Omega)}
    \\
    &\quad
    +
    \|\bar\rho_{\rd}^\rL\|_{\rL^\infty_\tau\rL^\infty(\T^2)}
    \|\dt T^\rL\|_{\rL^2_\tau\rH^1(\Omega)}
    \big )
    \\
    &\leq
    C
    \big (
        \|\dt\bar\rho_{\rd}^\rL\|_{\rL^2_\tau\rH^1(\T^2)}
        \|T^\rL\|_{\rL^\infty_\tau\rH^2(\Omega)}
        +
        \|\bar\rho_{\rd}^\rL\|_{\rL^\infty_\tau\rH^2(\T^2)}
        \|\dt T^\rL\|_{\rL^2_\tau\rH^1(\Omega)}
    \big )
    \\
    &\leq C\eps^2 .
\end{aligned}
\]
Together with the previous \(\rL^2_\tau\rL^2\)-estimate, this yields
\[
\begin{aligned}
    \big \|
        \frac{1}{\bar\rho_{\rd}^\ast\hat B^\ast}
        \nablaH
        \big(
            \bar\rho_{\rd}^\rL \hat B^\ast T^\rL
        \big)
    \big \|_{\rH^1_\tau\rL^2(\Omega)}
    \leq C\eps^2 .
\end{aligned}
\]
All other quadratic pressure terms, as well as the geometric pressure term
containing \(\mathrm Z^\top-\mathrm I_2\), are treated by the same product-rule
estimate.
Combining the estimates above gives
\[
    \|G_v\|_{\E^v_0(0,\tau)}
    \leq C\eps^2 .
\]
We finally estimate one representative term containing the positive part. This
is the only additional point which does not occur in the purely smooth product
estimates above. We use that the map
\(
    f \mapsto f^+:=\max\{f,0\}
\)
is Lipschitz on \(\rL^2(\Omega)\) and stable on \(\rH^1(\Omega)\). Moreover,
\[
    \nabla f^+
    =
    \mathbf 1_{\{f>0\}}\nabla f
    \quad\text{a.e. in }\Omega .
\]
Thus, we obtain
\[
    \|(q^\rL_{\rv})^+\|_{\rH^1_\tau\rH^1(\Omega) \cap \rL^\infty_\tau\rL^\infty(\Omega)}
    \leq C\eps .
\]
We now estimate
\[
    (T^\ast+T^\rL)
    \frac{1+q_{\rv}^\rL}{1+Q_{\rm}^\rL}
    (-q_{\rv}^\rL)^+
    q_{\rr}^\rL .
\]
Note that
\[
    \|\frac{1+q_{\rv}^\rL}{1+Q_{\rm}^\rL}\|_{\rL^\infty_\tau\rH^2(\Omega)}
    +
    \|\dt \big [\frac{1+q_{\rv}^\rL}{1+Q_{\rm}^\rL} \big ]\|_{\rL^2_\tau\rH^1(\Omega)}
    \leq C .
\]
Setting $R^\rL:=\frac{1+q_{\rv}^\rL}{1+Q_{\rm}^\rL}$ and
using the product rule, we get
\[
\begin{aligned}
    \nabla
    \big[
        (T^\ast+T^\rL)
        R^\rL
        (q^\rL_{\rv})^+
        q_{\rr}^\rL
    \big]&=
    \nabla(T^\ast+T^\rL)
    R^\rL
    (q^\rL_{\rv})^+
    q_{\rr}^\rL
    +
    (T^\ast+T^\rL)
    \nabla R^\rL
    (q^\rL_{\rv})^+
    q_{\rr}^\rL
    \\
    &\quad
    +
    (T^\ast+T^\rL)
    R^\rL
    \nabla(q^\rL_{\rv})^+
    q_{\rr}^\rL
    +
    (T^\ast+T^\rL)
    R^\rL
    (q^\rL_{\rv})^+
    \nabla q_{\rr}^\rL .
\end{aligned}
\]
Therefore,
\[
\begin{aligned}
    \big \|
        \nabla
        \big[
            (T^\ast+T^\rL)
            R^\rL
            (q^\rL_{\rv})^+
            q_{\rr}^\rL
        \big]
    \big \|_{\rL^2_\tau\rL^2(\Omega)}
    &\leq
    C \big (
    \|\nabla(T^\ast+T^\rL)\|_{\rL^\infty_\tau\rH^1(\Omega)}
    \|R^\rL\|_{\rL^\infty_\tau\rH^2(\Omega)}
    \|(q^\rL_{\rv})^+\|_{\rL^2_\tau\rH^1(\Omega)}
    \|q_{\rr}^\rL\|_{\rL^\infty_\tau\rH^2(\Omega)}
    \\
    &\quad
    +
    \|T^\ast+T^\rL\|_{\rL^\infty_\tau\rH^2(\Omega)}
    \|\nabla R^\rL\|_{\rL^\infty_\tau\rH^1(\Omega)}
    \|(q^\rL_{\rv})^+\|_{\rL^2_\tau\rH^1(\Omega)}
    \|q_{\rr}^\rL\|_{\rL^\infty_\tau\rH^2(\Omega)}
    \\
    &\quad
    +
    \|T^\ast+T^\rL\|_{\rL^\infty_\tau\rH^2(\Omega)}
    \|R^\rL\|_{\rL^\infty_\tau\rH^2(\Omega)}
    \|\nabla(q^\rL_{\rv})^+\|_{\rL^2_\tau\rL^2(\Omega)}
    \|q_{\rr}^\rL\|_{\rL^\infty_\tau\rH^2(\Omega)}
    \\
    &\quad
    +
    \|T^\ast+T^\rL\|_{\rL^\infty_\tau\rH^2(\Omega)}
    \|R^\rL\|_{\rL^\infty_\tau\rH^2(\Omega)}
    \|(q^\rL_{\rv})^+\|_{\rL^2_\tau\rH^1(\Omega)}
    \|\nabla q_{\rr}^\rL\|_{\rL^\infty_\tau\rH^1(\Omega)} \big )
    \\
    &\leq C\eps^2 .
\end{aligned}
\]
The lower-order term is estimated in the same way. Hence
\[
\begin{aligned}
    \big \|
        (T^\ast+T^\rL)
        \frac{1+q_{\rv}^\rL}{1+Q_{\rm}^\rL}
        (-q_{\rv}^\rL)^+
        q_{\rr}^\rL
    \big \|_{\rL^2_\tau\rH^1(\Omega)}
    \leq C\eps^2 .
\end{aligned}
\]
The \(\rH^1_\tau\rL^2\)-part is estimated analogously.
Consequently,
\[
\begin{aligned}
    &\quad
    \big \|
        (T^\ast+T^\rL)
        \frac{1+q_{\rv}^\rL}{1+Q_{\rm}^\rL}
        (-q_{\rv}^\rL)^+
        q_{\rr}^\rL
    \big \|_{\rL^2_\tau\rH^1(\Omega)\cap\rH^1_\tau\rL^2(\Omega)}
    \leq C\eps^2 .
\end{aligned}
\]

All remaining terms containing positive parts are estimated in the same way.
Together with the estimates above for the time derivative contribution, the
transport terms, the transformed viscous terms, and the pressure terms, this
yields the claimed estimates.
\end{proof}
\noindent
Finally, we derive estimates on the switch terms, which were neglected in
\autoref{lem: est nonlinear moist}. In the cutoff regime considered above, the
saturation mixing ratio vanishes, and the relevant switch term is simply
\(Q^+\). We therefore consider the auxiliary problem in $\Omega_\tau$
\begin{equation}
    \label{eq: auxillary for q cutoff}
    \dt Q-\Delta Q+Q^+=f,
    \qquad
    Q(0)=Q_0 ,
\end{equation}
supplemented with homogeneous Neumann boundary conditions. We assume that
\(Q\in\mathbb E_1^h(0,\tau)\) is a local solution. Similar estimates in the $\rLp$-framework are established in \cite{BHMZ:26a}.

\begin{prop}
    \label{prop: est on q cutoff}
    Let \(\tau>0\), and assume that
    \(
        f\in\mathbb E_0^h(0,\tau).
    \)
    Moreover, assume that the initial datum and the compatibility datum
    \[
        Q_1
        :=
        \Delta Q_0
        +
        f(0)
        -
        Q_0^+
    \]
    satisfy
    \begin{equation}
        \label{eq: assu q cutoff}
        \|f\|_{\mathbb E_0^h(0,\tau)}
        +
        \|Q_0\|_{\rH^2(\Omega)}
        +
        \|Q_1\|_{\rH^1(\Omega)}
        \leq
        C_0\eps^2 .
    \end{equation}
    Then there exists a constant \(C>0\), depending on \(C_0\) and \(\tau\),
    such that
    \[
        \|Q\|_{\mathbb E_1^h(0,\tau)}
        \leq C\eps^2 .
    \]
\end{prop}

\begin{proof}
Multiplying  \eqref{eq: auxillary for q cutoff}  by \(Q\) and integrating in \(\Omega\), we obtain
\[
    \frac{1}{2}\dt\|Q\|_2^2
    +
    \|\nabla Q\|_2^2
    +
    \int_\Omega Q^+Q\,\rd x
    =
    \int_\Omega fQ\,\rd x .
\]
Since
\[
    Q^+Q=(Q^+)^2\geq 0,
\]
we have
\[
    \frac{1}{2}\dt\|Q\|_2^2
    +
    \|\nabla Q\|_2^2
    \leq
    \int_\Omega fQ\,\rd x .
\]
By Hölder's and Young's inequalities,
\[
    \frac{1}{2}\dt\|Q\|_2^2
    +
    \|\nabla Q\|_2^2
    \leq
    C\|f\|_2^2
    +
    C\|Q\|_2^2 .
\]
An integration in time, Gronwall's inequality, and the assumption on the data
yield
\[
    \|Q\|_{\rL^\infty_\tau\rL^2(\Omega)}
    +
    \|\nabla Q\|_{\rL^2_\tau\rL^2(\Omega)}
    \leq
    C\eps^2 .
\]
Next, multiply \eqref{eq: auxillary for q cutoff} by \(\dt Q\). Then
\[
    \|\dt Q\|_2^2
    +
    \frac{1}{2}\dt\|\nabla Q\|_2^2
    +
    \int_\Omega Q^+\dt Q\,\rd x
    =
    \int_\Omega f\dt Q\,\rd x .
\]
Since
\[
    \int_\Omega Q^+\dt Q\,\rd x
    =
    \frac{1}{2}\dt\|Q^+\|_2^2,
\]
we obtain
\[
    \|\dt Q\|_2^2
    +
    \frac{1}{2}\dt\|\nabla Q\|_2^2
    +
    \frac{1}{2}\dt\|Q^+\|_2^2
    =
    \int_\Omega f\dt Q\,\rd x .
\]
Similarly, multiplying \eqref{eq: auxillary for q cutoff} by \(-\Delta Q\)
and integrating by parts gives
\[
    \frac{1}{2}\dt\|\nabla Q\|_2^2
    +
    \|\Delta Q\|_2^2
    +
    \int_\Omega \nabla Q^+\cdot\nabla Q\,\rd x
    =
    -\int_\Omega f\Delta Q\,\rd x .
\]
Using
\[
    \nabla Q^+
    =
    \mathbf 1_{\{Q>0\}}\nabla Q
    \quad\text{a.e.},
\]
we have
\[
    \int_\Omega \nabla Q^+\cdot\nabla Q\,\rd x
    =
    \|\nabla Q^+\|_2^2
    \geq 0.
\]
Adding the two identities yields
\[
\begin{aligned}
    \|\dt Q\|_2^2
    +
    \|\Delta Q\|_2^2
    +
    \dt\|\nabla Q\|_2^2
    +
    \frac{1}{2}\dt\|Q^+\|_2^2
    +
    \|\nabla Q^+\|_2^2
    =
    \int_\Omega
    f
    \bigl(\dt Q-\Delta Q\bigr)
    \,\rd x .
\end{aligned}
\]
By Hölder's and Young's inequalities,
\[
\begin{aligned}
    \big |
    \int_\Omega
    f
    \bigl(\dt Q-\Delta Q\bigr)
    \,\rd x
    \big |
\leq
    \frac{1}{2}\|\dt Q\|_2^2
    +
    \frac{1}{2}\|\Delta Q\|_2^2
    +
    C\|f\|_2^2 .
\end{aligned}
\]
Hence, after integration in time,
\[
    \|Q\|_{\rL^2_\tau\rH^2(\Omega)}
    +
    \|Q\|_{\rH^1_\tau\rL^2(\Omega)}
    \leq
    C\eps^2 .
\]
It remains to estimate the switch term in \(\mathbb E_0^h(0,\tau)\). The
positive part map is Lipschitz and satisfies
\[
    \nabla Q^+
    =
    \mathbf 1_{\{Q>0\}}\nabla Q
    \quad\text{a.e.}
\]
Hence
\[
    \|Q^+\|_{\rL^2_\tau\rH^1(\Omega)}
    \leq
    \|Q\|_{\rL^2_\tau\rH^1(\Omega)}
    \leq
    C\eps^2 .
\]
Similarly,
\[
    \dt Q^+
    =
    \mathbf 1_{\{Q>0\}}\dt Q
    \quad\text{a.e.},
\]
and therefore
\[
    \|\dt Q^+\|_{\rL^2_\tau\rL^2(\Omega)}
    \leq
    \|\dt Q\|_{\rL^2_\tau\rL^2(\Omega)}
    \leq
    C\eps^2 .
\]
Combining these two estimates gives
\[
    \|Q^+\|_{\mathbb E_0^h(0,\tau)}
    =
    \|Q^+\|_{\rL^2_\tau\rH^1(\Omega)\cap\rH^1_\tau\rL^2(\Omega)}
    \leq
    C\eps^2 .
\]
Finally, rewriting \eqref{eq: auxillary for q cutoff} as
\[
    \dt Q-\Delta Q
    =
    f-Q^+,
\]
\autoref{lem: linear max reg} yields
\[
\begin{aligned}
    \|Q\|_{\mathbb E_1^h(0,\tau)}
    &\leq
    C
    \bigl(
        \|f-Q^+\|_{\mathbb E_0^h(0,\tau)}
        +
        \|Q_0\|_{\rH^2(\Omega)}
        +
        \|Q_1\|_{\rH^1(\Omega)}
    \bigr)
    \\
    &\leq
    C
    \bigl(
        \|f\|_{\mathbb E_0^h(0,\tau)}
        +
        \|Q^+\|_{\mathbb E_0^h(0,\tau)}
        +
        \|Q_0\|_{\rH^2(\Omega)}
        +
        \|Q_1\|_{\rH^1(\Omega)}
    \bigr)
    \\
    &\leq C\eps^2 .
\end{aligned}
\]
This proves the claim.
\end{proof}
\noindent
We observe that $s\mapsto s^+ $ is Lipschitz continuous but not differentiable at $s=0$. Thus, the term $Q^+$ requires special care in the contraction argument.
 The following corollary provides the required stability estimate.

\begin{cor}
    \label{cor: difference est q cutoff}
    Let \(\tau>0\), and let \(Q^{(i)}\in\mathbb E_1^h(0,\tau)\),
    \(i=1,2\), be local solutions of
    \[
         \dt Q^{(i)}-\Delta Q^{(i)}+\bigl(Q^{(i)}\bigr)^+
        =
         f_i,
         \qquad
         Q^{(i)}(0)=Q_0,
     \]
     supplemented with homogeneous Neumann boundary conditions. Set
     \[
         \bar Q:=Q^{(1)}-Q^{(2)},
        \qquad
         \bar f:=f_1-f_2, .
    \]
    Then there exists a constant \(C>0\), depending on \(\tau\), such that
    \[
         \|\bar Q\|_{\rH^1_\tau\rL^2(\Omega)}
         +
         \|\bar Q\|_{\rL^2_\tau\rH^2(\Omega)}
         +
         \|\bar Q\|_{\rL^\infty_\tau\rH^1(\Omega)}
         \leq
         C
             \|\bar f\|_{\rL^2_\tau\rL^2(\Omega)}
           .
    \]
 \end{cor}

 \begin{proof}
 Subtracting the two equations gives
 \[
     \dt \bar Q-\Delta \bar Q
     +
     \bigl(Q^{(1)}\bigr)^+
     -
     \bigl(Q^{(2)}\bigr)^+
     =
    \bar f,
     \qquad
     \bar Q(0)=0 .
 \]
 Testing this equation with \(\bar Q\), we obtain
 \[
     \frac12\dt\|\bar Q\|_2^2
     +
    \|\nabla\bar Q\|_2^2
     +
     \int_\Omega
         \Big(
             \bigl(Q^{(1)}\bigr)^+
             -
             \bigl(Q^{(2)}\bigr)^+
         \Big)
         \bar Q
     \,\rd x
     =
     \int_\Omega \bar f \,\bar Q\,\rd x .
 \]
 Since the positive part is monotone, we have
 \[
     \Big(
        \bigl(Q^{(1)}\bigr)^+
         -
         \bigl(Q^{(2)}\bigr)^+
     \Big)
     \bigl(Q^{(1)}-Q^{(2)}\bigr)
     \geq 0 .
 \]
 Hence
 \[
     \frac12\dt\|\bar Q\|_2^2
     +
     \|\nabla\bar Q\|_2^2
     \leq
     \int_\Omega \bar f\bar Q\,\rd x .
 \]
 By Hölder's and Young's inequalities,
 \[
     \frac12\dt\|\bar Q\|_2^2
     +
     \|\nabla\bar Q\|_2^2
     \leq
     C\|\bar f\|_2^2
     +
    C\|\bar Q\|_2^2 .
 \]
 An integration in time and Gronwall's inequality yield
 \[
     \|\bar Q\|_{\rL^\infty_\tau\rL^2(\Omega)}
     +
     \|\nabla\bar Q\|_{\rL^2_\tau\rL^2(\Omega)}
    \leq
     C
         \|\bar f\|_{\rL^2_\tau\rL^2(\Omega)}
         .
 \]

 Next, we test the difference equation with \(\dt\bar Q\). This gives
 \[
 \begin{aligned}
    \|\dt\bar Q\|_2^2
    +
    \frac12\dt\|\nabla\bar Q\|_2^2
     &=
     \int_\Omega \bar f\dt\bar Q\,\rd x
     -
     \int_\Omega
        \Big(
             \bigl(Q^{(1)}\bigr)^+
             -
             \bigl(Q^{(2)}\bigr)^+
         \Big)
        \dt\bar Q
     \,\rd x .
 \end{aligned}
 \]
 Using the Lipschitz continuity of the positive part in \(\rL^2(\Omega)\), we
 have
 \[
     \big \|
         \bigl(Q^{(1)}\bigr)^+
         -
         \bigl(Q^{(2)}\bigr)^+
     \big \|_2
     \leq
     \|\bar Q\|_2 .
 \]
 Therefore, by Hölder's and Young's inequalities,
 \[
 \begin{aligned}
     \|\dt\bar Q\|_2^2
    +
     \dt\|\nabla\bar Q\|_2^2
    \leq
    C\|\bar f\|_2^2
     +
     C\|\bar Q\|_2^2 .
 \end{aligned}
 \]
 Integrating in time and using the previous estimate, we obtain
 \[
     \|\dt\bar Q\|_{\rL^2_\tau\rL^2(\Omega)}
     +
     \|\nabla\bar Q\|_{\rL^\infty_\tau\rL^2(\Omega)}
     \leq
     C
         \|\bar f\|_{\rL^2_\tau\rL^2(\Omega)}.
 \]

 Similarly, testing the difference equation with \(-\Delta\bar Q\) gives
 \[
 \begin{aligned}
     \frac12\dt\|\nabla\bar Q\|_2^2
     +
    \|\Delta\bar Q\|_2^2
    &=
     -\int_\Omega \bar f\Delta\bar Q\,\rd x
     +
    \int_\Omega
         \Big(
             \bigl(Q^{(1)}\bigr)^+
            -
            \bigl(Q^{(2)}\bigr)^+
         \Big)
         \Delta\bar Q
     \,\rd x .
 \end{aligned}
 \]
 Again, using
 \[
     \big \|
         \bigl(Q^{(1)}\bigr)^+
         -
         \bigl(Q^{(2)}\bigr)^+
    \big \|_2
     \leq
    \|\bar Q\|_2 ,
 \]
we infer
 \[
     \dt\|\nabla\bar Q\|_2^2
     +
     \|\Delta\bar Q\|_2^2
     \leq
     C\|\bar f\|_2^2
    +
   C\|\bar Q\|_2^2 .
 \]
 After integration in time, the lower-order estimate gives
 \[
     \|\Delta\bar Q\|_{\rL^2_\tau\rL^2(\Omega)}
     \leq
     C
         \|\bar f\|_{\rL^2_\tau\rL^2(\Omega)} 
   .
 \]
 Combining the preceding bounds, we conclude that
 \[
     \|\bar Q\|_{\rH^1_\tau\rL^2(\Omega)}
     +
     \|\bar Q\|_{\rL^2_\tau\rH^2(\Omega)}
     +
    \|\bar Q\|_{\rL^\infty_\tau\rH^1(\Omega)}
     \leq
     C
         \|\bar f\|_{\rL^2_\tau\rL^2(\Omega)}
  .
 \]
 \end{proof}

\section{Proof of the Main Theorem}
\label{sec: proof}
\noindent
Finally, in this section, we assemble all the ingredients and prove the main
theorem by means of the contraction mapping principle.

\begin{proof}[Proof of \autoref{thm: global WP}]
Let \(\mathbb B_\eps(0,\tau)\) be the solution ball introduced in \autoref{lem: est nonlinear moist}. Given
\[
    U_1
    =
    (\bar\rho_{\rd,1}^\rL,v_1^\rL,T_1^\rL,
    q_{\rv,1}^\rL,q_{\rc,1}^\rL,q_{\rr,1}^\rL)
    \in \mathbb B_\eps(0,\tau),
\]
we first solve the moisture equations. The
non-quadratic switch term is kept on the left-hand side of the vapor equation.
Thus the relevant equation has the form
\[
    \dt q_{\rv}^\rL-\Delta q_{\rv}^\rL+(q_{\rv}^\rL)^+
    =
    h_{\rv}(U_1).
\]
The remaining moisture equations are solved with the corresponding right-hand
sides \(h_{\rc}(U_1)\) and \(h_{\rr}(U_1)\). By the nonlinear estimates obtained
above,
\[
    \|h_j(U_1)\|_{\mathbb E_0^h(0,\tau)}
    \leq
    C\eps^2,
    \qquad
    j\in\{\rv,\rc,\rr\}.
\]
We assume that the initial and compatibility data satisfy
\[
    \sum_{j\in\{\rv,\rc,\rr\}}
    \big(
        \|q_{j,0}\|_{\rH^2(\Omega)}
        +
        \|\dt q_j(0)\|_{\rH^1(\Omega)}
    \big)
    \leq
    C_0\eps^2 .
\]
Hence, by the auxiliary estimate for the switch equation \eqref{eq: auxillary for q cutoff}, we obtain
\[
    \sum_{j\in\{\rv,\rc,\rr\}}
    \|q_j^\rL\|_{\mathbb E_1^h(0,\tau)}
    \leq
    C\eps^2 .
\]
In particular,
\[
    \sum_{j\in\{\rv,\rc,\rr\}}
    \|\dt q_j^\rL\|_{\mathbb E_0^h(0,\tau)}
    \leq
    C\eps^2 .
\]
We now insert the resulting moisture variables into the temperature equation.
The only point which was not covered by the genuinely quadratic nonlinear
estimate is the linear moisture-time contribution $\mathcal Q_T$ as well as the additional term appearing in $G_v$. Recall
\[
    \mathcal Q_T[\dt Q]
    :=
    \frac{T^\ast}{\hat B^\ast}
    D\hat B^\ast
    \bigl[
        0,\dt(q_{\rv}^\rL,q_{\rc}^\rL,q_{\rr}^\rL)
    \bigr]
    +
    \frac{1}{\hat B^\ast}
    \int_0^z
    D\hat B^\ast
    \bigl[
        0,\dt(q_{\rv}^\rL,q_{\rc}^\rL,q_{\rr}^\rL)
    \bigr](\cdot,\eta)\,\rd\eta .
\]
Using the boundedness of \(D\hat B^\ast\) and of vertical integration, we get
\[
\begin{aligned}
    \|\mathcal Q_T[\dt Q]\|_{\mathbb E_0^h(0,\tau)}
    \leq
    C
    \sum_{j\in\{\rv,\rc,\rr\}}
    \|\dt q_j^\rL\|_{\mathbb E_0^h(0,\tau)}
 \leq C\eps^2 .
\end{aligned}
\]
Moreover, since
\[
    \|(q_{\rv}^\rL)^+\|_{\mathbb E_0^h(0,\tau)}
    \leq
    \|q_{\rv}^\rL\|_{\mathbb E_0^h(0,\tau)}
    \leq
    C\eps^2 ,
\]
the full temperature remainder satisfies
\[
\begin{aligned}
    \|G_T\|_{\mathbb E_0^h(0,\tau)}
    \leq
    \|\widetilde G_T\|_{\mathbb E_0^h(0,\tau)}
    +
    \|\mathcal Q_T[\dt Q]\|_{\mathbb E_0^h(0,\tau)}
    +
    \|(q_{\rv}^\rL)^+\|_{\mathbb E_0^h(0,\tau)}
 \leq C\eps^2 .
\end{aligned}
\]
Analogously one shows that 
\begin{equation*}
    \| G_v \|_{\E_0^v(0,\tau)} \leq C\eps^2.
\end{equation*}
We now solve the linearized \((\bar\rho_{\rd},v,T)\)-system with these
right-hand sides. We define
\[
    \Psi(U_1)
    :=
    (\bar\rho_{\rd}^\rL,v^\rL,T^\rL,
    q_{\rv}^\rL,q_{\rc}^\rL,q_{\rr}^\rL),
\]
where \((q_{\rv}^\rL,q_{\rc}^\rL,q_{\rr}^\rL)\) are obtained from the moisture
step above, and where \((\bar\rho_{\rd}^\rL,v^\rL,T^\rL)\) denotes the unique
strong solution of the linearized \((\bar\rho_{\rd},v,T)\)-problem with
right-hand sides \((G_\rho,G_v,G_T)\).
We next verify that \(\Psi\) maps \(\mathbb B_\eps(0,\tau)\) into itself.
By the moisture estimate derived above, we have
\[
    \sum_{j\in\{\rv,\rc,\rr\}}
    \|q_j^\rL\|_{\mathbb E_1^h(0,\tau)}
    \leq
    C\eps^2 .
\]
Moreover, the nonlinear estimates \autoref{lem: est nonlinear moist}
and the linear maximal-regularity estimate \autoref{lem: linear max reg} yield,
\[
\begin{aligned}
 &\quad   \|\bar\rho_{\rd}^\rL\|_{\mathbb E_1^\rho(0,\tau)}
    +
    \|v^\rL\|_{\mathbb E_1^v(0,\tau)}
    +
    \|T^\rL\|_{\mathbb E_1^h(0,\tau)}
\\ &\leq
    C
    \big (
        \|G_\rho\|_{\mathbb E_0^\rho(0,\tau)}
        +
        \|G_v\|_{\E_0^v(0,\tau)}
        +
        \|G_T\|_{\mathbb E_0^h(0,\tau)}
        +
        \|\bar\rho_{\rd,0} - \bar \rho^\ast_\rd\|_{\rH^2(\T^2)}
        +
        \|v_0\|_{\rH^2(\Omega)}
        +
        \|T_0 - T^\ast\|_{\rH^2(\Omega)}
    \big ) \\ &\leq
    C\eps^2 .
\end{aligned}
\]
Combining this with the moisture estimate yields
\[
    \|\Psi(U_1)\|_{\mathbb E_1(0,\tau)}
    \leq
    C\eps^2 .
\]
Thus, choosing \(\eps>0\) sufficiently small, $\Psi$
is a self-map. In the help of \autoref{cor: difference est q cutoff},
the contraction property is proved in the same way in the less regular space 
\begin{equation*}
    \rL^2(0,\tau; \rH^1(\T^2)) \times \rL^2(0,\tau; \rH^2(\Omega;\R^2)) \cap \rH^1(0,\tau;\rL^2(\Omega;\R^2)) \times \big (\rL^2(0,\tau; \rH^2(\Omega)) \cap \rH^1(0,\tau;\rL^2(\Omega)) \big )^4.
\end{equation*}
Finally, since \(\eps>0\) is chosen sufficiently small,
\autoref{lem:ests of trafo moist} implies that the transform \(\rX\) is invertible
with inverse \(\mathrm Y\). Defining the pulled-back functions by
\[
    \bar\rho_{\rd}(t,\cdot)
    :=
    \bar\rho_{\rd}^\rL(t,\mathrm Y(t,\cdot))+ \bar{\rho}^\ast_\rd,
\]
\[
    v(t,\cdot,z)
    :=
    v^\rL(t,\mathrm Y(t,\cdot),z),
    \quad
    T(t,\cdot,z)
    :=
    T^\rL(t,\mathrm Y(t,\cdot),z)+T^\ast,
\]
and
\[
    q_j(t,\cdot,z)
    :=
    q_j^\rL(t,\mathrm Y(t,\cdot),z)
    \ \text{ for }
    j\in\{\rv,\rc,\rr\},
\]
and introducing the diagnostic quantities \(w\), \(p\), and \(\rho\) in terms of
\((\bar\rho_{\rd},v,T,q_{\rv},q_{\rc},q_{\rr})\) as described in
\autoref{sec:coupled model}, we readily verify that
\[
    (\rho,u=(v,w),T,q_{\rv},q_{\rc},q_{\rr})
\]
is the unique strong solution of the full moist compressible primitive
equations. This proves the theorem.
\end{proof}
\noindent
{\bf Acknowledgements}
Tarek Z\"ochling acknowledges the support by DFG project FOR~5528.
\medskip

\noindent
 {\bf Statements and Declarations}

\smallskip

\noindent
    {\bf Data Availability Statement} Data sharing is not applicable to this article as no datasets were generated or analyzed during the current study.

 \smallskip
\noindent
    {\bf Competing Interests} The authors have no Conflict of Interest to declare that is relevant to the content of the article.

\end{document}